\documentclass[12pt,english]{oxarticle}
\usepackage{graphicx, float, array, xspace, amscd, amsmath, hyperref, amsfonts, amsthm, amssymb, latexsym, bbm, kbordermatrix, booktabs,mathrsfs,mathtools}
\usepackage[table,svgnames]{xcolor}
\usepackage[sort&compress, comma, square, numbers]{natbib}
\usepackage{tikz-cd}
\usetikzlibrary{cd}

\newtheorem{theorem}{Theorem}[section]
\newtheorem{lemma}[theorem]{Lemma}
\newtheorem{conjecture}[theorem]{Conjecture}
\newtheorem{remark}[theorem]{Remark}
\newtheorem{definition}[theorem]{Definition}
\newtheorem{proposition}[theorem]{Proposition}

\newtheorem{corollary}[theorem]{Corollary}

\DeclareFontFamily{OT1}{pzc}{}
\DeclareFontShape{OT1}{pzc}{m}{it}{<-> s * [1.200] pzcmi7t}{}
\DeclareMathAlphabet{\mathpzc}{OT1}{pzc}{m}{it}

\font\eightrm=cmr8 at 8pt

\DeclareFontFamily{U}{wncy}{}
\DeclareFontShape{U}{wncy}{m}{n}{<->wncyr10}{}
\DeclareSymbolFont{mcy}{U}{wncy}{m}{n}
\DeclareMathSymbol{\sha}{\mathord}{mcy}{"58}

\def\Q{\mathbb{Q}}
\def\DM{ \textbf{DM}_{\text{gm}}(\mathbb{Q},\mathbb{Q}) }
\def\DTM{ \textbf{DTM}_{\Q} }

\def\MHS{ \textbf{MHS}_{\Q}}
\def\MTM{ \textbf{TM}_{\Q}}
\def\Re{\mathfrak{R}}

\usepackage{authblk}
\title{Mirror symmetry, mixed motives, and $\zeta(3)$}
\author[1]{Minhyong Kim}
\author[2]{Wenzhe Yang}
\affil[1]{Mathematical Institute, University of Oxford}
\affil[2]{Stanford Institute for Theoretical Physics}

\numberwithin{equation}{section}
\proofmodefalse

\begin{document}
%
%
%
%
%
%
%
%
%
%
%
\maketitle
 		
 \begin{abstract}		 	
In this paper, we present an application of mirror symmetry to arithmetic geometry. The main result is the computation of the period of a mixed Hodge structure, which lends evidence to its expected motivic origin.
 More precisely, given a mirror pair $(M,W)$ of Calabi-Yau threefolds, the prepotential of the complexified K\"ahler moduli space of $M$ admits an expansion with a constant term that is frequently of the form $$-3\, \chi (M) \,\zeta(3)/(2 \pi i)^3+r,$$
where $r \in \mathbb{Q}$ and $\chi(M)$ is the Euler characteristic of $M$. We focus on the mirror pairs for which the deformation space of the mirror threefold $W$ forms part of a one-parameter algebraic family $W_{\varphi}$ defined over $\mathbb{Q}$ and the large complex structure limit is a rational point. Assuming a version of the mirror conjecture, we  compute the limit mixed Hodge structure on $H^3(W_{\varphi})$ at the large complex structure limit. It turns out to have a direct summand expressible as an extension of $\mathbb{Q}(-3)$ by $\mathbb{Q}(0)$ whose isomorphism class can be computed in terms of the prepotential of $M$, and hence, involves $\zeta(3)$. By way of  Ayoub's works on the motivic nearby cycle functor, this reveals in precise form a connection between mirror symmetry and a  variant of the Hodge conjecture for mixed Tate motives.
 		
\end{abstract}

\section{Introduction}
Over 50 years of research on motives has continued to enrich virtually all areas of number theory and algebraic geometry with a wide-ranging supply of unifying themes as well as deep formulas. Nevertheless, Grothendieck's original vision, supplemented by Beilinson and Deligne \cite{BeilinsonHeight, Deligne-Soule}, whereby an abelian category of the motives over $\Q$  is supposed to provide the building blocks for  arithmetic varieties, has yet to be realised. So far the best approximation to this idea is the construction of a triangulated category $\textbf{DM}_{\text{gm}}(\mathbb{Q},\mathbb{Q})$ of the mixed motives \cite{LevineMM,VoevodskyMM} that has nearly all the properties conjectured for the derived category of the desired abelian category.  That is,  $\textbf{DM}_{\text{gm}}(\mathbb{Q},\mathbb{Q})$ is expected to carry a natural motivic $t$-structure, whose heart should be Grothendieck's abelian category, however the construction of such a motivic $t$-structure appears inaccessible for the time being. As far as an abelian category is concerned, the best constructed is that of the mixed Tate motives $\MTM$, the one whose semi-simplification consists only of the simplest possible pure motives $\Q(n),\,n \in \mathbb{Z}$. That is to say, the full triangulated subcategory $\DTM$ of $\DM$ generated by $\Q(n),\,n\in \mathbb{Z} $ {\em does} have a motivic $t$-structure whose heart is by definition the abelian category $\MTM$. However in fact, even the way in which the mixed Tate motives sit inside $\textbf{DM}_{\text{gm}}(\mathbb{Q},\mathbb{Q})$ remains rather mysterious. 

Here is a natural conjecture in this regard, which could be regarded as a generalised Hodge conjecture concerning the Hodge realisation functor
$$\Re: \DM \rightarrow D^b(\MHS)$$
from $\textbf{DM}_{\text{gm}}(\mathbb{Q},\mathbb{Q})$ to the derived category of mixed Hodge structures.

\textbf{Conjecture GH} \textit{Suppose a mixed Hodge-Tate structure $\mathbf{P}$ occurs as a direct summand of $H^q(\Re(\mathcal{N}))$ for some $q \in \mathbb{Z}$ and $\mathcal N \in \DM$, then $$\mathbf{P}\simeq \Re(\mathcal{M})$$ where $\mathcal{M}$ is a mixed Tate motive.
}

Here, a mixed Hodge-Tate structure is a mixed Hodge structure whose semi-simplification is a direct sum of $\mathbb{Q}(n)$ for various $n \in \mathbb{Z}$. Of course this conjecture is likely to be inaccessible in general using current day technologies. Nevertheless we might be concerned   with the possibility of testing it numerically in some sense. In this paper, we will focus on the following period-theoretic version conjecture:

\textbf{Conjecture GHP}  \textit{Suppose a mixed Hodge-Tate structure $\mathbf{P}$ occurs as a direct summand of $H^q(\Re(\mathcal{N}))$ for some $q \in \mathbb{Z}$ and $\mathcal N \in \DM$.
Suppose further that $\mathbf{P}$ is an extension of $\mathbb{Q}(0)$ by $\mathbb{Q}(n), n\geq 3$, then its image in $$ Ext^1_{\MHS}(\Q(0), \Q(n)) \simeq \mathbb{C}/(2 \pi i)^n\Q $$ is the coset of a rational multiple of $\zeta(n)$.
}

As we will explain later, Conjecture \textbf{GHP} follows from Conjecture {\bf GH} and the known computations of the periods of mixed Tate motives \cite{DeligneFund}. The main purpose of this paper is to construct mixed motives in the theory of the mirror symmetry of the Calabi-Yau threefolds, and show they satisfy the predictions of Conjecture \textbf{GHP}, establishing thereby a highly interesting connection between the theory of mixed motives and mirror symmetry.

Roughly speaking, mirror symmetry is a conjecture that predicts the existence of mirror pairs $(M, W)$ of Calabi-Yau threefolds such that the complexified K\"ahler moduli space of $M$ is isomorphic to an open subset of the complex moduli space of $W$. Furthermore, this open subset is given as a neighbourhood of a special boundary point of the complex moduli space of $W$ known as the large complex structure limit \cite{MarkGross,CoxKatz}. The complexified K\"ahler moduli space of $M$, denoted by $\mathscr{M}_K(M)$, is essentially the space whose points represent the K\"ahler structures on $M$, while the complex moduli space of $W$, denoted by $\mathscr{M}_C(W)$, is the space whose points represent the complex structures of $W$ \cite{MarkGross,CoxKatz}. The isomorphism between $\mathscr{M}_K(M)$ and a neighborhood of the large complex structure limit in $\mathscr{M}_C(W)$ is called the mirror map, which is constructed by the identifications of certain functions on $\mathscr{M}_K(M)$ with those on $\mathscr{M}_C(W)$ \cite{MarkGross,CoxKatz,PhilipXenia1}. Mirror symmetry is a very powerful tool in the study of algebraic geometry, as it allows one to transfer the questions that are very difficult on one side to considerably easier questions on the other side. Let's now give a brief description of $\mathscr{M}_K(M)$ based on \cite{MarkGross}. Recall that the Hodge diamond of $M$ is of the form
\begin{center} 
\begin{tabular}{ c c c c c c c }
 &  &  & 1 &  &  &  \\ 
 &  & 0&   & 0&  &  \\   
 & 0&  & \,\,\,\,$h^{11}$ &  & 0&   \\  
1&  & $h^{21}$&  &$h^{21}$ & & 1 \\ 
 & 0&  & \,\,\,\,$h^{11}$ &  & 0&   \\ 
 &  & 0&   & 0&  &  \\   
 &  &  & 1 &  &  &  \\
\end{tabular}
\end{center}
where $h^{11}:=\text{dim}_{\mathbb{C}}\,H^{1,1}(M)$ and $h^{21}:=\text{dim}_{\mathbb{C}}\,H^{2,1}(M)$. Notice that both $H^1(M,\mathbb{R})$ and $H^5(M,\mathbb{R})$ are zero. Define $H^{1,1}(M,\mathbb{R})$ by
\begin{equation}
H^{1,1}(M,\mathbb{R}):=H^2(M,\mathbb{R}) \cap H^{1,1}(M,\mathbb{C}),
\end{equation}
i.e. it consists of the elements of $H^2(M,\mathbb{R})$ that can be represented by real closed forms of the type $(1,1)$. The K\"ahler cone of $M$ is defined by 
\begin{equation}
\mathcal{K}_{M}:= \{ \omega \in H^{1,1}(M,\mathbb{R})\,|\, \,\omega \,\text{ can be represented by a K\"ahler form of }\,M \},
\end{equation}
which is an open subset of $H^{1,1}(M,\mathbb{R})$ \cite{MarkGross}. The complexified K\"ahler moduli space of $M$ is defined by \cite{MarkGross}
\begin{equation}
\mathscr{M}_K(M):=(H^2(M,\mathbb{R})+i\, \mathcal{K}_{M})/H^2(M,\mathbb{Z}).
\end{equation}
In this paper we will only consider the one-parameter mirror pairs $(M,W)$ of Calabi-Yau threefolds, i.e.
\begin{equation}
\text{dim}_{\mathbb{C}}\,H^{1,1}(M)=\text{dim}_{\mathbb{C}}\,H^{2,1}(W)=1,
\end{equation}
in which case the K\"ahler cone $\mathcal{K}_{M}$ of $M$ is the open ray $ \mathbb{R}_{> 0}$. Hence $\mathscr{M}_K(M)$ has a very simple description \cite{MarkGross}
\begin{equation}
\mathscr{M}_K(M)=(\mathbb{R}+i\, \mathbb{R}_{> 0})/\mathbb{Z}=\mathbb{H}/\mathbb{Z},
\end{equation}
where $\mathbb{H}$ is the upper half plane of $\mathbb{C}$. Now let $e$ be a basis of $H^2(M,\mathbb{Z})$ (modulo torsion) that lies in the K\"ahler cone $\mathcal{K}_{M}$ \cite{MarkGross}, then every point of $\mathscr{M}_K(M)$ is represented by $t\, e, \,t \in \mathbb{H}$, while $e\,t$ is equivalent to $e\,(t+1)$ under the quotient by $\mathbb{Z}$. Conventionally $t$ is called the flat coordinate of $\mathscr{M}_K(M)$ by physicists \cite{MarkGross, PhilipXenia1}. In mirror symmetry, the prepotential $\mathcal{F}$ on the K\"ahler side admits an expansion near $i\, \infty$ that is of the form \cite{PhilipXenia,PhilipXenia1}
\begin{equation} \label{eq:Prepotential}
\mathcal{F}=-\frac{1}{6}\, Y_{111}\, t^3 -\frac{1}{2}\, Y_{011}\,t^2-\frac{1}{2}\,Y_{001}\, t-\frac{1}{6}\,Y_{000}+\mathcal{F}^{\text{np}},
\end{equation}
where $\mathcal{F}^{\text{np}}$ is the non-perturbative instanton correction. Moreover, $\mathcal{F}^{\text{np}}$ is invariant under the translation $t \rightarrow t+1$ and it is exponentially small when $t \rightarrow i\, \infty$, i.e. it admits a series expansion in $\exp 2\pi i\,t$ of the form \cite{PhilipXenia1}
\begin{equation}
\mathcal{F}^{\text{np}}=\sum_{n=1}^{\infty} a_n \exp 2 \pi i \,nt.
\end{equation}
The coefficient $Y_{111}$ in \ref{eq:Prepotential} is the topological intersection number given by \cite{PhilipXenia, PhilipXenia1}
\begin{equation}
Y_{111}=\int _{M} e\wedge e \wedge e,
\end{equation}
which is a positive integer. By the Lefschetz (1,1)-theorem, $e$ can be represented by a divisor of $M$ and $Y_{111}$ is the triple intersection number of this divisor with itself. The coefficients $Y_{011}$ and $Y_{001}$ will be shown to be rational numbers by mirror symmetry. The coefficient $Y_{000}$ is certainly the most mysterious one and in all examples of mirror pairs where it has been computed, it is always of the form \cite{PhilipXenia}
\begin{equation} \label{eq:PhysicistsY000}
Y_{000}=-3\, \chi (M)\, \frac{\zeta(3)}{(2 \pi i)^3} +r,~ r\in \mathbb{Q},
\end{equation}
where $\chi(M)$ is the Euler characteristic of $M$. It is speculated that equation \ref{eq:PhysicistsY000} is valid in general for an arbitrary mirror pair, but currently there is no proof available. The occurrence of $\zeta(3)$ in $Y_{000}$ is highly interesting, and one might ask if it has an arithmetic origin when combined with the geometry of the mirror threefold $W$.

In order to study this question, we will compute the limit mixed Hodge structure on $H^3(W_{\varphi},\mathbb{Q})$ at the large complex structure limit of $\varphi$ (the relation between $t$ and $\varphi$ is given by the {\em mirror map} defined in section 4). It is shown to split into a direct sum
\begin{equation}
\mathbb{Q}(-1) \oplus \mathbb{Q} (-2) \oplus \mathbf{M},
\end{equation}
where $\mathbf{M}$ is a two-dimensional mixed Hodge structure that forms an extension of $\mathbb{Q}(-3)$ by $\mathbb{Q}(0)$ (Theorem \ref{MainTheorem}). In fact, mirror symmetry plays an essential role in the computation of the limit mixed Hodge structure, and leads to the following result.

\begin{theorem} \label{eq:ThmPeriodofMdual}
Assuming the mirror symmetry conjecture(which will be stated in section 4.3), the dual of $\mathbf{M}$, denoted by $\mathbf{M}^{\vee}$, forms an extension of $\mathbb{Q}(0)$ by $\mathbb{Q}(3)$ whose class in 
$$\text{Ext}^1_{\textbf{MHS}_{\mathbb{Q}}}\left(\mathbb{Q}(0),\mathbb{Q}(3)\right) \simeq  \mathbb{C}/(2 \pi i)^3\,\mathbb{Q}$$
 is the coset of $-(2\pi i)^3 \, Y_{000}/(3\,Y_{111})$.
\end{theorem}

Thus, the mixed Hodge structure arising from the B-model on $W$ is naturally expressed in terms of $A$-model invariants on $M$.
In a number of examples of one-parameter mirror pairs defined over $\Q$ where $Y_{000}$ has been computed \cite{PhilipXenia, PhilipXenia1}, it is  of the form \ref{eq:PhysicistsY000}. Therefore for these mirror pairs the class of $\mathbf{M}^{\vee}$ in $\mathbb{C}/(2 \pi i)^3\,\mathbb{Q}$ is the coset of a rational multiple of $\zeta(3)$. As explained in section 1, this is exactly what would be expected if  $\mathbf{M}^{\vee}$ were the Hodge realisation of a mixed Tate motive over $\mathbb{Q}$. The occurrence of zeta values as limit period integrals has been observed before in a number of contexts. However, the significance of this result is the splitting off of a precise two-step extension inside the limit Hodge structure. This relies on an expression for the monodromy filtration on the limit Hodge structure that makes use of mirror symmetry.

In fact, when the deformation space of $W$ forms part of a one-parameter algebraic family defined over $\mathbb{Q}$ and the large complex structure limit is a rational point (which is the case in the aforementioned examples),   Ayoub's work on the motivic nearby cycle functor produces a limit mixed motive at the large complex structure limit \cite{AyoubT1}. On the other hand, from the works of Steenbrink \cite{SteenbrinkLMHS} and Schmid \cite{Schmid} on the limit mixed Hodge structures, there exists a complex in $D^b(\MHS)$ whose cohomologies compute the limit mixed Hodge structures on $H^n(W,\mathbb{Q}),n \in \mathbb{Z}$. The limit mixed motive is supposed to be the motivic lift of this complex in $D^b(\MHS)$, i.e. the Hodge realisation of the limit mixed motive should be this complex.  Unfortunately the proof of this statement seems to be quite technical and not available yet in literature. Therefore in this paper it is stated as a separate conjecture (Conjecture \ref{limitmotiveconjecture}), due to Ayoub. The main conclusion of this paper is that assuming mirror symmetry and Conjecture \ref{limitmotiveconjecture}, the existing computations of $Y_{000}$ (\ref{eq:PhysicistsY000}) provide highly interesting affirmative examples  of Conjecture \textbf{GHP}. On the other hand, Conjecture \textbf{GHP} also sheds light on the motivic nature of the $\zeta(3)$ in $Y_{000}$, that is, from the converse point of view, if we assume \textbf{GHP} and \ref{limitmotiveconjecture} instead, the computations in this paper show that $Y_{000}$ must be of the form
\begin{equation}
Y_{000}=\frac{r_1}{(2\pi i)^3} \,\zeta(3)+\,r_2,~r_1,r_2 \in \mathbb{Q}.
\end{equation}
Thus what we have observed is an approximate equivalence between Conjecture \textbf{GHP} for  $\mathbf{M}^{\vee}$ arising from this family  and equation \ref{eq:PhysicistsY000} in mirror symmetry. 

The structure of this paper is as follows. Section 1 is a very brief introduction to Voevodsky's category of mixed motives and mixed Tate motives. Section 2 is a short overview of the Gauss-Manin connection and the construction of the limit mixed Hodge structures by Schmid and Steenbrink . Section 3 discusses the construction of the limit mixed motive by Ayoub's motivic nearby cycle functor. Section 4 is concerned with the computation of the limit mixed Hodge structure at the large complex structure limit. Section 5 studies the motivic nature of $\zeta(3)$ and its implications. The substance of the paper is in sections 4 and 5 and the reader familiar with the background  in algebraic and arithmetic geometry should have no trouble jumping directly to that portion following a brief look over the notation.  On the other hand, the first three sections and the appendices are included in order to give a brief exposition of the prerequisites, mostly for the benefit of a readership approaching the paper with a physics background.

\section{Mixed motives}

This section is a very brief introduction to Voevodsky's triangulated category of mixed motives and the abelian category of mixed Tate motives over $\mathbb{Q}$. Of course it is not meant to be complete, hence necessary references will be  provided for further reading. At the same time, Appendix \ref{MHSandExtensions} contains an introduction to mixed Hodge structures, while Appendix \ref{PureMotive} is an overview of pure motives, both of which are prerequisites for this section. Therefore  readers who are not familiar with this background are urged to consult the appendices and the references given there. 

\subsection{Voevodsky's mixed motives} \label{VoevodskyMM}

Let $k$ be a field that admits resolution of singularities and let $\Lambda$ be a commutative ring with unit, Voevodsky's category of mixed motives, denoted by $\textbf{DM}(k, \Lambda)$, is a rigid tensor triangulated category \cite{VoevodskyMM,VoevodskyCycles}. The ring $\Lambda$ is called the coefficient ring, while in this paper we are mostly concerned with the case where $\Lambda$ is $\mathbb{Q}$. We will list some properties of $\textbf{DM}(k, \Lambda)$ that we will use, a detailed discussion of which can be found in the references \cite{VoevodskyMM, VoevodskyCycles, AyoubFirst}.

\begin{enumerate}
	\item The category $\textbf{DM}(k, \Lambda)$ is a rigid tensor triangulated category that contains pure Tate motive $\Lambda(n),\,n\in \mathbb{Z}$. Moreover $\Lambda(0)$ is a unit object and $\Lambda(-1)$ is the dual of $\Lambda(1)$                                                                                                                                                                                                                                       
	\begin{equation}
	\Lambda(-1)= \text{Hom}(\Lambda(1),\Lambda(0)),
	\end{equation}
	where $ \text{Hom}$ is the internal Hom operator. The Tate motive $\Lambda(n)$ satisfies
	\begin{equation}
	\Lambda(n)=
	\begin{cases*}
	\Lambda(1)^{\otimes n} & if $n \geq 0$ \\
	\Lambda(-1)^{\otimes n} & if $n < 0$ \\
	\end{cases*}.
	\end{equation}
	For an object $\mathcal{N}$ of $\textbf{DM}(k, \Lambda)$, its Tate twist $\mathcal{N}(n)$ is defined as $\mathcal{N} \otimes \Lambda(n)$, while its dual $\mathcal{N}^{\vee}$ is defined as $\text{Hom}(\mathcal{N},\Lambda(0))$.

	\item There exists a contravariant functor from the category of non-singular projective varieties over $k$ to the category $\textbf{DM}(k, \Lambda)$
	\begin{equation}
	M_{gm}: \textbf{SmProj}/k^{\text{op}} \rightarrow \textbf{DM}(k, \Lambda),
	\end{equation} 
	which sends a non-singular projective variety $X$ to a constructible object of $\textbf{DM}(k, \Lambda)$ such that the product in $\textbf{SmProj}/k$ is sent to the tensor product in $\textbf{DM}(k, \Lambda)$
	\begin{equation}
	M_{gm}(X \times_k Y)=M_{gm}(X) \otimes M_{gm}(Y).
	\end{equation} 
	The definition of constructibility can be found in \cite{AyoubFirst}, and the full triangulated subcategory of $\textbf{DM}(k, \Lambda)$ consisting of constructible objects will be denoted by $\textbf{DM}_{\text{gm}}(k, \Lambda)$. It is the smallest full pseudoabelian triangulated subcategory of $\textbf{DM}(k, \Lambda)$ that contains the image of $M_{gm}$ and is also closed under Tate twists.

	\item When the field $k$ admits an embedding into $\mathbb{C}$, say $\sigma:k \rightarrow \mathbb{C}$, there exists a Hodge realisation functor
	\begin{equation}
	\mathfrak{R}_{\sigma}: \textbf{DM}_{\text{gm}}(k, \mathbb{Q}) \rightarrow D^b(\textbf{MHS}_{\mathbb{Q}})
	\end{equation}
such that $\mathfrak{R}_{\sigma}(M_{gm}(X))$ is a complex in $D^b(\textbf{MHS}_{\mathbb{Q}})$ whose cohomology computes the singular cohomology of $X(\mathbb{C})$ with the natural rational MHS \cite{HuberR1,HuberR2}. In this paper we are mostly concerned with the case where $k=\mathbb{Q}$, since there is only one embedding of $\mathbb{Q}$ into $\mathbb{C}$, let us denote the Hodge realisation functor by $\mathfrak{R}$ for simplicity.
	
	\item The composition of $\mathfrak{R}_{\sigma}$ with the forgetful functor from $D^b(\textbf{MHS}_{\mathbb{Q}})$ to the derived category of rational vector spaces $D^b(\textbf{Vec}_{\mathbb{Q}})$ is (up to an equivalence) the Betti realisation functor $\mathfrak{R}_{\text{Betti}}$ \cite{AyoubBetti}
	\begin{equation}
	\mathfrak{R}_{\text{Betti}}:\textbf{DM}_{\text{gm}}(k, \mathbb{Q})       \rightarrow D^b(\textbf{Vec}_{\mathbb{Q}}).
	\end{equation}   	
	
\end{enumerate}

\subsection{Mixed Tate motives}

We now briefly discuss the abelian category of mixed Tate motives $\textbf{TM}_{\mathbb{Q}}$, while the readers are referred to the paper \cite{LevineVC} for more details. Let $K_i(k)$ be the $i$-th algebraic $K$-group of $k$, then there exists a family of Adams operators $\{ \psi^l \}_{\,l \geq 1}$ that act on $K_i(k)$ as group homomorphisms \cite{Weibel}. These Adams operators induce linear maps on the rational vector space $K_i(k) \otimes_{\mathbb{Z}} \mathbb{Q}$, which induce a decomposition
\begin{equation} 
K_i(k) \otimes_{\mathbb{Z}} \mathbb{Q}= \oplus_{j \geq 0} \,K_i(k)^{(j)},
\end{equation}    
where the eigenspace $ K_i(k)^{(j)}$ is defined as
\begin{equation}
K_i(k)^{(j)}:= \{ x \in K_i(k) \otimes_{\mathbb{Z}} \mathbb{Q}:\psi^l(x)=l^j \,x, \,\forall \, l \geq 1 \}.
\end{equation}

The stronger version of Beilinson and Soul\'e's vanishing conjecture is stated as follows \cite{LevineVC}.

\textbf{Conjecture BS} \, \textit{$K_{2\,q-p}(k)^{(q)}=0$ if $p \leq 0$ and $q>0$}.
	 
Conjecture \textbf{BS} has been proved when $k$ is $\mathbb{Q}$ \cite{DeligneTate}. Let $\textbf{DTM}_{\mathbb{Q}}$ be the full triangulated subcategory of $\textbf{DM}_{\text{gm}}(\mathbb{Q},\mathbb{Q})$ generated by Tate objects $\mathbb{Q}(n),n \in \mathbb{Z}$, then from the paper \cite{LevineVC} there exists a motivic $t$-structure on $\textbf{DTM}_{\mathbb{Q}}$ whose heart is defined to be the abelian category of mixed Tate motives $\textbf{TM}_{\mathbb{Q}}$. Given two objects $A$ and $B$ of $\textbf{TM}_{\mathbb{Q}}$, an extension of $B$ by $A$ is a short exact sequence
\begin{equation} \label{eq:TMExtension}
\begin{tikzcd}
0 \arrow[r] &  A \arrow[r] & E \arrow[r] & B \arrow[r] & 0.
\end{tikzcd}
\end{equation}
Two extensions of $B$ by $A$ are said to be isomorphic if there exists a commutative diagram of the form
\begin{equation}
\begin{tikzcd}
0 \arrow[r]  &  A \arrow[r] \arrow[d,"\text{Id}"] & E \arrow[r] \arrow[d,"\simeq"] & B \arrow[r] \arrow[d,"\text{Id}"] & 0 \\
0 \arrow[r] & A \arrow[r] &  E' \arrow[r] & B \arrow[r] & 0
\end{tikzcd}.
\end{equation}
The extension \ref{eq:TMExtension} is said to split if it is isomorphic to the trivial extension
\begin{equation} \label{eq:TMTrivialExtension}
\begin{tikzcd}
0 \arrow[r] & A \arrow[r,"i"] & A \oplus B \arrow[r,"j"] & B \arrow[r]  &  0,
\end{tikzcd}
\end{equation}
where $i$ is the natural inclusion and $j$ is the natural projection. The set of the isomorphism classes of extensions of $B$ by $A$, denoted by $\text{Ext}^1_{\textbf{TM}_{\mathbb{Q}}}(B,A)$, has a group structure induced by Baer summation with zero element given by the trivial extension \ref{eq:TMTrivialExtension}. What is very important is that the extensions of $\mathbb{Q}(0)$ by $\mathbb{Q}(n),n \geq 3$ have an explicit description from
Corollary 4.3 of \cite{LevineVC}, which is stated as Lemma \ref{TMextensionslemma}.

\begin{lemma} \label{TMextensionslemma}
There exists an isomorphism $\tau_{n,1}$ of the form
\begin{equation}
\tau_{n,1}: \text{Ext}^1_{\textbf{TM}_{\mathbb{Q}}}\left(\mathbb{Q}(0),\mathbb{Q}(n)\right) \rightarrow K_{2n-1}(\mathbb{Q})^{(n)},\,n \geq 2.
\end{equation}
\end{lemma}

The rank of the algebraic $K$-group $K_{2n-1}(\mathbb{Q})$ is well-known \cite{Grayson}
\begin{equation}
\text{rank}\,K_{2n-1}(\mathbb{Q})=
\begin{cases*}
0          & if $n=2k, ~k\geq 1$  \\
1          & if  $n=2k+1,~k \geq 1$ 
\end{cases*},
\end{equation}
which shows
\begin{equation}
K_{2n-1}(\mathbb{Q}) \otimes_{\mathbb{Z}} \mathbb{Q}=
\begin{cases*}
0          & if $n=2\,k, \,k\geq 1$  \\
\mathbb{Q}          & if  $n=2\,k+1,\,k \geq 1$ 
\end{cases*}.
\end{equation}
Since $ K_{2n-1}(\mathbb{Q})^{(n)}$ is a linear subspace of $K_{2n-1}(\mathbb{Q}) \otimes_{\mathbb{Z}} \mathbb{Q}$, this immediately implies
\begin{equation}
 K_{2n-1}(\mathbb{Q})^{(n)}=0,\,\text{if}\,\,n=2\,k,\,k\geq 1.
\end{equation}
Similarly, when $n=2\,k+1,\,k \geq 1$, $K_{2n-1}(\mathbb{Q})^{(n)}$ is either 0 or $ \mathbb{Q}$. But from \cite{DeligneFund, DeligneTate, HainMatsumoto}, there exists a nontrivial extension of $\mathbb{Q}(0)$ by $\mathbb{Q}(n)$ when $n=2\,k+1,\,k \geq 1$, hence we deduce that
\begin{equation}
\text{Ext}^1_{\textbf{TM}_{\mathbb{Q}}}\left(\mathbb{Q}(0),\mathbb{Q}(n)\right) \simeq K_{2n-1}(\mathbb{Q})^{(n)}=\mathbb{Q}, \,\,\text{if}\,\,n=2\,k+1,\,k \geq 1.
\end{equation}

The restriction of the Hodge realisation functor $\mathfrak{R}$ to $\textbf{TM}_{\mathbb{Q}}$ is a functor whose image essentially lies in the full abelian subcategory $\textbf{MHT}_{\mathbb{Q}}$ of mixed Hodge-Tate structures
\begin{equation} \label{eq:HodgeRealTate}
\mathfrak{R}:\textbf{TM}_{\mathbb{Q}} \rightarrow \textbf{MHT}_{\mathbb{Q}},
\end{equation}
where $\textbf{MHT}_{\mathbb{Q}}$ consists of mixed Hodge structures whose semi-simplifications are direct sums of the Tate objects $\mathbb{Q}(n)$. From \cite{DeligneTate}, \ref{eq:HodgeRealTate} is exact and fully faithful, hence it induces an injective homomorphism from $\text{Ext}^1_{\textbf{TM}_{\mathbb{Q}}}\left(\mathbb{Q}(0),\mathbb{Q}(n)\right)$ to $\text{Ext}^1_{\textbf{MHT}_{\mathbb{Q}}}\left(\mathbb{Q}(0),\mathbb{Q}(n)\right)$. Moreover, the isomorphism \ref{eq:Isoq0q3} implies
\begin{equation} \label{eq:THexe3}
\text{Ext}^1_{\textbf{MHT}_{\mathbb{Q}}}\left(\mathbb{Q}(0),\mathbb{Q}(n)\right) = \text{Ext}^1_{\textbf{MHS}_{\mathbb{Q}}}\left(\mathbb{Q}(0),\mathbb{Q}(n)\right) \simeq  \mathbb{C}/(2 \pi i)^n\,\mathbb{Q}.
\end{equation}
For later convenience, we now summarise the results of this section in a lemma (which is of course well-known).

\begin{lemma} \label{zeta3lemma}
When $n \geq 3$, the image of $\text{Ext}^1_{\textbf{TM}_{\mathbb{Q}}}\left(\mathbb{Q}(0),\mathbb{Q}(n)\right)$ in $ \mathbb{C}/(2  \pi  i)^n\,\mathbb{Q}$ under the Hodge realisation functor is the subgroup of $ \mathbb{C}/(2  \pi  i)^n\,\mathbb{Q}$ consisting of elements that are the cosets of rational multiples of $\zeta(n)$.
\end{lemma}

\begin{proof}

When $n=2\,k,\,k \geq 2$, $\zeta(n)$ is a rational multiple of $(2\pi i)^n$, therefore the coset of $\zeta(n)$ in $ \mathbb{C}/(2  \pi  i)^n\,\mathbb{Q}$ is 0, hence this lemma follows immediately from Lemma \ref{TMextensionslemma}.

When $n=2\,k+1,\,k \geq 1$, from \cite{DeligneFund, DeligneTate, HainMatsumoto}, there exists a mixed Tate motive that forms a nontrivial extension of $\mathbb{Q}(0)$ by $\mathbb{Q}(n)$, furthermore its Hodge realisation in $\text{Ext}^1_{\textbf{MHT}_{\mathbb{Q}}}\left(\mathbb{Q}(0),\mathbb{Q}(n) \right)$ is the coset of a nonzero rational multiple of $\zeta(n)$. Then this lemma is an immediate result of Lemma \ref{TMextensionslemma}.
\end{proof}

\begin{remark}
From the results in this section, Conjecture \textbf{GHP} follows immediately from Conjecture \textbf{GH}.
\end{remark}

\section{Limit mixed Hodge structure} \label{LimitMHS}

In this section, we will discuss Steenbrink's and Schmid's constructions of limit mixed Hodge structures \cite{SteenbrinkLMHS,Schmid}. Assume given a flat proper map defined over $\mathbb{Q}$
\begin{equation} \label{eq:FibrationSRational}
\pi_{\mathbb{Q}}: X \rightarrow S,
\end{equation}
where $X$ is a quasi-projective variety over $\mathbb{Q}$ and $S$ is a smooth curve over $\mathbb{Q}$. We further assume that the only singular fiber is over a rational point $0\in S$
\begin{equation}
Y:=\pi_{\mathbb{Q}}^{-1}(0).
\end{equation}
We will assume $Y$ is reduced with nonsingular components crossing normally. For simplicity we now define
\begin{equation}
X^*:=X - Y,~~S^*:=S - \{0\}.
\end{equation}
By abuse of notations, the restriction of $\pi_{\mathbb{Q}}$ to $X^*$ will also be denoted by $\pi_{\mathbb{Q}}$
\begin{equation} \label{eq:Piq}
\pi_{\mathbb{Q}}:X^* \rightarrow S^*,
\end{equation}
which is a smooth fibration between smooth varieties. After field extension from $\mathbb{Q}$ to $\mathbb{C}$, we obtain a smooth fibration $\pi_{\mathbb{C}}$ between smooth varieties over $\mathbb{C}$
\begin{equation}
\pi_{\mathbb{C}}:X^*_{\mathbb{C}} \rightarrow S^*_{\mathbb{C}},
\end{equation}
where $X^*_{\mathbb{C}}:=X^* \times_{\text{Spec}\, \mathbb{Q}} \text{Spec} \,\mathbb{C}$, etc. While the analytification of $\pi_{\mathbb{C}}$, denoted by $\pi_{\mathbb{C}}^{\text{an}}$, is a smooth fibration between complex manifolds \cite{Hartshorne}
\begin{equation}
\pi_{\mathbb{C}}^{\text{an}}:X_{\mathbb{C}}^{*,\text{an}} \rightarrow S_{\mathbb{C}}^{*,\text{an}}.
\end{equation}

Since 0 is a smooth point of $S$, the local ring $\mathscr{O}_{S,0}$ is a discrete valuation ring (DVR) \cite{Hartshorne}. Let $\varphi$ be a uniformizer of $\mathscr{O}_{S,0}$, then there exists an open affine neighborhood $U$ of $0$ such that $\varphi$ defines a regular function on $U_{\mathbb{C}}^{\text{an}}$. Replace $S$ by $U$ if necessary we will assume $S$ is affine and $\varphi$ defines a regular function on $S_{\mathbb{C}}^{\text{an}}$. For later convenience, choose a small neighborhood $\Delta$ of 0 in $S_{\mathbb{C}}^{\text{an}}$ of the form
\begin{equation}
\Delta=\{\varphi \in \mathbb{C}:|\varphi|< \epsilon , ~0 < \epsilon \leq 1 \}.
\end{equation}

\begin{remark} \label{FixedParameter}
In this paper, we fix the choices of a uniformizer $\varphi$, a universal cover $\widetilde{\Delta}^*$ of $\Delta^*$ and a  holomorphic logarithm $\log \, \varphi$ of $\varphi$ on $\widetilde{\Delta}^*$ that we refer to loosely as a `multi-valued' holomorphic function on $\Delta^*$.
\end{remark}
The restriction of $\pi_{\mathbb{C}}^{\text{an}}$ to $\mathcal{X}:=(\pi_{\mathbb{C}}^{\text{an}})^{-1}(\Delta)$ will be denoted by
\begin{equation} \label{eq:pidelta}
\pi_{\Delta}:\mathcal{X} \rightarrow \Delta,
\end{equation}
where the only singular fiber $\mathcal{X}_0:=\pi_{\Delta}^{-1}(0)$ is the analytification of $Y_{\mathbb{C}}$. While the restriction of $\pi_{\Delta}$ to $\mathcal{X}^*:=\mathcal{X} - \mathcal{X}_0$ will be denoted by $\pi_{\Delta^*}$
\begin{equation}
\pi_{\Delta^*}: \mathcal{X}^* \rightarrow \Delta^*.
\end{equation}

\subsection{The Gauss-Manin connection}

The relative de Rham cohomology sheaf of the fibration \ref{eq:Piq} is defined by \cite{AltmanKleiman}
\begin{equation}
\mathscr{V}_{\mathbb{Q}}:=\mathbb{R}^q \,\pi_{\mathbb{Q},*}(\Omega_{X^*/S^*}^*),
\end{equation}
where $\Omega_{X^*/S^*}^*$ is the complex of sheaves of relative differentials \cite{AltmanKleiman,KatzOda}
\begin{equation} \label{eq:rationalcomplex}
\Omega_{X^*/S^*}^*: 0 \rightarrow \mathcal{O}_{X^*/S^*} \xrightarrow{d} \Omega_{X^*/S^*}^1 \xrightarrow{d} \cdots \xrightarrow{d}  \Omega_{X^*/S^*}^n \rightarrow 0.
\end{equation}
Where $n$ is the  dimension of the fibers of \ref{eq:Piq}. Since $\pi_{\mathbb{Q}}$ is a smooth fibration between smooth varieties,  $\mathscr{V}_{\mathbb{Q}}$ is a locally free sheaf over $S^*$ \cite{Grothendieck1, KatzOda}. The fiber of  $\mathscr{V}_{\mathbb{Q}}$ over a closed point $\varphi \in S^*$ is the $q$-th algebraic de Rham cohomology $\mathbb{H}^q(X_{\varphi},\Omega_{X_{\varphi}}^*)$, where $X_{\varphi}$ is a variety defined over the residue field $\kappa(\varphi)$ of $\varphi$ \cite{KatzOda, SteenbrinkLMHS}. The complex $\Omega_{X^*/S^*}^*$ is filtered by the following complexes
\begin{equation} \label{eq:rationalfiltercomplex}
F^p\, \Omega_{X^*/S^*}^*: 0 \rightarrow \cdots \rightarrow 0 \rightarrow \Omega_{X^*/S^*}^p \xrightarrow{d}  \Omega_{X^*/S^*}^{p+1} \xrightarrow{d} \cdots \xrightarrow{d}  \Omega_{X^*/S^*}^n \rightarrow 0,
\end{equation}
which induces a subsheaf filtration of $\mathscr{V}_{\mathbb{Q}}$
\begin{equation}
\mathscr{F}^p_{\mathbb{Q}}:=\text{Im}\,\big(\mathbb{R}^q \,\pi_{\mathbb{Q},*}(F^p\,\Omega_{X^*/S^*}^*) \rightarrow \mathbb{R}^q \,\pi_{\mathbb{Q},*}(\Omega_{X^*/S^*}^*)\big).
\end{equation}
Notice that $\mathscr{F}^p_{\mathbb{Q}}$ is also locally free. The complexification of the complex \ref{eq:rationalcomplex} is
\begin{equation} \label{eq:complexcomplex}
\Omega_{X_{\mathbb{C}}^*/S_{\mathbb{C}}^*}^*: 0 \rightarrow \mathcal{O}_{X_{\mathbb{C}}^*/S_{\mathbb{C}}^*} \xrightarrow{d} \Omega_{X_{\mathbb{C}}^*/S_{\mathbb{C}}^*}^1 \xrightarrow{d} \cdots \xrightarrow{d}  \Omega_{X_{\mathbb{C}}^*/S_{\mathbb{C}}^*}^n \rightarrow 0,
\end{equation}
which is filtered by the complexification of the complex \ref{eq:rationalfiltercomplex}
\begin{equation} \label{eq:complexfilteredcomplex}
F^p\, \Omega_{X_{\mathbb{C}}^*/S_{\mathbb{C}}^*}^*: 0 \rightarrow \cdots \rightarrow 0 \rightarrow \Omega_{X_{\mathbb{C}}^*/S_{\mathbb{C}}^*}^p \xrightarrow{d}  \Omega_{X_{\mathbb{C}}^*/S_{\mathbb{C}}^*}^{p+1} \xrightarrow{d} \cdots \xrightarrow{d}  \Omega_{X_{\mathbb{C}}^*/S_{\mathbb{C}}^*}^n \rightarrow 0.
\end{equation}
The sheaf $\mathscr{V}_{\mathbb{C}}$ and its subsheaf $\mathscr{F}^p_{\mathbb{C}}$ are defined by
\begin{equation}
\mathscr{V}_{\mathbb{C}}:=\mathbb{R}^q \, \pi_{\mathbb{C},*}(\Omega_{X^*_{\mathbb{C}}/S^*_{\mathbb{C}}}^*),~\mathscr{F}^p_{\mathbb{C}}:=\mathbb{R}^q \, \pi_{\mathbb{C},*}(F^p\,\Omega_{X^*_{\mathbb{C}}/S^*_{\mathbb{C}}}^*),
\end{equation}
which are the complexifications of $\mathscr{V}_{\mathbb{Q}}$ and $\mathscr{F}^p_{\mathbb{Q}}$ respectively. The analytification of the complex \ref{eq:complexcomplex} is
\begin{equation} \label{eq:complexanalyticcomplex}
\Omega_{X_{\mathbb{C}}^{*,\text{an}}/S_{\mathbb{C}}^{*,\text{an}}}^*: 0 \rightarrow \mathcal{O}_{X_{\mathbb{C}}^{*,\text{an}}/S_{\mathbb{C}}^{*,\text{an}}} \xrightarrow{d} \Omega_{X_{\mathbb{C}}^{*,\text{an}}/S_{\mathbb{C}}^{*,\text{an}}}^1 \xrightarrow{d} \cdots \xrightarrow{d}  \Omega_{X_{\mathbb{C}}^{*,\text{an}}/S_{\mathbb{C}}^{*,\text{an}}}^n \rightarrow 0,
\end{equation}
which has a filtration given by the analytification of the complex \ref{eq:complexfilteredcomplex}
\begin{equation} \label{eq:complexanalyticfilteredcomplex}
F^p\, \Omega_{X_{\mathbb{C}}^{*,\text{an}}/S_{\mathbb{C}}^{*,\text{an}}}^*: 0 \rightarrow \cdots \rightarrow 0 \rightarrow \Omega_{X_{\mathbb{C}}^{*,\text{an}}/S_{\mathbb{C}}^{*,\text{an}}}^p \xrightarrow{d}  \Omega_{X_{\mathbb{C}}^{*,\text{an}}/S_{\mathbb{C}}^{*,\text{an}}}^{p+1} \xrightarrow{d} \cdots \xrightarrow{d}  \Omega_{X_{\mathbb{C}}^{*,\text{an}}/S_{\mathbb{C}}^{*,\text{an}}}^n \rightarrow 0.
\end{equation}
The sheaf $\mathscr{V}^{\text{an}}_{\mathbb{C}}$ and its subsheaf $\mathscr{F}^{p,\text{an}}_{\mathbb{C}}$ are defined by
\begin{equation}
\mathscr{V}^{\text{an}}_{\mathbb{C}} :=\mathbb{R}^q\, \pi_{\mathbb{C},*}^{\text{an}}\,(\Omega_{X_{\mathbb{C}}^{*,\text{an}} / S_{\mathbb{C}}^{*,\text{an}}}^*),~\mathscr{F}^{p,\text{an}}_{\mathbb{C}} :=\mathbb{R}^q\, \pi_{\mathbb{C},*}^{\text{an}}\,(F^p\,\Omega_{X_{\mathbb{C}}^{*,\text{an}} / S_{\mathbb{C}}^{*,\text{an}}}^*),
\end{equation}
which are the analytifications of $\mathscr{V}_{\mathbb{C}}$ and $\mathscr{F}^p_{\mathbb{C}}$ respectively. The Gauss-Manin connection $\nabla_{\mathbb{Q}}$ of $\mathscr{V}_{\mathbb{Q}}$ is an integrable algebraic connection \cite{KatzOda,Zein} 
\begin{equation}
\nabla_{\mathbb{Q}}:\mathscr{V}_{\mathbb{Q}} \rightarrow \Omega_{S^*/\mathbb{Q}}^1 \otimes_{\mathcal{O}_{S^*}} \mathscr{V}_{\mathbb{Q}}
\end{equation} 
that satisfies Griffiths transversality
\begin{equation} \label{eq:GriffithsTrans}
\nabla_{\mathbb{Q}}  (\mathscr{F}^p_{\mathbb{Q}} ) \subset \Omega_{S^*/{\mathbb{Q}}}^1 \otimes_{\mathcal{O}_{S^*}} \mathscr{F}^{p-1}_{\mathbb{Q}}.
\end{equation}
The complexification of $\nabla_{\mathbb{Q}}$, denoted by $\nabla_{\mathbb{C}}$, is the Gauss-Manin connection of $\mathscr{V}_{\mathbb{C}}$ and the analytification of $\nabla_{\mathbb{Q}}$, denoted by $\nabla_{\mathbb{C}}^{\text{an}}$, is the Gauss-Manin connection of $\mathscr{V}^{\text{an}}_{\mathbb{C}}$. On the other hand, over the punctured disc $\Delta^*$ there is a local system \cite{Zein}
\begin{equation} \label{eq:IntegralLocalSystem}
V_{\mathbb{Z}}:=\text{R}^q \,\pi_{\Delta^*,*} \mathbb{Z},
\end{equation}
whose fiber over a point $\varphi \in \Delta^*$ is the singular cohomology group $H^q(\mathcal{X}_{\varphi},\mathbb{Z})$ \cite{Schmid,Zein}. The dual of $V_{\mathbb{Z}}$, denoted by $V^{\vee}_{\mathbb{Z}}$, is a local system over $\Delta^*$ whose fiber over $\varphi \in \Delta^*$ is the singular homology group $H_q(\mathcal{X}_{\varphi},\mathbb{Z})$ (modulo torsion). Similarly let $V_{\mathbb{Q}}$ (resp. $V_{\mathbb{C}}$) be the local system whose fiber at $\varphi$ is $H^q(\mathcal{X}_{\varphi},\mathbb{Q})$ (resp. $H^q(\mathcal{X}_{\varphi},\mathbb{C})$)
\begin{equation}
V_{\mathbb{Q}}:=\text{R}^q \,\pi_{\Delta^*,*} \mathbb{Q},~V_{\mathbb{C}}:=\text{R}^q \,\pi_{\Delta^*,*} \mathbb{C},
\end{equation}
and they are also given by \cite{Zein}
\begin{equation}
V_{\mathbb{Q}}=V_{\mathbb{Z}} \otimes \mathbb{Q},~V_{\mathbb{C}}=V_{\mathbb{Z}} \otimes \mathbb{C}.
\end{equation}

\begin{remark}
In the following, torsion of the singular homology and cohomology groups are irrelevant, and we will abuse notation somewhat and denote by $H^q(\mathcal{X}_{\varphi},\mathbb{Z})$, $\text{R}^q \,\pi_{\Delta^*,*} \mathbb{Z}$, etc., the objects modulo torsion.
\end{remark}
The local system $V_{\mathbb{Z}}$ defines a locally free sheaf $\mathcal{V}$ over $\Delta^*$
\begin{equation}
\mathcal{V}:=V_{\mathbb{Z}} \otimes \mathcal{O}_{\Delta^*},
\end{equation}
with dual denoted by $\mathcal{V}^{\vee}$. For every $\varphi \in \Delta^*$, $\mathcal{X}_{\varphi}$ is a projective complex manifold, hence $H^q(\mathcal{X}_{\varphi},\mathbb{C})$ admits a Hodge decomposition \cite{Schmid}
\begin{equation}
H^q(\mathcal{X}_{\varphi},\mathbb{C}) = \oplus_{0 \leq k \leq q} \, H^{k,q-k} (\mathcal{X}_{\varphi}),
\end{equation}
which induces a Hodge filtration on $H^q(\mathcal{X}_{\varphi},\mathbb{C})$ through
\begin{equation}
F_{\varphi}^p :=\oplus_{k \geq p}\, H^{k,q-k} (\mathcal{X}_{\varphi}).
\end{equation} 
The complex vector space $F_{\varphi}^p$ varies holomorphically with respect to $\varphi$, and its union forms a holomorphic vector bundle over $\Delta^*$, whose sheaf of sections is a locally free sheaf $\mathcal{F}^p$ that induces a subsheaf filtration on $\mathcal{V}$ \cite{Schmid}. From the standard comparison isomorphism, there are the following canonical isomorphisms \cite{SteenbrinkVMHS}
\begin{equation} \label{eq:FdeltaIso}
I_{\text{B}}: \mathscr{V}^{\text{an}}_{\mathbb{C}}\vert_{\Delta^*} \xrightarrow{\sim} \mathcal{V}, ~
I_{\text{B}}:\mathscr{F}^{p,\text{an}}_{\mathbb{C}}|_{\Delta^*} \xrightarrow{\sim} \mathcal{F}^p.
\end{equation}
On the sheaf $\mathcal{V}$, there is a Gauss-Manin connection $\nabla$ 
\begin{equation}
\nabla:\mathcal{V} \rightarrow \Omega^1_{\Delta^*} \otimes_{\mathcal{O}_{\Delta^*}} \mathcal{V},
\end{equation}
which is the unique connection such that the local sections of $V_{\mathbb{C}}$ are flat \cite{Schmid}. Moreover, it satisfies the Griffiths transversality \cite{Zein,Schmid}
\begin{equation}
\nabla \mathcal{F}^p \subset \Omega^1_{\Delta^*} \otimes_{\mathcal{O}_{\Delta^*}} \mathcal{F}^{p-1}.
\end{equation}
In fact the comparison isomorphism \ref{eq:FdeltaIso} sends the connection $\nabla_{\mathbb{C}}^{\text{an}}$ (restricted to $\mathscr{V}^{\text{an}}_{\mathbb{C}}|_{\Delta^*}$) to the connection $\nabla$, hence Gauss-Manin connection is unique \cite{PetersSteenbrink,Hain,DeligneMirror,DeligneEquations}.

\subsection{Canonical extension over the singular fiber} \label{sectionDCE}

Let $\Omega_{X/S}^*(\log\,Y)$ be the complex of sheaves of relative algebraic forms over $S$ with at worst logarithmic poles along the normal crossing divisor $Y$, as before it has a filtration $F^p\,\Omega_{X/S}^*(\log\,Y)$ defined as \ref{eq:rationalfiltercomplex} \cite{SteenbrinkLMHS}. Over the curve $S$, the sheaf $\widetilde{\mathscr{V}}_{\mathbb{Q}}$ and its filtration $\widetilde{\mathscr{F}}^p_{\mathbb{Q}}$ are defined by
\begin{equation}
\begin{aligned}
\widetilde{\mathscr{V}}_{\mathbb{Q}}&:=\mathbb{R}^q \,\pi_{\mathbb{Q},*}(\Omega_{X/S}^*(\log\,Y)), \\
\widetilde{\mathscr{F}}^p_{\mathbb{Q}}&:=\mathbb{R}^q \,\pi_{\mathbb{Q},*}(F^p\,\Omega_{X/S}^*(\log\,Y)), \\
\end{aligned}
\end{equation}
which form the canonical extensions of $\mathscr{V}_{\mathbb{Q}}$ and $\mathscr{F}^p_{\mathbb{Q}}$ respectively. The Gauss-Manin connection $\nabla_{\mathbb{Q}}$ of $\mathscr{V}_{\mathbb{Q}}$ can be canonically extended to a connection $\widetilde{\nabla}_{\mathbb{Q}}$ of $\widetilde{\mathscr{V}}_{\mathbb{Q}}$ which has a logarithmic pole along the rational point $0\in S$ with a nilpotent residue, and moreover this property also determines the extensions of $\mathscr{V}_{\mathbb{Q}}$ and $\mathscr{F}^p_{\mathbb{Q}}$ uniquely \cite{DeligneMirror,SteenbrinkLMHS}. Similarly, the extensions of $\mathscr{V}_{\mathbb{C}}$, $\mathscr{V}^{\text{an}}_{\mathbb{C}}$, $\mathscr{F}^p_{\mathbb{C}}$ and $\mathscr{F}^{p,\text{an}}_{\mathbb{C}}$ are obtained from the complexifications and analytifications of $\widetilde{\mathscr{V}}_{\mathbb{Q}}$ and $\widetilde{\mathscr{F}}^p_{\mathbb{Q}}$, which will be denoted by $\widetilde{\mathscr{V}}_{\mathbb{C}}$, $\widetilde{\mathscr{V}}^{\text{an}}_{\mathbb{C}}$, $\widetilde{\mathscr{F}}^p_{\mathbb{C}}$ and $\widetilde{\mathscr{F}}^{p,\text{an}}_{\mathbb{C}}$ respectively. Over the disc $\Delta$, $\widetilde{\mathscr{V}}^{\text{an}}_{\mathbb{C}}$ and $\widetilde{\mathscr{F}}^{p,\text{an}}_{\mathbb{C}}$ can also be constructed by \cite{SteenbrinkLMHS}
\begin{equation} \label{eq:DeligneAnalyticDelta}
\begin{aligned}
\widetilde{\mathscr{V}}^{\text{an}}_{\mathbb{C}}\vert_{\Delta}&=\mathbb{R}^q \,\pi_{\Delta,*}(\Omega_{\mathcal{X}/\Delta}^*(\log\,\mathcal{X}_0)),\\
\widetilde{\mathscr{F}}^{p,\text{an}}_{\mathbb{C}} \vert_{\Delta}&=\mathbb{R}^q \,\pi_{\Delta,*}(F^p\,\Omega_{\mathcal{X}/\Delta}^*(\log\,\mathcal{X}_0)).
\end{aligned}
\end{equation}
Under the comparison isomorphism \ref{eq:FdeltaIso}, $\widetilde{\mathscr{V}}^{\text{an}}_{\mathbb{C}}\vert_{\Delta}$ and $\widetilde{\mathscr{F}}^{p,\text{an}}_{\mathbb{C}} \vert_{\Delta}$ form the canonical extensions of $\mathcal{V}$ and $\mathcal{F}^p$ that will be denoted by $\widetilde{\mathcal{V}}$ and $\widetilde{\mathcal{F}}^p$ respectively \cite{DeligneMirror,Hain,Schmid}. Let us now look at the fibers of these extensions at $0 \in S$. The fiber $\widetilde{\mathscr{V}}_{\mathbb{Q}}\vert_0$ is a rational vector space given by \cite{SteenbrinkLMHS}
\begin{equation}
\widetilde{\mathscr{V}}_{\mathbb{Q}}\vert_0:=(\widetilde{\mathscr{V}}_{\mathbb{Q}})_0 \otimes_{\mathcal{O}_{S,0}} \mathcal{O}_{S,0}/\mathfrak{m}_{S,0} = \mathbb{H}^q \,(Y, \Omega_{X/S}^*(\log\,Y)|_Y),
\end{equation}
i.e. the hypercohomology of the restriction of the complex $\Omega_{X/S}^*(\log\,Y)$ to $Y$. By Serre's GAGA, there are canonical isomorphisms
\begin{equation}
\widetilde{\mathscr{V}}^{\text{an}}_{\mathbb{C}}\vert_0 = \widetilde{\mathscr{V}}_{\mathbb{C}}\vert_0 = \widetilde{\mathscr{V}}_{\mathbb{Q}}\vert_0 \otimes_{\mathbb{Q}} \mathbb{C}.
\end{equation}
The fiber $\widetilde{\mathscr{V}}^{\text{an}}_{\mathbb{C}}\vert_0$ can also be described as the hypercohomology \cite{SteenbrinkLMHS, PetersSteenbrink}
\begin{equation}
\widetilde{\mathscr{V}}^{\text{an}}_{\mathbb{C}}\vert_0:=\mathbb{H}^q(\mathcal{X}_0,\Omega_{\mathcal{X}/\Delta}^*(\log\,\mathcal{X}_0)|_{\mathcal{X}_0}).
\end{equation}
Given a closed point $\varphi$ of $S^*$, the fiber $\widetilde{\mathscr{V}}_{\mathbb{Q}} \vert_{\varphi}$ is a vector space over the residue field $\kappa(\varphi)$ of $\varphi$ given by \cite{KatzOda, PetersSteenbrink}
\begin{equation}
\widetilde{\mathscr{V}}_{\mathbb{Q}}\vert_{\varphi}=\mathbb{H}^q \,(Y, \Omega_{X/S}^*(\log\,Y)|_{X_{\varphi}}) = H^q_{\text{dR}}(X_{\varphi}),
\end{equation}
where $X_{\varphi}$ is a variety defined over $\kappa(\varphi)$. While in the analytic case, given a point $\varphi \in \Delta^*$, the fiber $\widetilde{\mathscr{V}}^{\text{an}}_{\mathbb{C}} \vert_{\varphi}$ is \cite{SteenbrinkLMHS, PetersSteenbrink}
\begin{equation}
\widetilde{\mathscr{V}}^{\text{an}}_{\mathbb{C}}\vert_{\varphi}=\mathbb{H}^q(\mathcal{X}_{\varphi},\Omega_{\mathcal{X}/\Delta}^*(\log\,\mathcal{X}_0)|_{\mathcal{X}_{\varphi}})=H^q_{\text{dR}}(X_{\varphi}^{\text{an}}) \simeq H^q(\mathcal{X}_{\varphi},\mathbb{C})
\end{equation}  
where we have used the comparison isomorphism \ref{eq:FdeltaIso} in the last isomorphism. Choose and fix a point $\varphi_0 \in \Delta^*$, then the local system $V_{\mathbb{C}}$ is uniquely determined by a representation of the fundamental group $\pi_1(\Delta^*,\varphi_0)$ \cite{Zein, Schmid}
\begin{equation} \label{eq:MonoPsifirst}
\Psi:\pi_1(\Delta^*,\varphi_0) \rightarrow  \text{GL}(H^q(\mathcal{X}_{\varphi_0},\mathbb{C})).
\end{equation}
The fundamental group $\pi_1(\Delta^*,\varphi_0)$ is isomorphic to $\mathbb{Z}$, and let us choose and fix a generator $T$ of it, then this representation is uniquely determined by the action of $T$ on $H^q(\mathcal{X}_{\varphi_0},\mathbb{C})$. The action of $T$ on $H^q(\mathcal{X}_{\varphi_0},\mathbb{C})$ can be extended to an automorphism of the sheaf $\widetilde{\mathcal{V}}$ (or $\widetilde{\mathscr{V}}^{\text{an}}_{\mathbb{C}}\vert_{\Delta}$), which induces an automorphism $T_0$ of the fiber $\widetilde{\mathcal{V}}\vert_0$ (or $\widetilde{\mathscr{V}}^{\text{an}}_{\mathbb{C}}\vert_0$). For more details, check Proposition 11.2 of \cite{PetersSteenbrink}.  On the other hand, let $\text{Res}_0$ be the residue map from $\Omega_S^1(\log \,0)$ to $\mathbb{C}$ given by
\begin{equation}
\text{Res}_0\,(\,g(t)\,dt/t\,):=g(0).
\end{equation}
The following homomorphism from the germ $(\widetilde{\mathscr{V}}^{\text{an}}_{\mathbb{C}})_0$ to the fiber $\widetilde{\mathscr{V}}^{\text{an}}_{\mathbb{C}}\vert_0$ vanishes on the ideal $(\widetilde{\mathscr{V}}^{\text{an}}_{\mathbb{C}})_0  \otimes_{\mathcal{O}_{S,0}} \mathfrak{m}_{S,0}$
\begin{equation}
\left(\text{Res}_0 \otimes (\otimes_{\mathcal{O}_{S,0}} \mathcal{O}_{S,0}/\mathfrak{m}_{S,0})\right)\circ \nabla_{\mathbb{C}}^{\text{an}} :(\widetilde{\mathscr{V}}^{\text{an}}_{\mathbb{C}})_0 \rightarrow \widetilde{\mathscr{V}}^{\text{an}}_{\mathbb{C}}\vert_0,
\end{equation}
hence it induces an endomorphism of $\widetilde{\mathscr{V}}^{\text{an}}_{\mathbb{C}}\vert_0$ that is denoted by $N_0$ and called the residue of $\nabla_{\mathbb{C}}^{\text{an}}$ at 0 \cite{SteenbrinkLMHS}. From Theorem II 3.11 of \cite{DeligneEquations} or Corollary 11.17 of \cite{PetersSteenbrink} we have
\begin{equation} \label{eq:T0N0Exp}
T_0=\exp(-2 \pi i\, N_0)
\end{equation}  
Furthermore, Corollary 11.19 of \cite{PetersSteenbrink} tells us that all the eigenvalues of $N_0$ are integers, therefore all the eigenvalues of $T_0$ are 1, which immediately implies that $T_0$ is unipotent.

\subsection{Limit mixed Hodge structures} \label{SectionlimitMHS}

We now give an overview of Steenbrink's construction of limit mixed Hodge structures, and the readers are referred to \cite{SteenbrinkLMHS, Illusie} for more complete treatments. First recall that we have fixed a universal cover $\widetilde{\Delta}^*$ of $\Delta^*$ in Remark \ref{FixedParameter}.  The nearby cycle sheaf $\text{R}\Psi_{\pi_{\Delta}}(\Lambda)$, where $\Lambda$ is $\mathbb{Z}$, $\mathbb{Q}$ or $\mathbb{C}$, is a complex of sheaves over the singular fiber $\mathcal{X}_0$ \cite{Illusie, SteenbrinkLMHS}. In the paper \cite{SteenbrinkLMHS}, Steenbrink has constructed the following data.
\begin{enumerate}
\item A representative of $\text{R}\Psi_{\pi_{\Delta}} \mathbb{Z}$ in the derived category $D^+(\mathcal{X}_0,\mathbb{Z})$.

\item A representative of $(\text{R}\Psi_{\pi_{\Delta}} \mathbb{Q},\,W_*)$  in the filtered derived category $D^+F(\mathcal{X}_0,\mathbb{Q})$, where $W_*$ is an increasing filtration of $\text{R}\Psi_{\pi_{\Delta}} \mathbb{Q}$ and
\begin{equation}
\text{R}\Psi_{\pi_{\Delta}} \mathbb{Q} \simeq \text{R}\Psi_{\pi_{\Delta}} \mathbb{Z} \otimes \mathbb{Q} ~\text{in}~D^+(\mathcal{X}_0,\mathbb{Q}).
\end{equation}

\item A representative of $(\text{R}\Psi_{\pi_{\Delta}} \mathbb{C},W_*,F^*)$ in the bifiltered derived category $D^+F_2(\mathcal{X}_0,\mathbb{C})$, where $W_*$ is an increasing filtration of $\text{R}\Psi_{\pi_{\Delta}} \mathbb{C}$ and $F^*$ is a decreasing filtration of $\text{R}\Psi_{\pi_{\Delta}} \mathbb{C}$ such that
\begin{equation}
(\text{R}\Psi_{\pi_{\Delta}} \mathbb{C},W_*) \simeq (\text{R}\Psi_{\pi_{\Delta}} \mathbb{Q} \otimes \mathbb{C},W_*) ~\text{in}~D^+F(\mathcal{X}_0,\mathbb{C}).
\end{equation}

\end{enumerate}
These constructions depend on the choice of $\varphi$, $\log \, \varphi$ and a universal cover $\widetilde{\Delta}^*$ of $\Delta^*$ \cite{Illusie, SteenbrinkLMHS}, all of which have been fixed in Remark \ref{FixedParameter}. From \cite{SteenbrinkLMHS}, we have the following important theorem.
\begin{theorem}
The data
\begin{equation} \label{eq:LimitMHC}
\left(\text{R}\Psi_{\pi_{\Delta}} \mathbb{Z},(\text{R}\Psi_{\pi_{\Delta}} \mathbb{Q},W_*),(\text{R}\Psi_{\pi_{\Delta}} \mathbb{C},W_*,F^*)\right)
\end{equation}
forms a cohomological mixed Hodge complex of sheaves in the sense of Deligne \cite{DeligneIII}.
\end{theorem}
\begin{proof}
See Chapter 11 of the book \cite{PetersSteenbrink}.
\end{proof}

Let us denote the triangulated category of $\mathbb{Z}$-mixed Hodge complexes by $D^*_{\textbf{MHS}_{\mathbb{Z}}}$, where $*$ is the boundedness condition, e.g. $*$ can be ${\emptyset}$, $+$, $-$ or $b$. From Proposition 8.1.7 of \cite{DeligneIII}, the mixed Hodge complex of sheaves \ref{eq:LimitMHC} defines an object of $D_{\textbf{MHS}_{\mathbb{Z}}}^+$
\begin{equation} \label{eq:LimitZZMHC}
\left(\text{R}\Gamma(\text{R}\Psi_{\pi_{\Delta}} \mathbb{Z}),\text{R}\Gamma(\text{R}\Psi_{\pi_{\Delta}} \mathbb{Q},W_*),\text{R}\Gamma(\text{R}\Psi_{\pi_{\Delta}} \mathbb{C},W_*,F^*)\right).
\end{equation}
From \cite{BeilinsonMHC}, for every $q \in \mathbb{Z}$ there exists a cohomological functor $\underline{H}^q$ from $D_{\textbf{MHS}_{\mathbb{Z}}}^+$ to $\textbf{MHS}_{\mathbb{Z}}$
\begin{equation}
\underline{H}^q:D_{\textbf{MHS}_{\mathbb{Z}}}^+ \rightarrow \textbf{MHS}_{\mathbb{Z}},
\end{equation}  
which sends the mixed Hodge complex \ref{eq:LimitZZMHC} to the MHS
\begin{equation} \label{eq:LimitMHCMHS}
(\underline{H}^q\circ\text{R}\Gamma(\text{R}\Psi_{\pi_{\Delta}} \mathbb{Z}),\underline{H}^q \circ \text{R}\Gamma(\text{R}\Psi_{\pi_{\Delta}} \mathbb{Q},W_*),\underline{H}^q \circ \text{R}\Gamma(\text{R}\Psi_{\pi_{\Delta}} \mathbb{C},W_*,F^*)).
\end{equation}
Furthermore, the MHS \ref{eq:LimitMHCMHS} is also described by
\begin{equation} \label{eq:OriginalLMHS}
\left(\mathbb{H}^q(\mathcal{X}_0,\text{R}\Psi_{\pi_{\Delta}} \mathbb{Z}),( \mathbb{H}^q(\mathcal{X}_0,\text{R}\Psi_{\pi_{\Delta}} \mathbb{Q}),W_*),(\mathbb{H}^q(\mathcal{X}_0,\text{R}\Psi_{\pi_{\Delta}} \mathbb{C}),W_*,F^*)\right).
\end{equation}
Steenbrink proves the following important proposition in \cite{SteenbrinkLMHS}.
\begin{proposition} \label{FiberComplexIso}
There is a quasi-isomorphism between the complexes of sheaves $\text{R}\Psi_{\pi_{\Delta}} \mathbb{C}$ and $\Omega_{\mathcal{X}/\Delta}^*(\log\,\mathcal{X}_0)|_{\mathcal{X}_0}$ in the derived category $D^+(\mathcal{X}_0,\mathbb{C})$, which depends on the choices of $\varphi$ and $\log \, \varphi$.
\end{proposition}
\begin{proof}
See Chapter 11 of the book \cite{PetersSteenbrink}.
\end{proof}

Hence from this proposition we have 
\begin{equation} \label{eq:SingularfiberIso}
\widetilde{\mathscr{V}}^{\text{an}}_{\mathbb{C}}\vert_0=\mathbb{H}^q(\mathcal{X}_0,\Omega_{\mathcal{X}/\Delta}^*(\log\,\mathcal{X}_0)|_{\mathcal{X}_0}) \simeq \mathbb{H}^q(\mathcal{X}_0, \text{R}\Psi_{\pi_{\Delta}} (\mathbb{C})),
\end{equation}
and the MHS \ref{eq:LimitMHCMHS} is called the limit mixed Hodge structure \cite{PetersSteenbrink}. Steenbrink also constructs a morphism $\nu_N$ on $\text{R}\Psi_{\pi_{\Delta}} \mathbb{C}$ of the form \cite{SteenbrinkLMHS}
\begin{equation} \label{eq:NuNQ}
\nu_N:\text{R}\Psi_{\pi_{\Delta}} \mathbb{Q} \rightarrow \text{R}\Psi_{\pi_{\Delta}} \mathbb{Q}(-1),
\end{equation}
where $(-1)$ is the Tate twist by $\mathbb{Q}(-1)$ in $D^+(\mathcal{X}_0,\mathbb{Q})$. After taking hypercohomology, the morphism \ref{eq:NuNQ} induces the homomorphism $N_0$ defined in last section
\begin{equation} \label{eq:Nmorphism1}
N_0:\mathbb{H}^n(\mathcal{X}_0, \text{R}\Psi_{\pi_{\Delta}} (\mathbb{Q})) \rightarrow \mathbb{H}^n(\mathcal{X}_0, \text{R}\Psi_{\pi_{\Delta}} (\mathbb{Q}))(-1),
\end{equation}
which also determines the weight filtration $W_*$ in \ref{eq:OriginalLMHS} uniquely \cite{Illusie, SteenbrinkLMHS, SteenbrinkVMHS}. From \cite{SteenbrinkLMHS, SteenbrinkVMHS}, the Hodge filtration $F^*$ in \ref{eq:OriginalLMHS} is in fact the filtration on $\widetilde{\mathscr{V}}^{\text{an}}_{\mathbb{C}}\vert_0$ given by the fiber $\widetilde{\mathscr{F}}^{p,\text{an}}_{\mathbb{C}} \vert_0$
\begin{equation}
F^p=\widetilde{\mathscr{F}}^{p,\text{an}}_{\mathbb{C}} \vert_0 = \widetilde{\mathscr{F}}^p_{\mathbb{C}} \vert_0=\widetilde{\mathscr{F}}^p_{\mathbb{Q}} \vert_0 \otimes_{\mathbb{Q}} \mathbb{C}.
\end{equation}

\subsection{Limit mixed Hodge complex}

The $\mathbb{Z}$-mixed Hodge complex \ref{eq:LimitZZMHC} naturally defines a $\mathbb{Q}$-mixed Hodge complex $\mathbf{Y}^{\bullet}$ by forgetting its integral structure
\begin{equation} \label{eq:LimitQQMHC}
\mathbf{Y}^{\bullet}:=\big(\text{R}\Gamma(\text{R}\Psi_{\pi_{\Delta}} \mathbb{Q}),\text{R}\Gamma(\text{R}\Psi_{\pi_{\Delta}} \mathbb{Q},W_*),\text{R}\Gamma(\text{R}\Psi_{\pi_{\Delta}} \mathbb{C},W_*,F^*)\big),
\end{equation}
whose hypercohomology is the underlying rational MHS of \ref{eq:OriginalLMHS}
\begin{equation} \label{eq:RationalLMHS}
\underline{H}^q(\mathbf{Y}^{\bullet})=\big(\mathbb{H}^q(\mathcal{X}_0,\text{R}\Psi_{\pi_{\Delta}} \mathbb{Q}), ( \mathbb{H}^q(\mathcal{X}_0,\text{R}\Psi_{\pi_{\Delta}} \mathbb{Q}),W_*),(\mathbb{H}^q(\mathcal{X}_0,\text{R}\Psi_{\pi_{\Delta}} \mathbb{C}),W_*,F^*)\big).
\end{equation}

Now we need the following lemma from \cite{SGA7}. 
\begin{lemma} \label{SGAlemma}
$\mathbb{H}^q(\mathcal{X}_0,\text{R}\Psi_{\pi_{\Delta}} \Lambda)$ is isomorphic to $H^q(\mathcal{X}_{\varphi},\Lambda)$ where $\Lambda$ is $\mathbb{Z}$, $\mathbb{Q}$ or $\mathbb{C}$.
\end{lemma}

Therefore $\mathbb{H}^q(\mathcal{X}_0,\text{R}\Psi_{\pi_{\Delta}} \mathbb{Q})$ is $0$ when $q <0$ or $q > 2\,\text{dim}\,\mathcal{X}_{\varphi}$, which immediately implies that $\underline{H}^q(\mathbf{Y}^{\bullet})$ is also $0$, thus $\mathbf{Y}^{\bullet}$ is essentially an object of $D^b_{\textbf{MHS}_{\mathbb{Q}}}$. From Theorem 3.4 of \cite{BeilinsonMHC}, the natural functor $D^b(\textbf{MHS}_{\mathbb{Q}}) \hookrightarrow D^b_{\textbf{MHS}_{\mathbb{Q}}}$ is an equivalence of categories, through which  $\mathbf{Y}^{\bullet}$ determines a complex $\mathbf{Z}^{\bullet}$ of $D^b(\textbf{MHS}_{\mathbb{Q}})$ such that
\begin{equation}  \label{eq:LimitHCQQMHC}
H^q(\mathbf{Z}^{\bullet})=\underline{H}^q(\mathbf{Y}^{\bullet}),\, \forall \, q \in \mathbb{Z}.
\end{equation}
In Section \ref{LimitMixedMotive}, we will construct a mixed motive of $\textbf{DM}_{\text{gm}}(\mathbb{Q},\mathbb{Q})$ whose Hodge realisation is conjectured to be isomorphic to $\mathbf{Z}^{\bullet}$.

\subsection{Schmid's construction of limit MHS}
Before we discuss Schmid's construction of limit MHS \cite{Schmid}, we need to introduce Deligne's canonical extension \cite{DeligneMirror,Hain,Zein,DeligneEquations}. From the assumptions on the fibration \ref{eq:FibrationSRational}, the operator $\Psi(T)$ is unipotent, so let us define $N$ as \cite{MarkGross}
\begin{equation} \label{eq:Ndefn}
N:=\log \Psi(T).
\end{equation}

\begin{remark}
Some literature instead defines $N$ as $-\log \Psi(T)/(2 \pi  i)$, which is  in accordance with formula \ref{eq:T0N0Exp}. However the definition \ref{eq:Ndefn} will make the computations of the limit MHS simpler.
\end{remark}

Suppose $\varphi_0$ is an arbitrary point of $\Delta^*$. Every element $\xi$ of $H^q(\mathcal{X}_{\varphi_0},\mathbb{C})$ extends to a multi-valued section $\xi(\varphi)$ of the local system $V_{\mathbb{C}}$, and it induces a single-valued section $\widehat{\xi}(\varphi)$ of $\mathcal{V}$ defined by \cite{DeligneMirror,Hain,Zein,DeligneEquations}
\begin{equation}
\widehat{\xi}(\varphi):=\exp\, (-\frac{\log \varphi}{2 \pi i}N) \,\xi(\varphi).
\end{equation}
Suppose $\{\sigma^a \}$ form a basis of $H^q(\mathcal{X}_{\varphi_0},\mathbb{C})$, then the induced sections $\{\widehat{\sigma}^a(\varphi) \}$ will form a frame of $\mathcal{V}$ that defines a trivialisation of $\mathcal{V}$ over $\Delta^*$. This trivialisation naturally induces an extension of $\mathcal{V}$ to a locally free sheaf $\widetilde{\mathcal{V}}$ over $\Delta$, while the sheaf $\mathcal{F}^p$ also extends to a locally free sheaf $\widetilde{\mathcal{F}}^p$ over $\Delta$ \cite{DeligneMirror,Hain,Zein,DeligneEquations,Schmid}. This extension is called Deligne's canonical extension.  
\begin{proposition}
The isomorphisms in \ref{eq:FdeltaIso} extend to the following isomorphisms
\begin{equation} 
I_{\text{B}}:\widetilde{\mathscr{V}}^{\text{an}}_{\mathbb{C}}\vert_{\Delta} \xrightarrow{\sim} \widetilde{\mathcal{V}},~ 
I_{\text{B}}:\widetilde{\mathscr{F}}^{p,\text{an}}_{\mathbb{C}}|_{\Delta^*} \xrightarrow{\sim} \widetilde{\mathcal{F}}^p,
\end{equation}
which induce isomorphisms between their fibers over 0
\begin{equation} \label{eq:VtildeIsoToAna}
I_{\text{B}}:\widetilde{\mathscr{V}}^{\text{an}}_{\mathbb{C}}\vert_0  \xrightarrow{\sim} \widetilde{\mathcal{V}}\vert_0,~
I_{\text{B}}:\widetilde{\mathscr{F}}^{p,\text{an}}_{\mathbb{C}}|_0 \xrightarrow{\sim} \widetilde{\mathcal{F}}^p \vert_0.
\end{equation}
\end{proposition}
\begin{proof}
See \cite{SteenbrinkLMHS,Hain, SteenbrinkVMHS,DeligneMirror,DeligneEquations}.
\end{proof}

In Deligne's canonical extension, the section $\widehat{\xi}(\varphi)$ of $\mathcal{V}$ extends to a section of $\widetilde{\mathcal{V}}$ that will be denoted by $\widetilde{\xi}(\varphi)$, in particular the frame $\{\widehat{\sigma}^a(\varphi) \}$ of $\mathcal{V}$ extends to a frame of $\widetilde{\mathcal{V}}$ that will be denoted by $\{\widetilde{\sigma}^a(\varphi) \}$. Therefore $\{\widetilde{\sigma}^a(0)\}$ form a basis of the fiber $\widetilde{\mathcal{V}}\vert_0$, and we have an isomorphism defined by
\begin{equation} \label{eq:Newrho}
\rho_{\varphi_0}: H^q(\mathcal{X}_{\varphi_0},\mathbb{C}) \rightarrow \widetilde{\mathcal{V}}\vert_0, ~~\rho_{\varphi_0}(\xi)=\widetilde{\xi}(0),
\end{equation}  
through which the lattice $H^q(\mathcal{X}_{\varphi_0},\mathbb{Z})$ defines a lattice structure on $\widetilde{\mathcal{V}}\vert_0$  \cite{Schmid,Hain,DeligneMirror}
\begin{equation}
\widetilde{\mathcal{V}}\vert_{0,\mathbb{Z}}:=\rho_{\varphi_0}\,(H^q(\mathcal{X}_{\varphi_0},\mathbb{Z})).
\end{equation}
Similarly $H^q(\mathcal{X}_{\varphi_0},\mathbb{Q})$ defines a rational structure on $\widetilde{\mathcal{V}}\vert_{0}$ which satisfies \cite{Schmid, Hain, PetersSteenbrink}
\begin{equation}
\widetilde{\mathcal{V}}\vert_{0,\mathbb{Q}}=\widetilde{\mathcal{V}}\vert_{0,\mathbb{Z}} \otimes_{\mathbb{Z}} \mathbb{Q}.
\end{equation}
The action of $T$ in  \ref{eq:MonoPsifirst} is unipotent, hence it defines a weight filtration $W_*$ on $\widetilde{\mathcal{V}}\vert_{0,\mathbb{Q}}$, while the fiber $\widetilde{\mathcal{F}}^p \vert_0$ defines a Hodge filtration $F^*$ on $\widetilde{\mathcal{V}} \vert_0$. We have the following important theorem of Schmid.

\begin{theorem}
The following data forms a MHS
\begin{equation} \label{eq:NewLMHS}
\left(\widetilde{\mathcal{V}}\vert_{0,\mathbb{Z}},(\widetilde{\mathcal{V}}\vert_{0,\mathbb{Q}},W_*),(\widetilde{\mathcal{V}}\vert_{0},W_*,F^*)\right).
\end{equation}
\end{theorem}

\begin{proof}
See \cite{Schmid, PetersSteenbrink}.
\end{proof}
Moreover, Schmid's construction of limit MHS is compatible with Steenbrink's more algebraic construction of limit MHS.
\begin{proposition} \label{SteenbrinkSchmid}
The underlying rational MHS of \ref{eq:NewLMHS}
\begin{equation} \label{eq:NewRationalLMHS}
\left(\widetilde{\mathcal{V}}\vert_{0,\mathbb{Q}},(\widetilde{\mathcal{V}}\vert_{0,\mathbb{Q}},W_*),(\widetilde{\mathcal{V}}\vert_{0},W_*,F^*) \right)
\end{equation}
is isomorphic to the limit MHS $\underline{H}^q(\mathbf{Y}^{\bullet})$ in \ref{eq:RationalLMHS}.
\end{proposition}
\begin{proof}
See \cite{SteenbrinkLMHS, SteenbrinkVMHS, PetersSteenbrink}.
\end{proof}

\section{Limit mixed motives} \label{LimitMixedMotive}

This section is devoted to the construction of the limit mixed motive by Ayoub's motivic nearby cycle functor, but first we need to briefly discuss the construction of the category of \'etale motivic sheaves, the details of which are left to the papers \cite{AyoubT1, AyoubMS}. 

\subsection{A naive construction of the \'etale motivic sheaves}

Suppose $\Lambda$ is a commutative ring and in this paper we are mostly concerned with the case where $\Lambda$ is $\mathbb{Q}$. In order to satisfy some technical assumptions, all schemes in this section are assumed to be separated, Noetherian and of finite Krull dimension. Given a base scheme $U$, the category of \'etale motivic sheaves with coefficients ring $\Lambda$ will be denoted by $\textbf{DA}^{\text{\'et}}(U,\Lambda)$. In this section, we will follow \cite{AyoubMS} and give an `incorrect' naive construction of the category $\textbf{DA}^{\text{\'et,naive}}(U,\Lambda)$, which nonetheless catches some essences of $\textbf{DA}^{\text{\'et}}(U,\Lambda)$ and suffices for this paper. The rigorous construction of $\textbf{DA}^{\text{\'et}}(U,\Lambda)$ is left to \cite{AyoubMS, AyoubT1}.

Let $\textbf{Sm}/U$ be the category of smooth $U$-schemes  endowed with \'etale topology and let $\textbf{Sh}_{\text{\'et}}(\textbf{Sm}/U;\Lambda)$ be the abelian category of \'etale sheaves on $\textbf{Sm}/U$ that take values in the category of $\Lambda$-modules \cite{MilneEC}. A smooth $U$-scheme $Z$ defines a presheaf through 
\begin{equation}
V\in \textbf{Sm}/U \mapsto \Lambda \otimes \text{Hom}_U(V,Z),
\end{equation}
where $\Lambda \otimes \text{Hom}_U(V,Z)$ is the $\Lambda$-module generated by the set $\text{Hom}_U(V,Z)$. The sheaf associated to this presheaf defined by $Z$ will be denoted by $\Lambda_{\text{\'et}}(Z)$. In this way we have found a Yoneda functor
\begin{equation}
\Lambda_{\text{\'et}}: \textbf{Sm}/U \rightarrow \textbf{Sh}_{\text{\'et}}(\textbf{Sm}/U;\Lambda),
\end{equation} 
which can be considered as the first-step linearisation of the category $\textbf{Sm}/S$.
It is easy to see that 
the \'etale sheaf $\Lambda_{\text{\'et}}(U)$ associated to the identity morphism of $U$ is the constant \'etale sheaf  $\underline{\underline{\Lambda}}$ on $\textbf{Sm}/U$.

The next step in the construction is to take $\mathbb{A}^1$-localisation. Let $D(\textbf{Sh}_{\text{\'et}}(\textbf{Sm}/U;\Lambda))$ be the derived category of $\textbf{Sh}_{\text{\'et}}(\textbf{Sm}/U;\Lambda)$, and let $\mathcal{T}_{\mathbb{A}^1}$ be the smallest full triangulated subcategory of $D(\textbf{Sh}_{\text{\'et}}(\textbf{Sm}/U;\Lambda))$ that is closed under arbitrary direct sums and also contains all the complexes of the form
\begin{equation}
\begin{tikzcd}
\cdots \arrow[r] & 0 \arrow[r] & \Lambda_{\text{\'et}}(\mathbb{A}_U^1 \times_U V) \arrow[r] &  \Lambda_{\text{\'et}}(V) \arrow[r] & 0 \arrow[r] & \cdots
\end{tikzcd}
\end{equation}
Here $V$ is a smooth $U$-scheme and the morphism from $\Lambda_{\text{\'et}}(\mathbb{A}_U^1 \times_U V)$ to $\Lambda_{\text{\'et}}(V)$ is induced by the natural projection $\mathbb{A}_U^1 \times_U V \rightarrow V$. Define the category $\textbf{DA}^{\text{\'et,eff}}(U,\Lambda)$ as the following Verdier quotient \cite{AyoubMS}
\begin{equation}
\textbf{DA}^{\text{\'et,eff}}(U,\Lambda):=D(\textbf{Sh}_{\text{\'et}}(\textbf{Sm}/U;\Lambda))/\mathcal{T}_{\mathbb{A}^1}.
\end{equation}  
Its objects are the same as that of $D(\textbf{Sh}_{\text{\'et}}(\textbf{Sm}/U;\Lambda))$, hence by abuse of notations, the objects of $\textbf{DA}^{\text{\'et,eff}}(U,\Lambda)$ will be denoted by the same symbols. As the name has implied, objects of $\textbf{DA}^{\text{\'et,eff}}(U,\Lambda)$ will be called the effective $U$-motives. The effect of the Verdier quotient is to invert the morphisms of $D(\textbf{Sh}_{\text{\'et}}(\textbf{Sm}/U;\Lambda))$ whose cones lie in $\mathcal{T}_{\mathbb{A}^1}$. For example, the cone of the morphism 
\begin{equation}
\Lambda_{\text{\'et}}(\mathbb{A}_U^1 \times_U V) \rightarrow \Lambda_{\text{\'et}}(V)
\end{equation}
is in $\mathcal{T}_{\mathbb{A}^1}$, hence it becomes an isomorphism in $\textbf{DA}^{\text{\'et,eff}}(U,\Lambda)$, so  $\Lambda_{\text{\'et}}(\mathbb{A}_U^1 \times_U V)$ is isomorphic to $\Lambda_{\text{\'et}}(V)$ in $\textbf{DA}^{\text{\'et,eff}}(U,\Lambda)$. 
\begin{definition}

Suppose the scheme $Z \in \textbf{Sm}/U$ is of finite presentation \cite{Hartshorne}. The category $\textbf{DA}_{\text{ct}}^{\text{\'et,eff}}(U,\Lambda)$ is by definition the smallest full triangulated subcategory of $\textbf{DA}^{\text{\'et,eff}}(U,\Lambda)$ that contains all the objects $\Lambda_{\text{\'et}}(Z)$ and is also closed under taking direct summand. Objects of $\textbf{DA}_{\text{ct}}^{\text{\'et,eff}}(U,\Lambda)$ will be called the constructible $U$-motives.
\end{definition}
The last step in the construction is stabilisation, and we will only discuss the naive stabilisation in this paper while leave the rigorous one to \cite{AyoubMS}. The injection $\infty_U \hookrightarrow \mathbb{P}^1_U$ of $U$-schemes defines a morphism
\begin{equation}
\Lambda_{\text{\'et}}(\infty_U) \rightarrow \Lambda_{\text{\'et}}(\mathbb{P}^1_U),
\end{equation}
whose cokernel in $\textbf{Sh}_{\text{\'et}}(\textbf{Sm}/U;\Lambda)$ will be denoted by $\Lambda_{\text{\'et}}(\mathbb{P}^1_U,\infty_U)$. The object $\Lambda_{\text{\'et}}(\mathbb{P}^1_U,\infty_U)$ defines a motive in $\textbf{DA}^{\text{\'et,eff}}(U,\Lambda)$ that will be called the Lefschetz motive
\begin{equation}
L:=\Lambda_{\text{\'et}}(\mathbb{P}^1_U,\infty_U).
\end{equation}
The naive stabilisation is to invert the Lefschetz motive $L$
\begin{equation}
\textbf{DA}^{\text{\'et,naive}}(U,\Lambda):=\textbf{DA}^{\text{\'et,eff}}(U,\Lambda)[L^{-1}].
\end{equation}
More precisely, objects of $\textbf{DA}^{\text{\'et,naive}}(U,\Lambda)$ are formal pairs $(M,m)$ where $M$ is an object of $\textbf{DA}^{\text{\'et,eff}}(U,\Lambda)$ and $m$ is an integer. Morphisms between two objects $(M,m)$ and $(N,n)$ are given by
\begin{equation}
 \varinjlim_{r \geq -\min(m,n)} \text{Hom}_{\textbf{DA}^{\text{\'et,eff}}(U,\Lambda)}(M\otimes L^{r+m},N \otimes L^{r+n}).
\end{equation}
This naive stabilisation has the merit of being very straightforward, but it suffers many technical problems, e.g. $\textbf{DA}^{\text{\'et,naive}}(U,\Lambda)$ is not even a triangulated category. However it still catches some essences of the category of \'etale motivic sheaves $\textbf{DA}^{\text{\'et}}(U,\Lambda)$. More precisely, let $\textbf{DA}_{\text{ct}}^{\text{\'et,naive}}(U,\Lambda)$ be the full subcategory of $\textbf{DA}^{\text{\'et,naive}}(U,\Lambda)$ generated by objects $(M,m)$ where $M$ is an object of $\textbf{DA}_{\text{ct}}^{\text{\'et,eff}}(U,\Lambda)$.  Let $\textbf{DA}_{\text{ct}}^{\text{\'et}}(U,\Lambda)$ be the full triangulated subcategory of $\textbf{DA}^{\text{\'et}}(U,\Lambda)$ generated by constructible objects, which is certainly the most important subcategory of $\textbf{DA}^{\text{\'et}}(U,\Lambda)$. Under some technical assumptions, which are all satisfied when $U$ is a quasiprojective variety over a field of characteristic 0 and the coefficient ring $\Lambda$ is $\mathbb{Q}$, the category $\textbf{DA}_{\text{ct}}^{\text{\'et,naive}}(U,\Lambda)$ is equivalent to $\textbf{DA}_{\text{ct}}^{\text{\'et}}(U,\Lambda)$ \cite{AyoubMS}. Since in this paper we will only be concerned with the constructible objects of $\textbf{DA}^{\text{\'et}}(U,\Lambda)$, the naive stabilisation will suffice for our purpose. The defect-free construction of $\textbf{DA}^{\text{\'et}}(U,\Lambda)$ is left to the paper \cite{AyoubMS}.

\begin{remark}
In the construction of $\textbf{DA}^{\text{\'et}}(U,\Lambda)$, there exists a covariant functor from $\textbf{Sm}/U$ to $\textbf{DA}^{\text{\'et}}(U,\Lambda)$, so its objects are homological mixed motives. While in Section \ref{VoevodskyMM}, the functor which sends a variety to a mixed motive is contravariant, so the motives there are cohomological. The difference is a dual operation \cite{LevineMM}. 
\end{remark}

The category $\textbf{DA}^{\text{\'et}}(U,\Lambda)$ satisfies Grothendieck's six operations formalism \cite{AyoubMS,AyoubT1}.  But in this paper we will only need one such operation, i.e. given a morphism $g:U \rightarrow V$, there exists a pushforward functor $g_*$
\begin{equation}
g_*: \textbf{Sh}_{\text{\'et}}(\textbf{Sm}/U;\Lambda) \rightarrow \textbf{Sh}_{\text{\'et}}(\textbf{Sm}/V;\Lambda),
\end{equation} 
such that for an \'etale sheaf $\mathcal{G}$ of $\textbf{Sh}_{\text{\'et}}(\textbf{Sm}/U;\Lambda)$ we have
\begin{equation}
g_*\,\mathcal{G}(W):=\mathcal{G}(W\times_V U),\,W \in \textbf{Sm}/V.
\end{equation}
The functor $g_*$ can be derived and it defines a functor $\text{R}g_*$ \cite{AyoubMS}
\begin{equation}
\text{R}g_*:\textbf{DA}^{\text{\'et,eff}}(U,\Lambda)  \rightarrow \textbf{DA}^{\text{\'et,eff}}(V,\Lambda).
\end{equation}
Moreover, the functor $g_*$ can be extended to the $L$-spectra, which can be derived and yields a functor $\text{R}g_*$ \cite{AyoubMS,AyoubT1}
\begin{equation}
\text{R}g_*:\textbf{DA}^{\text{\'et}}(U,\Lambda)  \rightarrow \textbf{DA}^{\text{\'et}}(V,\Lambda).
\end{equation}
What is important to us is that, the functor $\text{R}g_*$ sends constructible objects of $\textbf{DA}^{\text{\'et}}(U,\Lambda)$ to the constructible objects of $\textbf{DA}^{\text{\'et}}(V,\Lambda)$ \cite{AyoubT1}.

\subsection{Motivic nearby cycle functor}

The category $\textbf{DA}^{\text{\'et}}(U,\Lambda)$ also satisfies the nearby cycle formalism, whose realisation is the classical nearby cycle functor \cite{AyoubT1,AyoubNCF}. Since $0$ is a smooth point of $S$, the local ring $\mathscr{O}_{S,0}$ is a discrete valuation ring, and the affine scheme $\text{Spec}\,\mathscr{O}_{S,0}$ admits an injection into $S$ \cite{Hartshorne}
\begin{equation} \label{eq:localinjectionat0}
\text{Spec}\,\mathscr{O}_{S,0} \rightarrow S.
\end{equation}
Let $ \mathscr{O}^{\text{h}}_{S,0}$ be the henselisation of $\mathscr{O}_{S,0}$, then there is an injective local ring homomorphism from $\mathscr{O}_{S,0}$ to $ \mathscr{O}^{\text{h}}_{S,0}$ that induces a morphism \cite{MilneEC}
\begin{equation} \label{eq:henselizationInjection}
\text{Spec}\, \mathscr{O}^{\text{h}}_{S,0} \rightarrow \text{Spec}\,\mathscr{O}_{S,0}.
\end{equation}
The affine scheme $\text{Spec}\, \mathscr{O}^{\text{h}}_{S,0}$ consists of two points: a generic point $\eta$ and a closed point $s$ with residue field $\mathbb{Q}$. For simplicity, let us denote $\text{Spec}\, \mathscr{O}^{\text{h}}_{S,0}$ by $B$, and we obtain a henselian trait $(B,s,\eta)$
\begin{equation} \label{eq:henseliantrait}
\begin{tikzcd}
\eta \arrow[r] & B  & \arrow[l] s 
\end{tikzcd}
\end{equation}
The composition of \ref{eq:localinjectionat0} and \ref{eq:henselizationInjection} defines a morphism $i: B \rightarrow S$. Let $f:X_B \rightarrow B$ be the pull-back of $\pi_{\mathbb{Q}}$ \ref{eq:FibrationSRational} along $i$
\begin{equation}
\begin{tikzcd}
X_B \arrow[r] \arrow[d,"f"] & X \arrow[d,"\pi_{\mathbb{Q}}"]   \\
B \arrow[r,"i"]& S 
\end{tikzcd},
\end{equation}
and moreover the pull-backs of $f$ along $\eta \rightarrow B$ and $s \rightarrow B$ form a commutative diagram
\begin{equation} \label{eq:henseliantraitfibration}
\begin{tikzcd}
X_{\eta} \arrow[d,"f_{\eta}"] \arrow[r] & X_B \arrow[d,"f"] & \arrow[l] X_{s} \arrow[d,"f_s"]  \\
\eta \arrow[r] & B & \arrow[l] s 
\end{tikzcd}
\end{equation}
Notice that the fiber $ X_s$ is just $Y$. From \cite{AyoubT1,AyoubNCF}, there exists a motivic nearby cycle functor $\text{R}\Psi_f$
\begin{equation} \label{eq:MotivicNCF}
\text{R}\Psi_f:\textbf{DA}^{\text{\'et}}( X_{\eta},\mathbb{Q}) \rightarrow \textbf{DA}^{\text{\'et}}( X_s,\mathbb{Q}).
\end{equation}
and from Theorem 10.9 of \cite{AyoubNCF}, the functor $\text{R}\Psi_f$ sends constructible objects of $\textbf{DA}^{\text{\'et}}( X_{\eta},\mathbb{Q})$ to the constructible objects of $\textbf{DA}^{\text{\'et}}( X_s,\mathbb{Q})$. As noted earlier, the identity morphism of $X_{\eta}$ defines the constant \'etale sheaf $\mathbb{Q}_{\text{\'et}}(X_{\eta})$ on $\textbf{Sm}/X_{\eta}$, which produces a constructible motive of $\textbf{DA}^{\text{\'et}}(X_{\eta},\Lambda)$ that is also denoted by $\mathbb{Q}_{\text{\'et}}(X_{\eta})$. Thus $\text{R}\Psi_f(\mathbb{Q}_{\text{\'et}}(X_{\eta}))$ is a constructible motive of $\textbf{DA}^{\text{\'et}}( X_s,\mathbb{Q})$, which is called the  nearby motivic sheaf, whose realisations are the classical nearby cycle sheaves by Th\'eor\`eme 4.9 of \cite{AyoubBetti} in the Betti realisation case and by Th\'eor\`eme 10.11 of \cite{AyoubNCF} in the $\ell$-adic realisation case. The structure morphism $f_s$ in \ref{eq:henseliantraitfibration} yields a functor
\begin{equation}
\text{R}(f_{s})_*:\textbf{DA}^{\text{\'et}}( X_s,\mathbb{Q}) \rightarrow \textbf{DA}^{\text{\'et}}(\mathbb{Q}, \mathbb{Q}),
\end{equation} 
which sends the constructible objects of $\textbf{DA}^{\text{\'et}}( X_s,\mathbb{Q})$ to constructible objects of $\textbf{DA}^{\text{\'et}}(\mathbb{Q}, \mathbb{Q})$ \cite{AyoubT1}. The limit mixed motive $\mathcal{Z}$ is defined as
\begin{equation} \label{eq:limitmotive}
\mathcal{Z}:=\text{R}(f_{s})_* \circ \text{R}\Psi_f(\mathbb{Q}_{\text{\'et}}(X_{\eta})),
\end{equation}
which is a constructible object of $\textbf{DA}^{\text{\'et}}(\mathbb{Q}, \mathbb{Q})$. While from the Theorem 4.4 of \cite{AyoubMS}, the category $\textbf{DA}^{\text{\'et}}(\mathbb{Q}, \mathbb{Q})$ is equivalent to the category $\textbf{DM}^{\text{\'et}}(\mathbb{Q}, \mathbb{Q})$, which is discussed in Section 4.3 of \cite{AyoubMS}. From Theorem 14.30 of \cite{VoevodskyMM}, the dual of $\textbf{DM}^{\text{\'et}}(\mathbb{Q}, \mathbb{Q})$ is equivalent to the category $\textbf{DM}(\mathbb{Q}, \mathbb{Q})$ introduced in Section \ref{VoevodskyMM}, therefore the limit mixed motive $\mathcal{Z}$ defines a constructible object of $\textbf{DM}(\mathbb{Q}, \mathbb{Q})$, i.e. an object of $\textbf{DM}_{\text{gm}}(\mathbb{Q}, \mathbb{Q})$, that will be denoted by $\mathcal{Z}^{\vee}$. Since the realisations of the motive $\text{R}\Psi_f(\mathbb{Q}_{\text{\'et}}(X_{\eta}))$ are the classical nearby cycle sheaves by Th\'eor\`eme 4.9 of \cite{AyoubBetti} and Th\'eor\`eme 10.11 of \cite{AyoubNCF}, therefore the realisations of $\mathcal{Z}^{\vee}$ are complexes that compute the cohomologies of the classical nearby cycle sheaves (in Betti case or $\ell$-adic case). In fact, Ayoub has conjectured that more is true.

\begin{conjecture} \label{limitmotiveconjecture}
The Hodge realisation of the limit mixed motive $\mathcal{Z}^{\vee}$ is isomorphic to the complex $\mathbf{Z}^{\bullet}$ in \ref{eq:LimitHCQQMHC}, i.e. the Hodge realisation of the limit mixed motive $\mathcal{Z}^{\vee}$ is the complex in the derived category $D^b(\textbf{MHS}_{\mathbb{Q}})$ that computes the limit MHS in formula \ref{eq:RationalLMHS}.
\end{conjecture}

However currently this conjecture is not available in literature and the proof of it, even though could be considered as `routine,' can be technically very hard.

\section{Computations of the limit MHS} \label{ConstructionLMHS}

In this section, we will compute the limit MHS of the mirror family at the large complex structure limit. In the rest of this paper, we will focus on the fibration that forms a deformation of the mirror threefold. Given a one-parameter mirror pair $(M,W)$ of Calabi-Yau threefolds, the deformation of the mirror threefold $W$ is said to be rationally defined if there exists a fibration of the form $\pi_{\mathbb{Q}}$ \ref{eq:Piq} whose analytification forms a deformation of $W$. Furthermore, we will assume that $0$ is the large complex structure limit, the meaning of which will be explained later. Recall that the analytification of $\pi_{\mathbb{Q}}$ is $\pi_{\mathbb{C}}^{\text{an}}$, and let its restriction to $\mathcal{W}:=(\pi_{\mathbb{C}}^{\text{an}})^{-1}(\Delta)$ be
\begin{equation} \label{eq:pideltaW}
\pi_{\Delta}:\mathcal{W} \rightarrow \Delta,
\end{equation}
where the only singular fiber is $\mathcal{W}_0:=\pi_{\Delta}^{-1}(0)$, which is also the analytification of $Y=\pi_{\mathbb{Q}}^{-1}(0)$. Let the restriction of $\pi_{\Delta}$ to $\mathcal{W}^*:=\mathcal{W} - \mathcal{W}_0$  be $\pi_{\Delta^*}$
\begin{equation}
\pi_{\Delta^*}: \mathcal{W}^* \rightarrow \Delta^*.
\end{equation}
A smooth fiber $\mathcal{W}_{\varphi}$ of the fibration \ref{eq:pideltaW} is a Calabi-Yau threefold with Hodge diamond
\begin{equation}\label{eq:HodgeDiamond} 
\begin{tabular}{ c c c c c c c }
 &  &  & 1 &  &  &  \\ 
 &  & 0&   & 0&  &  \\   
 & 0&  & \,\,\,\,$h^{11}$ &  & 0&   \\  
1&  & 1&  &1 & & 1 \\ 
 & 0&  & \,\,\,\,$h^{11}$ &  & 0&   \\ 
 &  & 0&   & 0&  &  \\   
 &  &  & 1 &  &  &  \\
\end{tabular},
\end{equation}
where $h^{11}=\text{dim}_{\mathbb{C}}\,H^{1,1}(\mathcal{W}_{\varphi})$ is a positive integer. From Section \ref{LimitMHS}, there exists a complex $ \mathbf{Z}^{\bullet}_{\text{MS}}$ in the derived category $D^b(\textbf{MHS}_{\mathbb{Q}})$ whose cohomologies compute the limit MHS of the mirror family \ref{eq:pideltaW} at the large complex structure limit. While from last section, there exists a limit mixed motive $\mathcal{Z}_{\text{MS}} \in \textbf{DM}_{\text{gm}}(\mathbb{Q}, \mathbb{Q})$ constructed at the large complex structure limit, whose Hodge realisation is conjectured to be isomorphic to $\mathbf{Z}^{\bullet}_{\text{MS}}$. The following two easy propositions deal with the cases where $q \neq 3$.
\begin{proposition}
\begin{equation}
H^q(\mathbf{Z}^{\bullet}_{\text{MS}}) =0, ~~\text{when}~q \neq 0,2,3,4,6
\end{equation}
\end{proposition}
\begin{proof}
The Hodge diamond \ref{eq:HodgeDiamond} of a smooth fiber $\mathcal{W}_{\varphi}$ tells us
\begin{equation}
H^q(\mathcal{W}_{\varphi},\mathbb{Q})=0, ~~\text{when}~q \neq 0,2,3,4,6,
\end{equation} 
hence this proposition is an immediate result of Lemma \ref{SGAlemma}.
\end{proof}

\begin{proposition}
$H^q(\mathbf{Z}^{\bullet}_{\text{MS}})$ is a Hodge-Tate object when $q=0,2,4,6$, and we have
\begin{equation}
H^0(\mathbf{Z}^{\bullet}_{\text{MS}})=\mathbb{Q}(0),~H^2(\mathbf{Z}^{\bullet}_{\text{MS}})=\mathbb{Q}(-1)^{h^{11}},~ H^4(\mathbf{Z}^{\bullet}_{\text{MS}})= \mathbb{Q}(-2)^{h^{11}},~H^6(\mathbf{Z}^{\bullet}_{\text{MS}}) = \mathbb{Q}(-3).
\end{equation}
\end{proposition}
\begin{proof}
When $q=0$, the Hodge diamond \ref{eq:HodgeDiamond} tells us
\begin{equation}
H^0(\mathcal{W}_{\varphi},\mathbb{Q})=\mathbb{Q},
\end{equation}
and the pure Hodge structure on $H^0(\mathcal{W}_{\varphi},\mathbb{Q})$ is just $\mathbb{Q}(0)$. The subsheaf filtration $\mathcal{F}^p$ of $\mathcal{V}:=\text{R}^0 \,\pi_{\Delta^*,*} \mathbb{Z} \otimes \mathcal{O}_{\Delta^*}$ is given by
\begin{equation}
\mathcal{F}^0=\mathcal{V},~\mathcal{F}^1=0,
\end{equation}
therefore the limit Hodge filtration on $\widetilde{\mathcal{V}}\vert_0$ is given by
\begin{equation}
	F^0=\widetilde{\mathcal{V}}\vert_0,~F^1=0.
\end{equation} 
This limit Hodge filtration imposes very strong restriction on the possible limit MHS and in fact the only possibility is
\begin{equation}
H^0(\mathbf{Z}^{\bullet}_{\text{MS}}) = \mathbb{Q}(0).
\end{equation}
When $q=2$, the Hodge diamond \ref{eq:HodgeDiamond} tells us
\begin{equation}
H^{2,0}(\mathcal{W}_{\varphi},\mathbb{Q})=H^{0,2}(\mathcal{W}_{\varphi},\mathbb{Q})=0,
\end{equation} 
thus the pure Hodge structure on $H^2(\mathcal{W}_{\varphi},\mathbb{Q})$ is isomorphic to $\mathbb{Q}(-1)^{h^{11}}$. The subsheaf filtration $\mathcal{F}^p$ of $\mathcal{V}:=\text{R}^2 \,\pi_{\Delta^*,*} \mathbb{Z} \otimes \mathcal{O}_{\Delta^*}$ is given by
\begin{equation}
\mathcal{F}^0=\mathcal{F}^1=\mathcal{V},~\mathcal{F}^2=0,
\end{equation}
so the limit Hodge filtration on $\widetilde{\mathcal{V}}\vert_0$ is
\begin{equation}
F^0=F^1=\widetilde{\mathcal{V}}\vert_0,~F^2=0.
\end{equation} 
This limit Hodge filtration again imposes very strong restriction on the possible limit MHS and the only possibility is 
\begin{equation}
H^2(\mathbf{Z}^{\bullet}_{\text{MS}})= \mathbb{Q}(-1)^{h^{11}}.
\end{equation}
Similarly we also have
\begin{equation}
H^4(\mathbf{Z}^{\bullet}_{\text{MS}})= \mathbb{Q}(-2)^{h^{11}},~H^6(\mathbf{Z}^{\bullet}_{\text{MS}})= \mathbb{Q}(-3).
\end{equation}
\end{proof}

The rest of this section is devoted to the computation of $H^3(\mathbf{Z}^{\bullet}_{\text{MS}})$, which depends on the mirror symmetry conjecture in an essential way.

\subsection{The periods of the holomorphic three-form}

An essential ingredient in mirror symmetry is the nowhere-vanishing holomorphic three-form $\Omega$. Since in this paper we are only concerned with the mirror pairs where the deformation of the mirror threefold is rationally defined, hence we will assume that the three-form $\Omega$ is also rationally defined, i.e. it is a section of the bundle $\mathscr{F}^3_{\mathbb{Q}}$ associated to the rationally defined mirror family. We will further assume $\Omega$ has logarithmic poles along the smooth components of the fiber $Y$ over the large complex structure limit, hence it extends to a global section of $\widetilde{\mathscr{F}}^3_{\mathbb{Q}}$ with nonzero value at the large complex structure limit. 

Now, let us look at the periods of $\Omega$. Given an arbitrary point $\varphi_0 \in \Delta^*$, Poincar\'e duality implies the existence of a unimodular skew symmetric pairing on $H_3(\mathcal{W}_{\varphi_0},\mathbb{Z})$ (modulo torsions), which allows us to choose a symplectic basis $\{A_0,A_1,B_0,B_1\}$ that satisfy the following pairings \cite{MarkGross,PhilipXenia, PhilipXenia1}
\begin{equation}
A_a \cdot A_b=0,~~B_a \cdot B_b=0,~~A_a \cdot B_b= \delta_{ab}.
\end{equation}
Suppose the dual of this basis is given by $\{\alpha^0,\alpha^1,\beta^0,\beta^1\}$, i.e. the only non-vanishing pairings are
\begin{equation}
\alpha^a (A_b)=\delta^a_b, ~ \beta^a(B_b)=\,\delta^a_b,
\end{equation}
and this dual forms a basis of $H^3(\mathcal{W}_{\varphi_0},\mathbb{Z})$ (modulo torsions). 

\begin{remark}
The torsions of homology or cohomology groups will be ignored in this paper.
\end{remark}

For simplicity, we will also denote $B_0$, $B_1$ by $A_2$, $A_3$ and denote $\beta^0$, $\beta^1$ by $\alpha^2$, $\alpha^3$ respectively. Recall that the local system $V_{\mathbb{Z}}:=\text{R}^3 \,\pi_{\Delta^*,*} \mathbb{Z}$, with its dual denoted by $V_{\mathbb{Z}}^{\vee}$. In a simply connected local neighborhood of $\varphi_0$, $A_a$ extends to a local section $A_a(\varphi)$ of the local system $V_{\mathbb{Z}}^{\vee}$ and $\alpha^a$ extends to a local section $\alpha^a(\varphi)$ of the local system $V_{\mathbb{Z}}$. Notice that  $\{A_a(\varphi)\}_{a=0}^3$ form a basis of $H_3(\mathcal{W}_{\varphi},\mathbb{Z})$, while its dual $\{\alpha^a(\varphi)\}_{a=0}^3$ form a basis of $H^3(\mathcal{W}_{\varphi},\mathbb{Z})$ \cite{Zein}. Since the unimodular skew symmetric pairing is preserved by the extension, $\{A_a(\varphi)\}_{a=0}^3$ is actually a symplectic basis of $H_3(\mathcal{W}_{\varphi},\mathbb{Z})$. The integral periods of the three-form $\Omega$ with respect to the basis $\{A_a(\varphi)\}_{a=0}^3$ are defined by
\begin{equation} \label{eq:IntegralPeriodDefinition}
z_a(\varphi)= \int_{A_a(\varphi)} \Omega(\varphi),~\mathcal{G}_b(\varphi)=\int_{B_b(\varphi)} \Omega(\varphi),
\end{equation}
which are holomorphic (multi-valued) functions \cite{MarkGross,PhilipXenia, PhilipXenia1}. Define the period vector $\amalg(\varphi)$ by
\begin{equation}
\amalg(\varphi):=(\mathcal{G}_0(\varphi),\mathcal{G}_1(\varphi),z_0(\varphi),z_1(\varphi))^t,
\end{equation}
where $^t$ means transpose. Under the comparison isomorphism \ref{eq:FdeltaIso}, $\Omega(\varphi)$ has an expansion with respect to the basis $\{ \alpha^a(\varphi) \}_{a=0}^3$ given by
\begin{equation} \label{eq:expansionOmega}
I_{\text{B}}(\Omega(\varphi))=z_0(\varphi)\, \alpha^0(\varphi)+z_1(\varphi)\, \alpha^1(\varphi)+\mathcal{G}_0(\varphi) \, \beta^0(\varphi) +\mathcal{G}_1(\varphi) \,\beta^1(\varphi) 
\end{equation}
However the extension of the integral period to the punctured disc $\Delta^*$ is generally multi-valued, which is called the monodromy. Recall that the local system $V_{\mathbb{Z}}^\vee$ is uniquely determined by a representation $\Phi$ of the fundamental group of $\pi_1(\Delta^*,\varphi_0)$ \cite{Zein}
\begin{equation} \label{eq:Phirepn}
\Phi: \pi_1(\Delta^*,\varphi_0) \rightarrow \text{Aut}(H_3(\mathcal{W}_{\varphi_0},\mathbb{Z})).
\end{equation}
The fundamental group $\pi_1(\Delta^*,\varphi_0)$ is isomorphic to $\mathbb{Z}$ with a generator $T$, and the representation $\Phi$ is uniquely determined by the image of $\Phi(T)$. Since unimodular pairing is preserved by extension, the image of $\Phi$ lies in $\text{Sp}(4,\mathbb{Z})$ with respect to the basis $\{ A_a \}_{a=0}^3$ \cite{CoxKatz, PhilipXenia, PhilipXenia1}. Let the matrix of $\Phi(T)$ with respect to the basis $\{ A_a \}_{a=0}^3$ be $T_C \in \text{Sp}(4,\mathbb{Z})$, i.e.
\begin{equation}
\Phi(T).A_a=\sum_{b=0}^3\,A_b\,(T_C)_{b\,a}.
\end{equation}
The monodromy of the integral period vector $\amalg$ is given by
\begin{equation}
\amalg_a(\varphi_0)=\int_{A_a} \Omega(\varphi_0) \rightarrow \sum_b (T_C)_{ba} \int_{A_b} \Omega(\varphi_0) =\sum_b \,(T_C)_{ba} \,\amalg_b(\varphi_0).
\end{equation}
As in Section \ref{sectionDCE}, let the dual representation of $\Phi$ be
\begin{equation} \label{eq:PsiRepn}
\Psi: \pi_1(\Delta^*,\varphi_0)  \rightarrow \text{Aut}(H^3(\mathcal{W}_{\varphi_0},\mathbb{Z})),
\end{equation}
and the nilpotent operator $N=\log \Psi(T)$ is defined in formula \ref{eq:Ndefn}. We have assumed 0 is the large complex structure limit of the deformation of $W$, the meaning of which is given as below.

\begin{definition}
The point 0 is the large complex structure limit of the mirror family if the monodromy around it is maximally unipotent \cite{MarkGross}, i.e. 
\begin{equation} \label{eq:Nunipotent}
N^3 \neq 0,~N^4=0.
\end{equation}
\end{definition}

\subsection{The canonical periods of the three-form}
From Griffiths transversality, the three-form $\Omega$ satisfies a Picard-Fuchs equation of order 4
\begin{equation}
\mathcal{L}\,\Omega=0,
\end{equation}
where $\mathcal{L}$ is a differential operator with polynomial coefficients $R_i(\varphi)$ of the form \cite{MarkGross,CoxKatz,PhilipXenia}
\begin{equation}
\mathcal{L}=R_4(\varphi) \,\vartheta^4+R_3(\varphi)\, \vartheta^3 +R_2(\varphi)\, \vartheta^2+ R_1(\varphi) \, \vartheta^1+R_0(\varphi), ~\text{with}~ \vartheta=\varphi \,\frac{d}{d\varphi}.
\end{equation}
Therefore the integral periods $z_a(\varphi)$ and $\mathcal{G}_b(\varphi)$ are solutions of the differential equation \ref{eq:PFequation}. Moreover, the operator $\mathcal{L}$ has a regular singularity at 0 \cite{KatzRegularity, DeligneEquations}. From the definition of the large complex structure limit, the monodromy around 0 is maximally unipotent, hence the solution space of the Picard-Fuchs equation
\begin{equation} \label{eq:PFequation}
\mathcal{L}\,\varpi(\varphi)=0
\end{equation}
has a distinguished basis consists of four linearly independent solutions of the form 
\begin{equation} \label{eq:PeriodsCan}
\begin{aligned}
\varpi_0 &= f_0,  \\
\varpi_1 &=\frac{1}{2\pi i}\left(f_0 \log \varphi+f_1\right), \\
\varpi_2 &=\frac{1}{(2\pi i)^2}\left( f_0 \log^2 \varphi +2\, f_1 \log \varphi + f_2\right), \\
\varpi_3 &=\frac{1}{(2 \pi i)^3} \left( f_0 \log^3 \varphi +3 \, f_1 \log^2 \varphi +3\, f_2 \log \varphi +f_3 \right),
\end{aligned} 
\end{equation}
where $\{f_j\}_{j=0}^3$ are holomorphic functions on $\Delta$, so they admit power series in $\varphi$ that converge on $\Delta$ \cite{MarkGross,CoxKatz,PhilipXenia}. If we impose the conditions
\begin{equation} \label{eq:boundarycondition}
f_0(0)=1,~f_1(0)=f_2(0)=f_3(0)=0,
\end{equation}
the four solutions \ref{eq:PeriodsCan} will be uniquely determined, which are called the canonical periods of the three-form $\Omega$ \cite{PhilipXenia2}. The canonical period vector $\varpi$ is defined as
\begin{equation}
\varpi:=(\varpi_0,\,\varpi_1,\,\varpi_2,\,\varpi_3)^t.
\end{equation}
Since the integral periods $\{\amalg_a\}_{a=0}^3$ form another basis of the solution space of \ref{eq:PFequation}, there exists a matrix $S \in \text{GL}(4,\mathbb{C})$ such that
\begin{equation} \label{eq:PiSomega}
\amalg_a =\sum_{b=0}^3\,S_{a\,b}\, \varpi_b.
\end{equation}
The expansion \ref{eq:expansionOmega} now becomes
\begin{equation} \label{eq:OmegaExpansionChange}
I_{\text{B}}(\Omega(\varphi))=\sum_{a=0}^3\,\alpha^a(\varphi) \, \amalg_a(\varphi)=\sum_{a,b}\,\alpha^a(\varphi)\, S_{a\,b}\, \varpi_b(\varphi).
\end{equation}
Let $\{\gamma^a\}_{a=0}^3$ be a basis of $H^3(\mathcal{W}_{\varphi_0},\mathbb{C})$ defined by
\begin{equation} \label{eq:TransGammaAlpha}
\gamma^a=\sum_{b=0}^3\,\alpha^b\,S_{b\,a},
\end{equation}
then the expansion \ref{eq:OmegaExpansionChange} becomes
\begin{equation} \label{eq:OmegaExpansionBetti}
I_{\text{B}}(\Omega(\varphi))=\sum_{a=0}^3 \gamma^a(\varphi)\, \varpi_a(\varphi),
\end{equation}
where $\gamma^a(\varphi)$ is the extension of $\gamma^a$ in a local neighborhood of $\varphi_0$. Let the dual of $\{\gamma^a\}_{a=0}^3$ be $\{C_a\}_{a=0}^3$,  which forms a basis of $H_3(\mathcal{W}_{\varphi_0},\mathbb{C})$. Furthermore, the canonical period $\varpi_a$ is also the integration of $\Omega(\varphi)$ over $C_a(\varphi)$
\begin{equation}
\varpi_a(\varphi)=\int_{C_a(\varphi)} \Omega(\varphi),
\end{equation}
where $C_a(\varphi)$ is the extension of $C_a$ in a local neighborhood of $\varphi_0$. Let the action of $T$ on the basis $\{C_a\}_{a=0}^3$ be
\begin{equation}
\Phi(T)\,C_a =\sum_b \,(T_{\text{Can}})_{b\,a}\,C_b,
\end{equation}
where $(T_{\text{Can}})_{b\,a}$ is the matrix of $\Phi(T)$ with respect to $\{C_a\}_{a=0}^3$. While the action of $T$ on the period $\varpi_a$ is given by
\begin{equation}
\varpi_a(\varphi_0)=\int_{C_a} \Omega(\varphi_0) \rightarrow \sum_b (T_{\text{Can}})_{b\,a} \int_{C_b} \Omega(\varphi_0) =\sum_b (T_{\text{Can}})_{b\,a} \,\varpi_b(\varphi_0),
\end{equation}
which is the monodromy of $\varpi_a$. On the other hand, the monodromy of $\varpi_a$ is induced by the analytical continuation of $\log \varphi$ around 0, i.e. $\log \varphi \rightarrow \log \varphi +2 \pi i$, so $(T_{\text{Can}})$ is easily found to be
\begin{align}
T_{\text{Can}}=
\begin{pmatrix}
 1 & 1 & 1 & 1 \\
 0 & 1 & 2 & 3 \\
 0 & 0 & 1 & 3 \\
 0 & 0 & 0 & 1 \\
\end{pmatrix}.
\end{align}
Since $\Psi$ is the dual representation of $\Phi$, and $\{\gamma^a\}_{a=0}^3$ is the dual of $\{C_a\}_{a=0}^3$, we have
\begin{equation}
\Psi(T)\,\gamma^a=\sum_{b=0}^3\,(T^{\vee}_{\text{Can}})_{b\,a}\,\gamma^b,
\end{equation}
so $T^{\vee}_{\text{Can}}$ is the matrix of $\Psi(T)$ with respect to $\{\gamma^a\}_{a=0}^3$
\begin{align}
T^{\vee}_{\text{Can}}=( (T_{\text{Can}})^t)^{-1}=
\begin{pmatrix}
 1 & 0 & 0 & 0 \\
 -1 & 1 & 0 & 0 \\
 1 & -2 & 1 & 0 \\
 -1 & 3 & -3 & 1 \\
\end{pmatrix}.
\end{align}
But in order to find the monodromy of the integral period $\Pi$, we will need to know how to compute the matrix $S$ in the formula \ref{eq:TransGammaAlpha}, which depends on the mirror symmetry conjecture in an essential way.

\subsection{Mirror symmetry} \label{TranIntCan}

The formulation of the mirror symmetry conjecture in this section comes from \cite{MarkGross, PhilipXenia,CoxKatz,PhilipXenia1}. From \cite{BryantGriffiths, GriffithsI, GriffithsII, GriffithsIII}, the integral periods $(z_0(\varphi),z_1(\varphi))$ cannot both vanish simultaneously and locally they define a projective coordinate of the moduli space $\mathscr{M}_C(W)$, in terms of which $\mathcal{G}_a(z)$ is a homogeneous function of degree one. The three-form $\Omega$ is also expressed as 
\begin{equation}
I_{\text{B}}(\Omega(z))=z_0\, \alpha^0(z)+z_1\, \alpha^1(z)+\mathcal{G}_0(z) \, \beta^0(z) +\mathcal{G}_1(z) \,\beta^1(z).
\end{equation}
The Griffiths transversality implies
\begin{equation}
\int_W\,\Omega(z) \wedge \frac{\partial\,\Omega(z)}{\partial\,z_a}=0,
\end{equation}
which yields the following relation
\begin{equation}
\mathcal{G}_a(z)=\frac{\partial  \mathcal{G}(z)}{\partial z_a},~\text{where}~\mathcal{G}(z):=\frac{1}{2}\,\sum_b\,z_b\,\mathcal{G}_b(z)
\end{equation}
The homogenous function $\mathcal{G}$ of degree two will be called the prepotential, and the Yukawa coupling $\kappa_{abc}$ is given by
\begin{equation}
\kappa_{abc}=\int_W\,\Omega(z) \wedge \frac{\partial^3\,\Omega(z)}{\partial z_a\partial z_b\partial z_c}=-\frac{\partial^3\,\mathcal{G}(z)}{\partial z_a\partial z_b\partial z_c}.
\end{equation} 
In all examples of one-parameter mirror pairs, there exists an integral symplectic basis $(A_0,A_1,B_0,B_1)$ of $H^3(\mathcal{W}_{\varphi_0},\mathbb{Z})$ such that 
\begin{equation} \label{eq:zivarpii01}
z_i(\varphi)=\lambda \, \varpi_i(\varphi),~i=0,1,
\end{equation}
where $\lambda$ is a nonzero constant. Notice that the canonical periods have been chosen to satisfy $\varpi_0(0)=1$ and $f_1(0)=0$ in \ref{eq:PeriodsCan}, hence the existence of such a basis is very critical in the mirror symmetry. Let us denote the quotient $\varpi_1/\varpi_0$ by $t'$
\begin{equation}
t' =\frac{z_1}{z_0}=\frac{\varpi_1}{\varpi_0}=\frac{1}{2 \pi i}\,\log \varphi+\frac{f_1(\varphi)}{f_0(\varphi)},
\end{equation}
and under the action of monodromy, it transforms in the way
\begin{equation}
t' \rightarrow t' +1.
\end{equation} 
\begin{definition} \label{Mirrormap}
The \textbf{mirror map} of the mirror pair $(M,W)$ maps a neighbourhood of the large complex structure limit to the K\"ahler moduli space $\mathscr{M}_K(M)$ via
$$\varphi\mapsto t'=\frac{\varpi_1}{\varpi_0}\in \mathscr{M}_K(M).$$\end{definition}

Now rescale the vector $\Pi$ by a factor $z_0$, and its normalisation is denoted by $\amalg_{\text{A}}$
\begin{equation}
\amalg_{\text{A}}=(\mathcal{G}_0/z_0,\mathcal{G}_1/z_0,\,1,\,z_1/z_0)^t,
\end{equation} 
which in terms of the affine coordinate $t'$, is expressed as
\begin{equation}
\amalg_{\text{A}}=( 2 \,\mathcal{G}-t' \,\frac{\partial  \mathcal{G}}{\partial t'},\,\frac{\partial \mathcal{G}}{\partial t'},\,1,\,t')^t.
\end{equation} 
While on the K\"ahler side, the mirror period vector $\Pi$ is defined by \cite{PhilipXenia, PhilipXenia1}
\begin{align}
\Pi =(\mathcal{F}_0,\mathcal{F}_1, 1,t)^t,~\text{with}~\mathcal{F}_0= 2 \,\mathcal{F}-t \,\frac{\partial  \mathcal{F}}{\partial t},~ \mathcal{F}_1=\frac{\partial \mathcal{F}}{\partial t}.
\end{align}
Since the prepotential $\mathcal{F}$ admits an expansion of the form \ref{eq:Prepotential}, $\Pi$ becomes
\begin{align} \label{eq:Piperiod}
\Pi =
\begin{pmatrix}
\frac{1}{6} \,Y_{111}\, t^3 -\frac{1}{2}\, Y_{001} \,t-\frac{1}{3}\, Y_{000} +2\, \mathcal{F}^{\text{np}}(t)-t \,\frac{d\mathcal{F}^{\text{np}}(t)}{dt}\\
-\frac{1}{2}\,Y_{111} \,t^2 - Y_{011}\,t-\frac{1}{2}\, Y_{001} +\frac{d\mathcal{F}^{\text{np}}(t)}{dt}\\
1 \\
t \\
\end{pmatrix}.
\end{align}

The formulation of the mirror symmetry conjecture in this paper follows from \cite{PhilipXenia, PhilipXenia1}.

\textbf{Mirror Symmetry Conjecture} 
\textit{ 
Given a mirror pair $(M,W)$ of Calabi-Yau threefolds, there exists an integral symplectic basis $(A_0,A_1,B_0,B_1)$ of $H^3(\mathcal{W}_{\varphi_0},\mathbb{Z})$ such that the formula \ref{eq:zivarpii01} is satisfied for a non-zero constant $\lambda$. The mirror map in Definition \ref{Mirrormap} induces an isomorphism between $\mathscr{M}_K(M)$ and a neighbourhood of the large complex structure limit of $\mathscr{M}_C(W)$, under which the normalised integral period vector $\amalg_{\text{A}}$ of $\mathscr{M}_C(W)$ is identified with the mirror period vector $\Pi$ of $\mathscr{M}_K(M)$. In particular, the prepotential $\mathcal{G}$ on the complex side is identified with the prepotential $\mathcal{F}$ on the K\"ahler side under the mirror map.
}

\begin{remark} \label{Y000remark}
On the complex side, the definitions of the integral periods $\amalg_a$ and the prepotential $\mathcal{G}$ depend on the choice of an integral symplectic basis $\{A_0,A_1,B_0,B_1\}$ of $H_3(\mathcal{W}_{\varphi_0},\mathbb{Z})$. However on the K\"ahler side, there is no natural integral symplectic structure on it, and the identification of $\amalg_{\text{A}}$ with $\Pi$ under the mirror map transfers the integral symplectic structure on the complex side to the K\"ahler side \cite{PhilipXenia}. The exact values of the coefficients $Y_{011}$, $Y_{001}$ and $Y_{000}$ depend on the choice of such an integral symplectic structurem and for a different choice, these coefficients transform in the following way \cite{PhilipXenia1}
\begin{equation}
\begin{aligned}
Y_{011} &\rightarrow Y_{011} +n, \, n \in \mathbb{Z}, \\
Y_{001} &\rightarrow Y_{001} +k, \, k \in \mathbb{Z}, \\
Y_{000} &\rightarrow Y_{000} +r', \,r' \in \mathbb{Q}.
\end{aligned}
\end{equation}
\end{remark}

From now on, the notation $t'$ will be identified with $t$, and the differences between them will be ignored. The monodromy of the mirror period $\Pi_a$ around $t=i \,\infty$ is induced by the action $t \rightarrow t+1$. Let the monodromy matrix of $\Pi_a$ be $T_K$, i.e. under the monodromy action $\Pi_a$ transforms in the way
\begin{equation}
\Pi_a \rightarrow \sum_{b=0}^3\,(T_K)_{b\,a} \,\Pi_b.
\end{equation}
Since $\mathcal{F}^{\text{np}}(t)$ admits a series expansion in $\exp 2 \pi i \,t$, its derivative $d\,\mathcal{F}^{\text{np}}(t)/dt$ is also invariant under the monodromy action $t \rightarrow t+1$. Thus the form of $\Pi_a$ in the formula \ref{eq:Piperiod} immediately tells us \cite{PhilipXenia2}
\begin{equation}
T_K=
\left(
\begin{array}{cccc}
 1 & 0 & 0 & 0 \\
 -1 & 1 & 0 & 0 \\
 \frac{1}{6} \, Y_{111}-Y_{001} & -\frac{1}{2} \, Y_{111}-Y_{011} & 1 & 1 \\
 \frac{1}{2} \, Y_{111}-Y_{011} & -Y_{111} & 0 & 1 \\
\end{array}
\right).
\end{equation}

The identification of $\Pi$ with $\amalg_{\text{A}}$ under the mirror map shows that the monodromy of $\Pi$ is the same as the monodromy of $\amalg$, i.e. 
\begin{equation} \label{eq:MonodromyOfIntegral}
T_C=T_K
\end{equation}
where we have used the fact that $z_0= \lambda \varpi_0$ is a holomorphic function in a neighbourhood of the large complex structure limit \cite{PhilipXenia1,PhilipXenia2}. Since the matrix $T_C$ lies in $\text{Sp}(4,\mathbb{Z})$, therefore $T_K$ is also an integral symplectic matrix, which immediately yields the following well-known fact in mirror symmetry.
\begin{corollary}
The numbers $2\,Y_{011}$ and $6\,Y_{001}$ are both integers, a priori $Y_{011}$ and $Y_{001}$ are rational numbers.
\end{corollary}
Now we are ready to compute the matrix $S$ in the formula \ref{eq:PiSomega}. Near the large complex structure limit,  formula \ref{eq:boundarycondition} shows 
\begin{equation}
t=\frac{1}{2 \pi i} \,\log \varphi+ \mathcal{O}(\varphi),
\end{equation}
therefore the large complex structure limit on the complex side corresponds to $t= i\, \infty$ on the K\"ahler side \cite{PhilipXenia, PhilipXenia2}. In the limit $t \rightarrow i\, \infty$, the leading parts of $\amalg_{\text{A}}$ and  $\varpi$ are given by
\begin{align} \label{eq:Pilimit}
\amalg_{\text{A}} \equiv \Pi \sim
\begin{pmatrix}
\frac{1}{6}\, Y_{111}\, t^3 -\frac{1}{2}\, Y_{001}\, t-\frac{1}{3}\, Y_{000} \\
-\frac{1}{2}\,Y_{111} \,t^2 - Y_{011}\,t-\frac{1}{2} \,Y_{001} \\
1 \\
t \\
\end{pmatrix}
,~ \varpi \sim
\begin{pmatrix}
1 \\
t \\
t^2 \\
t^3 \\
\end{pmatrix},
\end{align}
from which the matrix $S$ can be easily evaluated \cite{PhilipXenia2}
\begin{equation}
S\,=\lambda \,
\left(
\begin{array}{cccc}
 -\frac{1}{3}\, Y_{000} & -\frac{1}{2} \,Y_{001} & 0 & \frac{1}{6}\, Y_{111} \\
 -\frac{1}{2} \,Y_{001} & -\,Y_{011} & -\frac{1}{2} \,Y_{111} & 0 \\
 1 & 0 & 0 & 0 \\
 0 & 1 & 0 & 0 \\
\end{array}
\right).
\end{equation}
Here $\lambda$ is the constant in the formula \ref{eq:zivarpii01}. Finally we are ready to compute $H^3(\mathbf{Z}^{\bullet}_{\text{MS}})$ using the methods in \cite{Schmid,MarkGross,Hain}. 
 
\subsection{The weight filtration}

The monodromy operator $N$ in the formula \ref{eq:Ndefn} defines a weight filtration on the rational vector space $H^3(\mathcal{W}_{\varphi_0},\mathbb{Q})$, which induces a weight filtration $W_*$ on $\widetilde{\mathcal{V}}\vert_{0,\mathbb{Q}}$ under the isomorphism in the formula \ref{eq:Newrho} \cite{PetersSteenbrink, Schmid, MarkGross}. To compute the weight filtration on $H^3(\mathcal{W}_{\varphi_0},\mathbb{Q})$ defined by $N$, it is more convenient to choose a new basis $\{\beta^a\}_{a=0}^3$. From section \ref{TranIntCan}, $Y_{111}$, $Y_{011}$ and $Y_{001}$ are all rational numbers, and let $S_1$ be the matrix
\begin{align}
S_1=
\begin{pmatrix}
0 & \frac{1}{2}\, Y_{001} & 0 & -Y_{111} \\
-\frac{1}{2}\, Y_{001} & Y_{011} & -Y_{111} & 0 \\
1 & 0 & 0 & 0 \\
0 & -1 & 0 & 0 \\
\end{pmatrix}
\end{align}
with determinant $Y_{111}^2$, therefore it lies in $\text{GL}(4,\mathbb{Q})$. Now define the new basis $\{\beta^a\}_{a=0}^3$ by
\begin{equation} \label{eq:betaalphas1}
\beta^a =\sum_{b=0}^3\,(S_1)_{b\,a}\,\alpha^b,
\end{equation}
Moreover we have
\begin{equation} \label{eq:gammabetaS2}
\gamma^a=\sum_{b=0}^3\,(S_2)_{b\,a}\,\beta^b,
\end{equation}
where the matrix $S_2$ is given by
\begin{align}
S_2
=\lambda\,
\begin{pmatrix}
1 & 0 & 0 & 0 \\
 0 & -1 & 0 & 0 \\
 0 & 0 & \frac{1}{2} & 0 \\
 \frac{Y_{000}}{3\, Y_{111}} & 0 & 0 & -\frac{1}{6} \\
\end{pmatrix}.
\end{align}

Recall that $\{\alpha^a\}_{a=0}^3$ is the dual basis of $\{A_a\}_{a=0}^3$, hence the matrix of $\Psi(T)$ with respect to $\{\alpha^a\}_{a=0}^3$ is given by $((T_C)^t)^{-1}$, which is equal to $((T_K)^t)^{-1}$ from the formula  \ref{eq:MonodromyOfIntegral}. From the formulas \ref{eq:Ndefn} and \ref{eq:betaalphas1}, the action of $N$ on $\{\beta^a \}_{a=0}^3$ is
\begin{equation}
N \beta^0 =\beta^1, N \beta^1=\beta^2, N \beta^2=\beta^3, N \beta^3=0.
\end{equation}
The weight filtration on $H^3(\mathcal{W}_{\varphi_0},\mathbb{Q})$ defined by $N$ can be computed inductively \cite{MarkGross,Schmid,Hain}. The first step in the induction process is to let $W_{-1}$ and $W_6$ be
\begin{equation}
W_{-1}H^3(\mathcal{W}_{\varphi_0},\mathbb{Q})=0,~
W_6 H^3(\mathcal{W}_{\varphi_0},\mathbb{Q})=H^3(\mathcal{W}_{\varphi_0},\mathbb{Q}).
\end{equation}
Then $W_0$ and $W_5$ are given by
\begin{equation}
\begin{aligned}
W_0 H^3(\mathcal{W}_{\varphi_0},\mathbb{Q})&= \text{im}\,\, N^3=\mathbb{Q}\, \beta^3, \\
W_5 H^3(\mathcal{W}_{\varphi_0},\mathbb{Q})&=\text{ker}\, N^3=\mathbb{Q}\,\beta^1+\mathbb{Q}\,\beta^2+\mathbb{Q}\,\beta^3,
\end{aligned}
\end{equation}
and the quotient $W_5/W_0$ is
\begin{equation}
W_5 H^3(\mathcal{W}_{\varphi_0},\mathbb{Q})/W_0 H^3(\mathcal{W}_{\varphi_0},\mathbb{Q}) \simeq \mathbb{Q}\,\beta^1+\mathbb{Q}\,\beta^2.
\end{equation}
Since the operator $N^2$ defines a zero map on the quotient $W_5/W_0$,  we have 
\begin{equation}
W_1H^3(\mathcal{W}_{\varphi_0},\mathbb{Q})=W_0 H^3(\mathcal{W}_{\varphi_0},\mathbb{Q}), ~
W_4H^3(\mathcal{W}_{\varphi_0},\mathbb{Q})=W_5 H^3(\mathcal{W}_{\varphi_0},\mathbb{Q}).
\end{equation}
and the quotient $W_4/W_1$ is 
\begin{equation}
 W_4 H^3(\mathcal{W}_{\varphi_0},\mathbb{Q})/W_1 H^3(\mathcal{W}_{\varphi_0},\mathbb{Q}) \simeq \mathbb{Q}\,\beta^1+\mathbb{Q}\,\beta^2.
\end{equation}
Now $N$ defines a map on $W_4/W_1$ such that
\begin{equation}
N:\beta^1 \mapsto \beta^2,~\beta^2 \mapsto 0,
\end{equation}
so we have 
\begin{equation}
W_2H^3(\mathcal{W}_{\varphi_0},\mathbb{Q})=W_3H^3(\mathcal{W}_{\varphi_0},\mathbb{Q})=\mathbb{Q}\,\beta^2+\mathbb{Q}\,\beta^3.
\end{equation}
Inductively, we have found the weight filtration on $H^3(\mathcal{W}_{\varphi_0},\mathbb{Q})$ defined by $N$. The isomorphism $\rho_{\varphi_0}$ in the formula \ref{eq:Newrho} sends the basis $\{\beta^a\}_{a=0}^3$ of $H^3(\mathcal{W}_{\varphi_0},\mathbb{Q})$ to the basis $\{\widetilde{\beta}^a(0)\}_{a=0}^3$ of $\widetilde{\mathcal{V}}\vert_{0,\mathbb{Q}}$, similarly we also have
\begin{equation}
\widetilde{\beta}^a(0) =\sum_{b=0}^3\,(S_1)_{b\,a}\,\widetilde{\alpha}^b(0),~~
\widetilde{\gamma}^a(0) =\sum_{b=0}^3\,(S_2)_{b\,a}\,\widetilde{\beta}^b(0).
\end{equation}
Finally, we obtain the weight filtration $W_*(\widetilde{\mathcal{V}}\vert_{0,\mathbb{Q}})$ in the formula \ref{eq:NewLMHS} from $W_*H^3(\mathcal{W}_{\varphi_0},\mathbb{Q})$ through the isomorphism $\rho_{\varphi_0}$  in the formula \ref{eq:Newrho} 
\begin{equation} \label{eq:WeightFiltrationBeta}
\begin{aligned}
W_0 (\widetilde{\mathcal{V}}\vert_{0,\mathbb{Q}}) &=W_1( \widetilde{\mathcal{V}}\vert_{0,\mathbb{Q}})=\mathbb{Q}\,\widetilde{\beta}^3(0), \\
W_2 (\widetilde{\mathcal{V}}\vert_{0,\mathbb{Q}}) &=W_3( \widetilde{\mathcal{V}}\vert_{0,\mathbb{Q}})=\mathbb{Q}\,\widetilde{\beta}^2(0)+\mathbb{Q}\, \widetilde{\beta}^3(0), \\
W_4 (\widetilde{\mathcal{V}}\vert_{0,\mathbb{Q}}) &= W_5( \widetilde{\mathcal{V}}\vert_{0,\mathbb{Q}})=\mathbb{Q}\,\widetilde{\beta}^1(0)+\mathbb{Q}\,\widetilde{\beta}^2(0)+\mathbb{Q}\, \widetilde{\beta}^3(0), \\
W_6 (\widetilde{\mathcal{V}}\vert_{0,\mathbb{Q}}) &=\mathbb{Q}\,\widetilde{\beta}^0(0)+\mathbb{Q}\, \widetilde{\beta}^1(0)+\mathbb{Q}\, \widetilde{\beta}^2(0)+\mathbb{Q}\, \widetilde{\beta}^3(0).
\end{aligned}
\end{equation}

\subsection{The limit Hodge filtration }

Now we will compute the limit Hodge filtration. The fiber $ \widetilde{\mathscr{F}}^p_{\mathbb{Q}}\vert_0$ defines a decreasing filtration on $\widetilde{\mathscr{V}}_{\mathbb{Q}}\vert_0$, which gives us the limit Hodge filtration on $\widetilde{\mathcal{V}}\vert_0$ under the comparison isomorphism \cite{PetersSteenbrink, SteenbrinkLMHS, SteenbrinkVMHS, Schmid, Hain}
\begin{equation} \label{eq:IBHodge}
F^p(\widetilde{\mathcal{V}}\vert_0) = I_{\text{B}}(F^p(\widetilde{\mathscr{V}}^{\text{an}}_{\mathbb{C}}\vert_0)) = I_{\text{B}}(\widetilde{\mathscr{F}}^p_{\mathbb{Q}}\vert_0) \otimes_{\mathbb{Q}} \mathbb{C},
\end{equation}
where we have used the following canonical isomorphisms
\begin{equation}
F^p(\widetilde{\mathscr{V}}^{\text{an}}_{\mathbb{C}}\vert_0) = F^p(\widetilde{\mathscr{V}}_{\mathbb{C}}\vert_0)= F^p (\widetilde{\mathscr{V}}_{\mathbb{Q}}\vert_0) \otimes_{\mathbb{Q}} \mathbb{C} = \widetilde{\mathscr{F}}^p_{\mathbb{Q}}\vert_0 \otimes_{\mathbb{Q}} \mathbb{C}.
\end{equation}
From the assumption that the three-form $\Omega$ has logarithmic poles along the smooth components of the singular fiber $Y$, it extends to a global section of $\widetilde{\mathscr{F}}^3_{\mathbb{Q}}$ which implies
\begin{equation}
F^3(\widetilde{\mathscr{V}}_{\mathbb{Q}}\vert_0) \supset \mathbb{Q}\,\Omega\vert_0.
\end{equation}
Shrink the curve $S$ to an open affine subset if necessary, we will assume its tangent sheaf has a section of the form
\begin{equation}
\vartheta:=\varphi \,d/d \varphi.
\end{equation}
From Section \ref{sectionDCE}, the Gauss-Manin connection $\nabla_{\mathbb{Q}}$ of $\mathscr{V}_{\mathbb{Q}}$ canonically extends to a connection $\widetilde{\nabla}_{\mathbb{Q}}$ of $\widetilde{\mathscr{V}}_{\mathbb{Q}}$ that has a logarithmic pole along the point 0. From Griffiths transversality, $\widetilde{\nabla}_{\mathbb{Q},\vartheta}\,\Omega$ is a section of $\widetilde{\mathscr{F}}^2_{\mathbb{Q}} $, hence we find \cite{SteenbrinkVMHS,DeligneMirror}
\begin{equation}
F^2(\widetilde{\mathscr{V}}_{\mathbb{Q}}\vert_0) \supset \mathbb{Q}\,\Omega \vert_{0}+\mathbb{Q}\, (\widetilde{\nabla}_{\mathbb{Q},\vartheta}\,\Omega)\vert_{0}.
\end{equation}
\begin{remark}
Notice that the additional $\varphi$ in the definition of $\vartheta$ is to clear the logarithmic pole of $\widetilde{\nabla}_{\mathbb{Q}}$ at $\varphi=0$.
\end{remark}
Similarly, $\widetilde{\nabla}^2_{\mathbb{Q},\vartheta}\,\Omega$ is a  section of $\widetilde{\mathscr{F}}^1_{\mathbb{Q}} $ and $\widetilde{\nabla}^3_{\mathbb{Q},\vartheta}\,\Omega$ is a section of $\widetilde{\mathscr{F}}^0_{\mathbb{Q}}$, which shows
\begin{equation}
\begin{aligned}
F^1(\widetilde{\mathscr{V}}_{\mathbb{Q}}\vert_{0}) &\supset \mathbb{Q}\,\Omega\vert_{0}+\mathbb{Q}\,(\widetilde{\nabla}_{\mathbb{Q},\vartheta}\,\Omega)\vert_{0}+\mathbb{Q}\,(\widetilde{\nabla}^2_{\mathbb{Q},\vartheta}\,\Omega)\vert_{0},\\
F^0(\widetilde{\mathscr{V}}_{\mathbb{Q}}\vert_{0})&\supset \mathbb{Q}\, \Omega\vert_{0}+\mathbb{Q}\,(\widetilde{\nabla}_{\mathbb{Q},\vartheta}\,\Omega)\vert_{0}+\mathbb{Q}\,(\widetilde{\nabla}^2_{\mathbb{Q},\vartheta}\,\Omega)\vert_{0}+\mathbb{Q}\,(\widetilde{\nabla}^3_{\mathbb{Q},\vartheta}\,\Omega)\vert_{0}.
\end{aligned}
\end{equation}
Under the comparison isomorphism \ref{eq:VtildeIsoToAna}, the restriction of $\Omega$ to $\Delta$ gives us a section of $\widetilde{\mathcal{V}}$. With respect to the regularised frame $\{\widehat{\gamma}^a(\varphi)\}_{a=0}^3$ of $\mathcal{V}$, $I_{\text{B}}(\Omega \vert_{\Delta^*})$ has an expansion of the form
\begin{equation}
\begin{aligned}
I_{\text{B}}(\Omega \vert_{\Delta^*})\vert_{\varphi}&=\sum_{a=0}^3\,\gamma^a(\varphi)\, \varpi_a(\varphi) \\
               &=\sum_{a,b,c}\,\gamma^a(\varphi)\,\big(\exp (-\frac{\log \varphi}{2\pi i}N)\big)_{ab}\,\big(\exp (\frac{\log \varphi}{2 \pi i}N)\big)_{bc} \, \varpi_c(\varphi) \\
               &=\sum_{a,b}\,\widehat{\gamma}^a(\varphi)\big(\exp (\frac{\log \varphi}{2 \pi i}N)\big)_{ab}\, \varpi_b(\varphi),
\end{aligned}
\end{equation}
where $\big(\exp (-\frac{\log \varphi}{2\,\pi\,i}N)\big)_{ab}$ is the matrix of the operator $\exp (-\frac{\log \varphi}{2\pi i}N)$ with respect to the basis $\{\gamma^a\}_{a=0}^3$. From this expansion we find that
\begin{equation}
I_{\text{B}}(\Omega \vert_{\Delta}) \vert_0=\sum_{a,b}\,\lim_{\varphi \rightarrow 0} \widehat{\gamma}^a(\varphi)\big(\exp (\frac{\log \varphi}{2 \pi i}N)\big)_{ab}\, \varpi_b(\varphi) =\widetilde{\gamma}^0(0).
\end{equation}
In order to compute $I_{\text{B}}(\widetilde{\nabla}^p_{\mathbb{Q},\vartheta}\, \Omega\vert_{\Delta})\vert_0$, we will need the following equation \cite{CoxKatz}
\begin{equation}
I_{\text{B}}(\nabla^p_{\mathbb{Q},\vartheta}\, \Omega \vert_{\Delta^*})=\sum_{a=0}^3\gamma^a(\varphi) \int_{C_a(\varphi)} \nabla^p_{\mathbb{Q},\vartheta} \,\Omega \vert_{\Delta^*}=\sum_{a=0}^3\gamma^a(\varphi)\,\vartheta^j\, \varpi_a(\varphi),
\end{equation}
from which we obtain
\begin{equation}
\begin{aligned}
I_{\text{B}}(\widetilde{\nabla}^1_{\mathbb{Q},\vartheta}\, \Omega\vert_{\Delta})\vert_0=& \sum_{a,b}\, \lim_{\varphi \rightarrow 0}  \,\widehat{\gamma}^a(\varphi)\big(\exp (\frac{\log \varphi}{2\,\pi\,i}N)\big)_{ab}\, \vartheta\, \varpi_b(\varphi)  =\frac{1}{(2 \pi i)}\,\widetilde{\gamma}^1(0), \\
I_{\text{B}}(\widetilde{\nabla}^2_{\mathbb{Q},\vartheta}\, \Omega \vert_{\Delta}) \vert_0=& \sum_{a,b}\, \lim_{\varphi \rightarrow 0} \widehat{\gamma}^a(\varphi)\big(\exp (\frac{\log \varphi}{2\,\pi\,i}N)\big)_{ab}\, \vartheta^2\, \varpi_b(\varphi)  =\frac{2}{(2 \pi i)^2}\widetilde{\gamma}^2(0), \\
I_{\text{B}}(\widetilde{\nabla}^3_{\mathbb{Q},\vartheta}\, \Omega \vert_{\Delta}) \vert_0=&\sum_{a,b}\, \lim_{\varphi \rightarrow 0}  \widehat{\gamma}^a(\varphi)\big(\exp (\frac{\log \varphi}{2\,\pi\,i}N)\big)_{ab}\, \vartheta^3\, \varpi_b(\varphi) = \frac{6}{(2 \pi i)^3} \widetilde{\gamma}^3(0).
\end{aligned}
\end{equation}
Therefore $\{I_{\text{B}}(\widetilde{\nabla}^p_{\mathbb{Q},\vartheta}\, \Omega\vert_{\Delta})\vert_0 \}_{p=0}^3$ are linearly independent, which immediately implies
\begin{equation} \label{eq:HodgeFiltrationGamma}
\begin{aligned}
I_{\text{B}}(F^3( \widetilde{\mathscr{V}}_{\mathbb{Q}}\vert_0 ))&= \mathbb{Q} \, \widetilde{\gamma}^0(0), \\
I_{\text{B}}(F^2( \widetilde{\mathscr{V}}_{\mathbb{Q}}\vert_0 ))&= \mathbb{Q}\, \widetilde{\gamma}^0(0)+\mathbb{Q}\, \frac{1}{2 \pi i}\widetilde{\gamma}^1(0),\\
I_{\text{B}}(F^1( \widetilde{\mathscr{V}}_{\mathbb{Q}}\vert_0 ))&= \mathbb{Q}\, \widetilde{\gamma}^0(0)+\mathbb{Q}\,\frac{1}{2 \pi i}\widetilde{\gamma}^1(0)+\mathbb{Q}\,\frac{1}{(2 \pi i)^2}\widetilde{\gamma}^2(0), \\
I_{\text{B}}(F^0( \widetilde{\mathscr{V}}_{\mathbb{Q}}\vert_0 ))&= \mathbb{Q}\, \widetilde{\gamma}^0(0)+\mathbb{Q}\,\frac{1}{2 \pi i}\widetilde{\gamma}^1(0)+\mathbb{Q}\,\frac{1}{(2 \pi i)^2}\widetilde{\gamma}^2(0)+\mathbb{Q}\,\frac{1}{(2 \pi i)^3} \widetilde{\gamma}^3(0).
\end{aligned}
\end{equation}
While from the formula \ref{eq:IBHodge}, the limit Hodge filtration on $\widetilde{\mathcal{V}}\vert_0$ is given by the complexification of \ref{eq:HodgeFiltrationGamma}. Thus we have found the limit MHS $H^3(\mathbf{Z}^{\bullet}_{\text{MS}})$ of the mirror family at the large complex structure limit, and now we are ready to study it more carefully.

\section{The \texorpdfstring{$\zeta(3)$}{z(3)}  in the prepotential and the motivic conjectures}

In this section we will show that the limit MHS $H^3(\mathbf{Z}^{\bullet}_{\text{MS}})$ at the large complex structure limit splits into the direct sum $ \mathbb{Q}(-1) \oplus \mathbb{Q} (-2) \oplus \mathbf{M}$, where $\mathbf{M}$ is an extension of $\mathbb{Q}(-3)$ by $\mathbb{Q}(0)$. By studying the mixed Hodge-Tate structure $\mathbf{M}$ carefully, we will show how the $\zeta(3)$ in the prepotential $\mathcal{F}$ (formula \ref{eq:Prepotential}) is  connected to the Conjecture \textbf{GHP}. In particular, this section will reveal the motivic nature of the $\zeta(3)$ in the prepotential.

\subsection{The splitting of the limit MHS}

For simplicity, let $\{x^j\}_{j=0}^3$ be a new basis of $\widetilde{\mathcal{V}}\vert_0$ defined by
\begin{equation}
x^j:=(2 \pi i)^{3-j} \,\widetilde{\beta}^j(0),~j=0,1,2,3.
\end{equation}
Now the rational vector space $\widetilde{\mathcal{V}}\vert_{0,\mathbb{Q}}$ is spanned by $\{(2 \pi i)^{j-3}\,x^j\}_{j=0}^3$, and the weight filtration $W_*(\widetilde{\mathcal{V}}\vert_{0,\mathbb{Q}})$ in the formula  \ref{eq:WeightFiltrationBeta} becomes
\begin{equation} \label{eq:LimitWeightFiltration}
\begin{aligned}
&W_0 (\widetilde{\mathcal{V}}\vert_{0,\mathbb{Q}}) = W_1(\widetilde{\mathcal{V}}\vert_{0,\mathbb{Q}})=\mathbb{Q}\,x^3, \\
&W_2(\widetilde{\mathcal{V}}\vert_{0,\mathbb{Q}}) =W_3(\widetilde{\mathcal{V}}\vert_{0,\mathbb{Q}})=\mathbb{Q}\,\frac{1}{(2 \pi i)}\,x^2+\mathbb{Q}\, x^3, \\
&W_4(\widetilde{\mathcal{V}}\vert_{0,\mathbb{Q}}) =W_5(\widetilde{\mathcal{V}}\vert_{0,\mathbb{Q}})=\mathbb{Q}\,\frac{1}{(2 \pi i)^2}\,x^1+\mathbb{Q}\,\frac{1}{(2 \pi i)}\,x^2+\mathbb{Q}\, x^3, \\
& W_6(\widetilde{\mathcal{V}}\vert_{0,\mathbb{Q}}) =\mathbb{Q}\,\frac{1}{(2 \pi i)^3}x^0+\mathbb{Q}\, \frac{1}{(2 \pi i)^2}\,x^1+\mathbb{Q}\,\frac{1}{(2 \pi i)}\,x^2+\mathbb{Q}\, x^3. \\
\end{aligned}
\end{equation}
The basis $\{\widetilde{\gamma}^a(0)\}_{a=0}^3$ in the formula \ref{eq:gammabetaS2} are also expressed as
\begin{equation}
\widetilde{\gamma}^0(0) =\frac{\lambda}{(2 \pi i)^3}\,x^0 +\frac{\lambda \,Y_{000}}{3\,Y_{111}} x^3, \\
\widetilde{\gamma}^1(0) =-\,\frac{\lambda}{(2 \pi i)^2}\,x^1, \\
\widetilde{\gamma}^2(0) =\frac{\lambda}{2(2 \pi i)}\,x^2, \\
\widetilde{\gamma}^3(0) =-\,\frac{\lambda}{6}\, x^3,
\end{equation}
and from the formula \ref{eq:HodgeFiltrationGamma}, the Hodge filtration $F^*(\widetilde{\mathcal{V}}\vert_0)$ is given by
\begin{equation} \label{eq:LimitHodgeFiltration}
\begin{aligned}
&F^3( \widetilde{\mathcal{V}}\vert_0) =\frac{\lambda}{(2 \pi i)^3} \mathbb{Q}\, \text{span} \{x^0 +\frac{(2 \pi i)^3\, Y_{000}}{3\,Y_{111}}\, x^3\} \otimes_{\mathbb{Q}} \mathbb{C},\\
&F^2( \widetilde{\mathcal{V}}\vert_0) =\frac{\lambda}{(2 \pi i)^3} \mathbb{Q}\, \text{span} \{x^0 +\frac{(2 \pi i)^3\, Y_{000}}{3\,Y_{111}}\, x^3, x^1\} \otimes_{\mathbb{Q}} \mathbb{C}, \\
&F^1( \widetilde{\mathcal{V}}\vert_0) =\frac{\lambda}{(2 \pi i)^3} \mathbb{Q}\, \text{span} \{x^0 +\frac{(2 \pi i)^3\, Y_{000}}{3\,Y_{111}}\, x^3, x^1,x^2\} \otimes_{\mathbb{Q}} \mathbb{C},\\
&F^0( \widetilde{\mathcal{V}}\vert_0) =\frac{\lambda}{(2 \pi i)^3} \mathbb{Q}\, \text{span} \{x^0 +\frac{(2 \pi i)^3 \,Y_{000}}{3\,Y_{111}} \, x^3, x^1, x^2, x^3\} \otimes_{\mathbb{Q}} \mathbb{C}. \\
\end{aligned}
\end{equation}

The key observation is the following theorem. 

\begin{theorem} \label{MainTheorem}
Given a mirror pair $(M,W)$ of Calabi-Yau threefolds, if assuming the mirror symmetry conjecture, the limit MHS $H^3(\mathbf{Z}^{\bullet}_{\text{MS}}) $ of the mirror family at large complex structure limit splits into the direct sum
\begin{equation} \label{eq:SplitH3C}
H^3(\mathbf{Z}^{\bullet}_{\text{MS}}) \simeq \mathbb{Q}(-1) \oplus \mathbb{Q} (-2) \oplus \mathbf{M},
\end{equation}
where $\mathbf{M}$ is a two-dimensional MHS with underlying rational vector space
\begin{equation}
\mathbf{M}_{\mathbb{Q}}=\mathbb{Q}\,\frac{1}{(2 \pi i)^3}\,x^0+\mathbb{Q}\, x^3.
\end{equation}
While the weight filtration $W_*\,\mathbf{M}$ is given by
\begin{equation}
\begin{aligned}
&W_{-1}\,\mathbf{M}=W_{-2}\,\mathbf{M}=\cdots=0,\\
&W_0\, \mathbf{M}=W_1 \,\mathbf{M}=\cdots=W_5 \,\mathbf{M}=\mathbb{Q}\,x^3,\\
&W_6 \,\mathbf{M} =W_7 \,\mathbf{M}=\cdots=\mathbb{Q}\,\frac{1}{(2 \pi i)^3}\,x^0+\mathbb{Q}\,x^3,
\end{aligned}
\end{equation}
and the Hodge filtration $F^*\textbf{M}$ is given by
\begin{equation}
\begin{aligned}
&F^4 \,\mathbf{M}=F^5\, \mathbf{M}=\cdots =0, \\
&F^3\, \mathbf{M} =F^2 \,\mathbf{M}=F^1\,\mathbf{M}=  \frac{\lambda}{(2 \pi i)^3} \mathbb{Q}\,\left(x^0 +\frac{(2 \pi i)^3\, Y_{000}}{3\,Y_{111}} \, x^3\right) \otimes_{\mathbb{Q}} \mathbb{C}, \\
&F^0\,\mathbf{M} = \frac{\lambda}{(2 \pi i)^3} \left(\mathbb{Q}\,\left(x^0 + \frac{(2 \pi i)^3\, Y_{000}}{3\,Y_{111}}\, x^3\right)+\mathbb{Q}\,x^3\right) \otimes_{\mathbb{Q}} \mathbb{C}, \\
&F^{-1} \,\mathbf{M}=F^{-2}\,\mathbf{M}=\cdots=F^0\,\mathbf{M}.
\end{aligned}
\end{equation}
\end{theorem}
\begin{proof}
This theorem follows immediately from the weight filtration $W_*(\widetilde{\mathcal{V}}\vert_{0,\mathbb{Q}})$ in the formula \ref{eq:LimitWeightFiltration} and the Hodge filtration $F^*(\widetilde{\mathcal{V}}\vert_0)$ in the formula \ref{eq:LimitHodgeFiltration}.
\end{proof}

\subsection{The extensions induced by the limit MHS}

We now study the extensions defined by $\mathbf{M}$ and its dual. From Theorem \ref{MainTheorem}, the graded piece 
\begin{equation}
\text{Gr}_0^W\, \mathbf{M}:=W_0\, \mathbf{M}/W_{-1}\, \mathbf{M}
\end{equation}
is just $W_0\, \mathbf{M}$ itself. The Hodge filtration $F^*\,\textbf{M}$ induces a  weight 0 pure Hodge structure on $W_0\, \mathbf{M}$, which is isomorphic to the Tate object $\mathbb{Q}(0)$. Similarly, $F^*\,\textbf{M}$ induces a weight 6 pure Hodge structure on the graded piece
\begin{equation}
\text{Gr}_6^W\, \mathbf{M}:=W_6\, \mathbf{M}/W_5\, \mathbf{M}
\end{equation}
that is isomorphic to $\mathbb{Q}(-3)$. Furthermore, the inclusion $W_0\, \mathbf{M} \subset \mathbf{M}$ defines an injective homomorphism from $\mathbb{Q}(0)$ to $\mathbf{M}$, the quotient of which is the pure Hodge structure $\text{Gr}_6^W\, \mathbf{M}$. Therefore we have found a short exact sequence in $\textbf{MHS}_{\mathbb{Q}}$ of the form
\begin{equation} \label{eq:SESM}
\begin{tikzcd}
0 \arrow[r] & \mathbb{Q}(0) \arrow[r] & \mathbf{M} \arrow[r] & \mathbb{Q}(-3)\arrow[r] & 0,
\end{tikzcd}
\end{equation}
hence $\textbf{M}$ forms an extension of $\mathbb{Q}(-3)$ by $\mathbb{Q}(0)$. 
\begin{remark}
Therefore we have shown that for every $q\in \mathbb{Z} $, $H^q(\mathbf{Z}_{\text{MS}}^{\bullet})$ is a mixed Hodge-Tate object which is nonzero for only finitely many $q$. So the equivalence in Proposition \ref{MTHequivalence} immediately implies that $\mathbf{Z}_{\text{MS}}^{\bullet}$ is essentially an object of the full-subcategory $D^b(\textbf{MHT}_{\mathbb{Q}})$.
\end{remark}

In the abelian category $\textbf{MHS}_{\mathbb{Q}}$, the dual of an object $\textbf{H}$ is defined as \cite{Carlson,PetersSteenbrink}
\begin{equation}
\mathbf{H}^{\vee}:=\text{Hom}_{\textbf{MHS}_{\mathbb{Q}}}(\textbf{H},\mathbb{Q}(0)).
\end{equation}
The dual operation is exact, i.e. it sends a short exact sequence to a short exact sequence \cite{PetersSteenbrink}.  Therefore the dual of the short exact sequence \ref{eq:SESM} is also a short exact sequence
\begin{equation} \label{eq:SESMDual}
\begin{tikzcd}
0 \arrow[r] & \mathbb{Q}(3) \arrow[r] & \mathbf{M}^{\vee} \arrow[r] & \mathbb{Q}(0)\arrow[r]& 0.
\end{tikzcd}
\end{equation}

\textbf{Theorem  \ref{eq:ThmPeriodofMdual} } \textit{Assuming the mirror symmetry conjecture, the dual object $\mathbf{M}^{\vee}$ is an extension of $\mathbb{Q}(0)$ by $\mathbb{Q}(3)$ whose image in $\mathbb{C}/(2  \pi  i)^3\,\mathbb{Q}$ is the coset of $-(2\pi i)^3 \, Y_{000}/(3\,Y_{111})$.}

\begin{proof}
The short exact sequence \ref{eq:SESMDual} immediately shows $\mathbf{M}^{\vee}$ is an extension of $\mathbb{Q}(0)$ by $\mathbb{Q}(3)$. Let $\{x_j\}_{j=0}^3$ be the dual of $\{x^j\}_{j=0}^3$, i.e. their pairings are
\begin{equation}
x_j(x^k)=\delta^k_j,
\end{equation}
so $\{x_j\}_{j=0}^3$ forms a basis of $(\widetilde{\mathcal{V}}\vert_{0,\mathbb{C}})^{\vee}$. From Definition \ref{internalhom}, the rational vector space of $\mathbf{M}^{\vee}$ is the subspace of $(\widetilde{\mathcal{V}}\vert_{0,\mathbb{Q}})^{\vee}$ spanned by $\{(2 \pi i)^3\, x_0, x_3 \}$
\begin{equation}
(\mathbf{M}^{\vee})_{\mathbb{Q}}:=\mathbb{Q}\,(2 \pi i)^3\, x_0+\mathbb{Q}\,x_3.
\end{equation}
The Definition \ref{internalhom} tells us that the weight filtration $W_*\,\mathbf{M}^{\vee}$ is given by
\begin{equation}
W_l\, \mathbf{M}^{\vee}:=\{ {\phi:\phi\,(W_r\, \mathbf{M}) \subset W_{r+l}\,\mathbb{Q}(0)} \},
\end{equation}
from which we find that
\begin{equation}
\begin{aligned}
&W_{-7}\,\mathbf{M}^{\vee}=W_{-8}\,\mathbf{M}^{\vee}= \cdots=0, \\
&W_{-6}\,\mathbf{M}^{\vee}=\cdots=W_{-1}\,\mathbf{M}^{\vee}=\mathbb{Q}\,(2 \pi i)^3\,x_0, \\
&W_0\, \mathbf{M}^{\vee}=W_1 \,\mathbf{M}^{\vee}=\cdots=\mathbb{Q}\,(2 \pi i)^3\,x_0+\mathbb{Q}\,x_3.
\end{aligned}
\end{equation}
While from Definition \ref{internalhom}, the Hodge filtration $F^*\,\mathbf{M}^{\vee}$ is given by
\begin{equation}
F^p\,\mathbf{M}^{\vee}:=\{ \phi:\phi\,(F^r\, \mathbf{M}) \subset F^{r+p} \,\mathbb{Q}(0) \},
\end{equation}
from which we find that
\begin{equation} \label{eq:MdualHodge}
\begin{aligned}
&F^1\,\mathbf{M}^{\vee}=F^2\,\mathbf{M}^{\vee}=\cdots=0, \\
&F^0\,\mathbf{M}^{\vee}=F^{-1}\,\mathbf{M}^{\vee}=F^{-2}\,\mathbf{M}^{\vee}=(2 \pi i)^3 \,\mathbb{Q}\, \left(-\frac{(2 \pi i)^3\, Y_{000}}{3\,Y_{111}}\,x_0 +x_3\right) \otimes_{\mathbb{Q}} \mathbb{C}, \\
&F^{-3}\,\mathbf{M}^{\vee}=F^{-4}\,\mathbf{M}^{\vee}=\cdots= (2 \pi i)^3\,\left(\mathbb{Q}\,\left(-\frac{(2 \pi i)^3 \, Y_{000}}{3\,Y_{111}}\,x_0 +x_3\right)+\mathbb{Q}\, x_0 \right)\otimes_{\mathbb{Q}} \mathbb{C}.
\end{aligned}
\end{equation} 
Hence Appendix \ref{subsec:ExtensionMHS} immediately shows that the image of $ \mathbf{M}^{\vee}$ in
\begin{equation}
\text{Ext}^1_{\textbf{MHS}_{\mathbb{Q}}}\left(\mathbb{Q}(0), \mathbb{Q}(3)\right) \simeq\mathbb{C}/(2 \pi i)^3 \,\mathbb{Q}
\end{equation}
is the coset of $-(2 \pi i)^3 \, Y_{000}/(3\,Y_{111})$. 

\end{proof}

Now we are ready to state the conclusion of this paper.

\subsection{Conclusion}

For all examples of the mirror pairs where $Y_{000}$ has been computed, it is always of the form
\begin{equation}
Y_{000}=-3\, \chi (M)\, \frac{\zeta(3)}{(2 \pi i)^3}+r,~r \in \mathbb{Q},
\end{equation}
hence for such a mirror pair the image of $\mathbf{M}^{\vee}$ in $\mathbb{C}/(2 \pi i)^3 \,\mathbb{Q}$ is the coset of a rational multiple of $\zeta(3)$. This is compatible with the Remark \ref{Y000remark}, i.e. given a different choice of an integral symplectic basis of $H_3(\mathcal{W}_{\varphi_0},\mathbb{Z})$, the coefficient $Y_{000}$ is changed to $Y_{000}+r',\,r' \in \mathbb{Q}$, but the coset of $-(2 \pi i)^3 \, Y_{000}/(3\,Y_{111})$ in $\mathbb{C}/(2 \pi i)^3 \,\mathbb{Q}$ does not change. 

From the studies of mirror symmetry, there are many one-parameter mirror pairs $(M,W)$ of Calabi-Yau threefolds that satisfy the requirements in this paper, i.e.
\begin{enumerate}
\item The deformation of the mirror threefold $W$ is rationally defined, and the large complex structure limit is a rational point, while the singular fiber over it is reduced with nonsingular components crossing normally;

\item The nowhere-vanishing three-form $\Omega$ of the mirror family is rationally defined and it has logarithmic poles along the smooth components of the singular fiber over the large complex structure limit;

\item The coefficient $Y_{000}$ in the expansion of the prepotential $\mathcal{F}$ (formula \ref{eq:Prepotential}) has been computed explicitly.
\end{enumerate}
The most famous example that satisfies these requirements is the quintic Calabi-Yau threefold and its mirror family \cite{PhilipXenia, MarkGross}. Now assuming Conjecture \ref{limitmotiveconjecture} and the mirror symmetry conjecture, for a one-parameter mirror pair that satisfies the above requirements, there exists a limit mixed motive $\mathcal{Z}_{\text{MS}} \in \textbf{DM}_{\text{gm}}(\mathbb{Q}, \mathbb{Q})$ constructed at the large complex structure limit such that its Hodge realisation $\mathfrak{R}(\mathcal{Z}_{\text{MS}})$ is a complex in $D^b(\textbf{MHT}_{\mathbb{Q}})$ whose cohomologies compute the limit MHS of the mirror family at the large complex structure limit. Furthermore, the computations in Section \ref{ConstructionLMHS} and this section show that the dual of $\mathcal{Z}_{\text{MS}}$ fulfill a compelling example of the Conjecture \textbf{GHP}.

\begin{remark}
On the other hand, if instead we assume the mirror symmetry conjecture, conjectures \textbf{GHP} and \ref{limitmotiveconjecture} from the beginning, given a one-parameter mirror pair $(M,W)$ that satisfies the conditions 1 and 2 above, the computations in this paper show that the coefficient $Y_{000}$ of the prepotential \ref{eq:Prepotential} must be of the form
\begin{equation}
Y_{000}=\frac{\,r_1}{(2 \pi i)^3} \,\zeta(3)+\,r_2,~r_1,r_2 \in \mathbb{Q}.
\end{equation}
Hence the Conjecture \textbf{GHP} provides a motivic interpretation of the occurrence of $\zeta(3)$ in the coefficient $Y_{000}$ of the prepotential \ref{eq:Prepotential}.
\end{remark}

\section*{Acknowledgments}

It is a great pleasure to acknowledge the many communications with Joseph Ayoub, who generously corrected and clarified many confusions about nearby cycle functors and mixed motives. Further thanks go to Francis Brown, Annette Huber-Klawitter and Marc Levine for very helpful answers to queries about the mixed Tate motives. We are also grateful to Noriko Yui for a reading of the draft. W.Y. is very grateful to the Mathoverflow community, especially Mikhail Bondarko, who helped him to learn the theories of Hodge conjectures and mixed motives. W.Y. is also very grateful for many discussions with Philip Candelas, Xenia de la Ossa and Noriko Yui on the arithmetic of Calabi-Yau threefolds. W.Y. wishes to acknowledge the support from the Oxford-Palmer Graduate Scholarship and the generosity of Dr. Peter Palmer and Merton College. M.K. was supported in part by the EPSRC grant  `Symmetries and Correspondences', EP/M024830/1.

\newpage

\appendix

\section{Mixed Hodge structures} \label{MHSandExtensions}

In these appendices, we include some well-known definitions and basic properties of motives and mixed Hodge structures. The intention is to make the paper more accessible to readers with a physics background.
Regarding mixed Hodge structures, the reader is referred to \cite{PetersSteenbrink,Carlson} for more systematic and complete treatments. 
For motives, there is the collection of articles in the now classic volume  \cite{motives}.

Throughout this section, the ring $R$ will be either $\mathbb{Z}$ or $\mathbb{Q}$.

\subsection{Definition of the mixed Hodge structures} \label{subsec:DefinitionMHS}

An (pure) $R$-Hodge structure $H$ of weight $l\in \mathbb{Z}$ consists of the following data:
\begin{enumerate}
	\item  An $R$-module $H_{R}$ of finite rank,
	
	\item  A decreasing filtration $F^*H$ of the complex vector space $H_{\mathbb{C}}:=H_{R} \otimes_{R} \mathbb{C}$,
\end{enumerate}
such that $H_{\mathbb{C}}$ admits a decomposition
\begin{equation}
H_{\mathbb{C}}=\oplus_{p+q=l}\, H^{p,q},
\end{equation}
where $H^{p,q}:=F^{p} \cap \overline{F}^q$ \cite{PetersSteenbrink}. Here the complex conjugation is defined with respect to the real structure $H_{\mathbb{R}}:=H_{R} \otimes_{R} \mathbb{R}$ of $H_{\mathbb{C}}$. This definition immediately implies that 
\begin{equation}
F^k=\oplus_{p \geq k}\, H^{p,\,l-p}.
\end{equation}
The simplest example of a pure Hodge structure is the Hodge-Tate object $R(n),n \in \mathbb{Z}$ with weight $-2\,n$.
\begin{definition}
The $R$ module of the Hodge-Tate object $R(n)$ is
\begin{equation}
(2 \pi i)^n R \subset \mathbb{C}
\end{equation}
and its Hodge decomposition is
\begin{equation}
R(n)^{-n,-n}=(2 \pi i)^n R \otimes_R \mathbb{C}.
\end{equation}
\end{definition}
An $R$-mixed Hodge structure (MHS) consists of the following data:
\begin{enumerate}
	\item An $R$-module $H_{R}$ of finite rank,
	
	\item An increasing weight filtration $W_*$ of $H_{\mathbb{Q}}:=H_{R}\otimes_{R} \mathbb{Q}$,
	
	\item A decreasing Hodge filtration $F^*$ of $H_{\mathbb{C}}:=H_{R} \otimes_{R} \mathbb{C}$,
\end{enumerate}
such that the Hodge filtration $F^*$ induces a pure Hodge structure of weight $l$ on every graded piece $\text{Gl}_{l}^W\,W $ \cite{PetersSteenbrink}
\begin{equation}
\text{Gl}_{l}^W\,W :=W_l/W_{l-1}.
\end{equation}
Morphisms between two $R$-MHS are defined as the linear maps which are compatible with both the weight filtrations and Hodge filtrations \cite{Carlson,PetersSteenbrink}.
\begin{definition}
Given two $R$-MHS $A$ and $B$, a morphism of weight $2\,m$ from $A$ to $B$ is a homomorphism $\phi$ from $A_R$ to $B_R$ such that
\begin{equation}
\phi \,(W_l\,A) \subset W_{l+2\,m} B~\forall\,l, ~~
\phi \,(F^p\,A) \subset F^{p+m} \,B, ~\forall\,p. 
\end{equation}
\end{definition}
The category of $R$-MHS will be denoted by $\textbf{MHS}_R$, and it is an abelian category \cite{PetersSteenbrink}. The internal Hom operation is defined as follows \cite{PetersSteenbrink,Carlson}.

\begin{definition} \label{internalhom}
Given two $R$-MHS $A$ and $B$, there exists an $R$-MHS $\text{Hom}(A,B)$ with $R$-module
\begin{equation}
\text{Hom}(A,B)_R:=\text{Hom}(A_R,B_R),
\end{equation}
while its weight filtration and Hodge filtration are given by
\begin{equation}
\begin{aligned}
&W_l\,(\text{Hom}(A,B))=\{\phi:\phi\,(W_r A) \subset W_{r+l}\, B, \forall\, r \}, \\
&F^p\,(\text{Hom}(A,B))=\{\phi: \phi\,(F^r A) \subset F^{r+p}\, B ,\forall\, r \}.
\end{aligned}
\end{equation} 
\end{definition}

\subsection{Extensions of MHS}\label{subsec:ExtensionMHS}

An extension of $B$ by $A$ in the abelian category $\textbf{MHS}_R$ is given by a short exact sequence
\begin{equation} \label{eq:ExtensionAHB}
\begin{tikzcd}
0 \arrow[r] & A \arrow[r] & H \arrow[r] & B \arrow[r] & 0.
\end{tikzcd}
\end{equation}
Two extensions of $B$ by $A$ are said to be isomorphic if there exists a commutative diagram of the form
\begin{equation}
\begin{tikzcd}
0 \arrow[r] & A \arrow[r] \arrow[d,"\text{Id}"]&  H \arrow[r] \arrow[d,"\simeq"] & B \arrow[r] \arrow[d,"\text{Id}"] & 0 \\
0 \arrow[r] & A \arrow[r] & H' \arrow[r] & B \arrow[r] & 0
\end{tikzcd}.
\end{equation}
The extension \ref{eq:ExtensionAHB} is said to split if it is isomorphic to the trivial extension 
\begin{equation} \label{eq:TrivialExtensionAB}
\begin{tikzcd}
0 \arrow[r,"i"] & A \arrow[r] & A \oplus B \arrow[r,"j"] & B \arrow[r] & 0,
\end{tikzcd}
\end{equation}
where $i$ is the natural inclusion and $j$ is the the natural projection \cite{Carlson,PetersSteenbrink}.

\begin{definition} \label{DefnTateHodge}
The abelian category of mixed Hodge-Tate structures, denoted by $\textbf{MHT}_R$, is defined as the smallest full abelian subcategory of $\textbf{MHS}_R$ that contains the Hodge-Tate objects $R(n),n \in \mathbb{Z}$ while also being closed under extensions.	
\end{definition}

The set of isomorphism classes of extensions of $B$ by $A$, denoted by $\text{Ext}^1_{\textbf{MHS}_R}(B,A)$, has a group structure imposed by the Baer summation, while the zero object is the trivial extension \ref{eq:TrivialExtensionAB} \cite{PetersSteenbrink,Carlson}. Two $R$-MHS $A$ and $B$ are said to be separated if the highest weight of $A$ is lower than the lowest weight of $B$, in which case the extension \ref{eq:ExtensionAHB} is also said to be separated. When $A$ and $B$ are separated, there is a canonical and functorial description of the group $\text{Ext}^1_{\textbf{MHS}_R}(B,A)$ given by \cite{Carlson,PetersSteenbrink}
\begin{equation} \label{eq:ExtensionIso}
\text{Ext}^1_{\textbf{MHS}_R}(B,A)=\text{Hom}(B,A)_R \otimes_R \mathbb{C}/(F^0\, \text{Hom}(B,A)+\text{Hom}(B,A)_R).
\end{equation}  
In particular we have the following important lemma.
\begin{lemma}
When $n \geq 1$, $\mathbb{Q}(n)$ and $\mathbb{Q}(0)$ are separated and we have
\begin{equation} \label{eq:Isoq0q3}
\text{Ext}^1_{\textbf{MHS}_{\mathbb{Q}}}(\mathbb{Q}(0),\mathbb{Q}(n)) = \mathbb{C}/(2 \pi i)^n \,\mathbb{Q}
\end{equation}
\end{lemma}
\begin{proof}
The weight of $\mathbb{Q}(n)$ is $-2\,n$ and the weight of $\mathbb{Q}(0)$ is 0, so they form a separated pair. The rational vector spaces of $\mathbb{Q}(0)$ and $\mathbb{Q}(n)$ are respectively
\begin{equation}
\mathbb{Q}(0): \mathbb{Q} \subset \mathbb{C},~~
\mathbb{Q}(n): (2 \pi i)^n\, \mathbb{Q}  \subset \mathbb{C} .\\
\end{equation} 
From Definition \ref{internalhom}, we have
\begin{equation}
F^0 \, \text{Hom}(B,A)=0.
\end{equation}
Now we choose an isomorphism
\begin{equation}
\text{Hom}(\mathbb{Q}(0),\mathbb{Q}(n)) \otimes_{\mathbb{Q}} \mathbb{C} \simeq \mathbb{C}
\end{equation}
such that
\begin{equation}
\text{Hom}(\mathbb{Q}(0),\mathbb{Q}(n))_{\mathbb{Q}} \simeq (2 \pi i)^n \,\mathbb{Q},
\end{equation}
then the formula \ref{eq:ExtensionIso} immediately implies the formula \ref{eq:Isoq0q3}.

\end{proof}

When $n \geq 1$, given an element $\bar{s}$ of $\mathbb{C}/(2 \pi i)^n\, \mathbb{Q}$, we now construct an extension $H$ such that its image in $\mathbb{C}/(2 \pi i)^n\, \mathbb{Q}$ is $\bar{s}$. The complex vector space $\mathbb{C}^2$ has a natural basis $\{e_j\}_{j=1}^2$
\begin{equation}
e_1=(1,0),~e_2=(0,1).
\end{equation}
Let the rational vector space of $H$ be
\begin{equation}
H_{\mathbb{Q}}:=\mathbb{Q}\,(2 \pi i)^n\, e_1+\mathbb{Q}\, e_2 \subset \mathbb{C}^2
\end{equation}
The weight filtration of $H$ is given by
\begin{equation}
\begin{aligned}
&W_{-2\,n-1}\,H=W_{-2\,n-2}\,H=\cdots=0, \\
&W_{-2\,n}=\cdots=W_{-1}=\mathbb{Q}\, (2 \pi i)^n\, e_1,\\
&W_0\, H=W_1\,H=\cdots=H_{\mathbb{Q}}.
\end{aligned}
\end{equation}
Let $s$ be an arbitrary complex number whose coset in $\mathbb{C}/(2 \pi i)^n\, \mathbb{Q}$ is $\bar{s}$, then the Hodge filtration of $H$ is given by
\begin{equation}
\begin{aligned}
&F^1=F^2=\cdots=0,\\
&F^0=\cdots =F^{-(n-1)}= \mathbb{C}\,(s\,e_1+e_2),\\
&F^{-n}=F^{-n-1}=\cdots=\mathbb{C}^2.
\end{aligned}
\end{equation}
Now $H$ defines a short exact sequence
\begin{equation}
\begin{tikzcd}
0 \arrow[r]& \mathbb{Q}(n) \arrow[r]&  H \arrow[r] & \mathbb{Q}(0) \arrow[r]& 0
\end{tikzcd}
\end{equation}
where the morphism from $\mathbb{Q}(n)$ to $H$ is the natural inclusion. Thus $H$ forms an extension of $\mathbb{Q}(0)$ by $\mathbb{Q}(n)$. From the proof of the formula \ref{eq:ExtensionIso} in \cite{Carlson, PetersSteenbrink}, the image of $H$ in $\mathbb{C}/(2 \pi i)^n\, \mathbb{Q}$ is $\bar{s}$. From the construction of $H$, we deduce that the extension defined by $H$ does not depend on the choice of $s$.

\section{Pure motives} \label{PureMotive}

This section is an overview of the construction of pure motives, however it is not meant to be complete and necessary references will be given for further reading. 

\subsection{Algebraic Cycles}

First, we need to give a brief introduction to algebraic cycles. An excellent reference to the theory of algebraic cycles is the book \cite{FultonI}, which is strongly recommended to the readers. Let $\textbf{SmProj}/k$ be the category of non-singular projective varieties over a field $k$, which is a symmetric monoidal category with product given by fiber product of varieties and symmetry given by the canonical isomorphism 
\begin{equation} \label{eq:ProjVarsymmetry}
X \times_k Y \rightarrow Y \times_k X.
\end{equation}
A prime cycle $Z$ of a non-singular projective variety $X$ is an irreducible algebraic subvariety, and its codimension is defined as $\text{dim}\,X-\text{dim}\,Z$. On the other hand, an irreducible closed subset of $X$ has a natural algebraic variety structure induced from that of $X$ \cite{Hartshorne}. The set of prime cycles of dimension $r$ (resp. codimension $r$) generates a free abelian group that will be denoted by $C_r(X)$ (resp. $C^r(X)$), and elements of $C_r(X)$ (resp. $C^r(X)$) will be called the algebraic cycles of dimension $r$ (resp. codimension $r$). Two prime cycles $Z_1$ and $Z_2$ are said to intersect with each other properly if
\begin{equation}
\text{codim}(Z_1 \cap Z_2) =\text{codim}(Z_1)+\text{codim}(Z_2),
\end{equation} 
where $Z_1 \cap Z_2$ means the set-theoretic intersection between $Z_1$ and $Z_2$. If two prime cycles $Z_1$ and $Z_2$ intersect with each other properly, the intersection product $Z_1 \cdot Z_2$ is defined as
\begin{equation}
Z_1 \cdot Z_2= \sum_T m(T;\,Z_1 \cdot Z_2)\, T,
\end{equation}
where the sum is over all irreducible components of $Z_1 \cap Z_2$ and $m(T;\,Z_1 \cdot Z_2)$ is Serre's intersection multiplicity formula \cite{FultonI}. Now extend the definition by linearity, the intersection product is defined for algebraic cycles $Z=\sum_j\,m_j\,Z_j$ and $W= \sum_l\, n_l\,W_l$ when $Z_j$ and $W_l$ intersect properly for all $j$ and $l$. Therefore there is a partially defined intersection product on algebraic cycles
\begin{equation}
\begin{tikzcd}
C^r(X) \times C^s(X) \arrow[r,dotted] &C^{r+s}(X).
\end{tikzcd}
\end{equation}
If $f:X \rightarrow Y$ is a morphism between two non-singular projective varieties $X$ and $Y$, the pushforward homomorphism $f_*$ on algebraic cycles is defined by
\begin{equation}
f_*(Z):=
\begin{cases*}
0 & if $\text{dim}\, f(Z) < \text{dim}\, Z$, \\
[k(Z):k(f(Z))] \cdot f(Z)        & if $\text{dim}\, f(Z) = \text{dim}\, Z$,
\end{cases*}
\end{equation}
where $Z$ is a prime cycle and $k(Z)$ (resp. $k(f(Z))$) is the function field of $Z$ (resp. $f(Z)$) \cite{Hartshorne}. Here $ [k(Z):k(f(Z))] $ is the degree of field extension. Now we want to define the pullback homomorphism $f^*$. Given a prime cycle $W$ of $Y$, the first attempt is to naively try 
\begin{equation} \label{eq:cyclepullback}
f^*(W):= \sum_{T \subset f^{-1}(Z)} \ell_{\mathcal{O}_{X,T}}(\mathcal{O}_{f^{-1}(Z),T}) \cdot T,
\end{equation}
where the sum is over the irreducible components of $f^{-1}(Z)$ and $\ell_{\mathcal{O}_{X,T}}(\mathcal{O}_{f^{-1}(Z),T})$ is the length of $\mathcal{O}_{f^{-1}(Z),T}$ in $\mathcal{O}_{X,T}$ \cite{FultonI}. However this definition is only partially defined and in general $f^*(W)$ does not make sense. The solution to the above problems is to find an equivalence relation $\sim$ on the algebraic cycles such that the quotient group $C^*(X)/\sim$ behaves very well.
\begin{definition}
An equivalence relation $\sim$ on the algebraic cycles is called an adequate equivalence relation if given two arbitrary cycles $Z_1$ and $Z_2$, there exists a cycle $Z'_1$ in the equivalence class of $Z_1$ such that $Z'_1$ intersects with $Z_2$ properly, while the equivalence class of the intersection $Z'_1 \cdot Z_2$ is independent of the choice of $Z'_1$.
\end{definition}

Hence for an adequate equivalence relation $\sim$, there is a well defined intersection product on the quotient group $C_{\sim}^*(X):=C^*(X)/\sim$
\begin{equation}
C^r(X)_{\sim} \times C^s(X)_{\sim}  \rightarrow  C^{r+s}(X)_{\sim},
\end{equation}
moreover, the pushforward and pullback homomorphisms are also well defined \cite{FultonI}
\begin{equation}
f_*:C_{r,\sim}(X) \rightarrow C_{r,\sim}(Y), ~~f^*:C^r_{\sim}(Y) \rightarrow C^r_{\sim}(X).
\end{equation}
The set of adequate equivalence relations is ordered in the way such that $\sim_1$ is said to be finer than $\sim_2$ if for every cycle $Z$, $Z \sim_1 0$ implies $Z \sim_2 0$. The most important adequate equivalence relations are the rational equivalence and the numerical equivalence. In fact, rational equivalence is the finest adequate equivalence relation and numerical equivalence is the coarsest adequate equivalence relation \cite{Andre, Scholl, FultonI}.

\subsection{Weil cohomology theory}

Now we will briefly discuss the Weil cohomology theory. Let $\text{Gr}^{\geq 0}\, \text{Vec}_K$ be the rigid tensor abelian category of finite dimensional graded vector spaces over a field $K$ with $\text{char}\,K=0$. An object $V$ of $\text{Gr}^{\geq 0}\, \text{Vec}_K$ has a decomposition of the form
\begin{equation}
V=\oplus_{r \geq 0} \, V_r,
\end{equation} 
where $V_r$ consists of homogeneous elements with degree $r$. Tensor product in $\text{Gr}^{\geq 0}\, \text{Vec}_K$ will be denoted by $\otimes_K$. The category $\text{Gr}^{\geq 0}\, \text{Vec}_K$ admits a graded symmetry defined by
\begin{equation}
v \otimes_K w \rightarrow (-1)^{\text{deg}\,v\, \text{deg}\, w}\, w\otimes_K v,
\end{equation}
where both $v$ and $w$ are homogeneous elements. On the other hand, the category $\textbf{SmProj}/k$ also admits product and symmetry operation, and a Weil cohomology theory is a symmetric monoidal functor 
\begin{equation}
H^*:\textbf{SmProj}/k^{\text{op}} \rightarrow \text{Gr}^{\geq 0}\, \text{Vec}_K 
\end{equation}
that satisfies a list of axioms \cite{MilneEC}. Among these axioms is the existence of a cycle map
\begin{equation}
\text{cl}:C^*_{\text{rat}}(X)_{\mathbb{Q}} \rightarrow H^*(X),
\end{equation}
which doubles the degree and sends the intersection product of cycles to the cup product of cohomology classes. We now introduce three classical examples of Weil cohomology theories. Let $X$ be a non-singular projective variety of dimension $n$ over a field $k$.
\begin{enumerate}
	\item If $\sigma:k \rightarrow \mathbb{C}$ is an embedding of $k$ into $\mathbb{C}$, the $\mathbb{C}$-valued points of $X$, denoted by $X_{\sigma}(\mathbb{C})$, form a projective complex manifold. The Betti cohomology $H_{B,\sigma}^*(X)$ is defined as the singular cohomology of $X_{\sigma}(\mathbb{C})$ with coefficient ring $\mathbb{Q}$
	\begin{equation}
	H_{B,\sigma}^*(X):=H^*(X_{\sigma}(\mathbb{C}),\mathbb{Q}).
	\end{equation}
	Since $X_{\sigma}(\mathbb{C})$ is projective, there exists a Hodge decomposition
	\begin{equation}
H_{B,\sigma}^m(X) \otimes_{\mathbb{Q}} \mathbb{C}=\oplus_{p+q=m}\,H^{p,q}(X_{\sigma}(\mathbb{C})),
	\end{equation}
	which induces a decreasing filtration of $H_{B,\sigma}^m(X) \otimes_{\mathbb{Q}} \mathbb{C}$
	\begin{equation}
	F^l(H_{B,\sigma}^m(X) \otimes_{\mathbb{Q}} \mathbb{C}):=\oplus_{p \geq l} H^{p,m-p}(X_{\sigma}(\mathbb{C})).
	\end{equation}
	
	\item If $\text{char}\,k=0$, take $K$ to be $k$. Let $\Omega_{X/k}^*$ be the complex of sheaves of algebraic forms on $X$
	\begin{equation}
	\Omega_{X/k}^*:0 \rightarrow \mathcal{O}_{X/k} \xrightarrow{d}\Omega_{X/k}^1 \xrightarrow{d} \Omega^2_{X/k} \xrightarrow{d} \cdots \xrightarrow{d} \Omega^n_{X/k} \rightarrow 0.
	\end{equation}
	The algebraic de Rham cohomology of $X$, denoted by $H^*_{\text{dR}}(X)$, is the hypercohomology of the complex $\Omega_{X/k}^*$
	\begin{equation}
	H^*_{\text{dR}}(X):=\mathbb{H}^*(X, \Omega_{X/k}^*),
	\end{equation}
	which is a vector space over $k$. The complex $\Omega^*_{X/k}$ admits a naive filtration
	\begin{equation}
	F^p \,\Omega^*_{X/k}:0 \xrightarrow{d} 0 \cdots \xrightarrow{d} 0 \xrightarrow{d} \Omega^p_{X/k} \xrightarrow{d} \cdots \xrightarrow{d} \Omega^n_{X/k} \rightarrow 0,
	\end{equation}
	which induces a decreasing filtration of $H^m_{\text{dR}}(X)$ given by
	\begin{equation}
		F^pH^m_{\text{dR}}(X):=\text{Im}\,\big(\mathbb{H}^m(X,F^p \,\Omega^*_{X/k}) \rightarrow \mathbb{H}^m(X,\Omega^*_{X/k})\big).
	\end{equation}

	\item Suppose $\ell$ is a prime number and $\ell \neq \text{char}\,k$. Let $K$ be $\mathbb{Q}_{\ell}$ and the \'etale cohomology of $X$ is defined as
	\begin{equation}
	H^*_{\text{\'et}}(X)_{\ell}:=\lim_{\xleftarrow[n]{}} H^*_{\text{\'et}}(X \times_k k^{\text{sep}},\mathbb{Z}/{\ell}^n) \otimes_{\mathbb{Z}_{\ell} } \mathbb{Q}_{\ell},
	\end{equation}
	which is a finite dimensional continuous representation of the absolute Galois group $\text{Gal}(k^{\text{sep}}/k)$  \cite{MilneEC}.
\end{enumerate}

These three classical examples are not totally independent from each other, and there exist canonical comparison isomorphisms between them:
\begin{enumerate}
	\item The standard comparison isomorphism between the Betti cohomology and the algebraic de Rham cohomology
	\begin{equation} \label{eq:BettiDeRham}
	I_{\sigma}:H^m_{B,\sigma}(X) \otimes_{\mathbb{Q}} \mathbb{C} \simeq H^m_{\text{dR}}(X) \otimes_{k,\sigma} \mathbb{C},
	\end{equation}
	under which $F^p(H^m_{B,\sigma}(X) \otimes_{\mathbb{Q}} \mathbb{C})$ is sent to $F^p(H^m_{\text{dR}}(X))\otimes_{k,\sigma} \mathbb{C}$.
	\item The standard comparison isomorphism between the Betti cohomology and the  \'etale cohomology
	\begin{equation}
	I_{\ell,\bar{\sigma}}:H^m_{B,\sigma}(X) \otimes_{\mathbb{Q}} \mathbb{Q}_{\ell} \simeq H^m_{\text{\'et}}(X)_{\ell},
	\end{equation}
	which depends on the choice of an extension of $\sigma$ to $\bar{\sigma}:k^{\text{sep}} \rightarrow \mathbb{C}$. 
\end{enumerate}

Given an adequate equivalence relation $\sim$, $C_{\sim}^*(X)_{\mathbb{Q}}$ is defined as
\begin{equation}
C_{\sim}^*(X)_{\mathbb{Q}}:=C_{\sim}^*(X) \otimes_{\mathbb{Z}} \mathbb{Q}.
\end{equation}
A Weil cohomology theory $H^*$ yields an adequate equivalence relation $\sim_{H^*}$ defined by \cite{Andre, Scholl}
\begin{equation}
Z \sim_{H^*}0  \Leftrightarrow \text{cl}(Z)=0.
\end{equation}
Equivalences like $\sim_{H^*}$ are often expected to be independent of the specific cohomology theory.
\subsection{Pure Motives}

The three examples of classical Weil cohomology theories we give above behave as if they all arise from an algebraically defined cohomology theory over $\mathbb{Q}$, however this is known to be false \cite{Andre, Scholl}. Grothendieck's idea to explain this phenomenon is that there exists a universal cohomology theory in the sense that all Weil cohomology theories are the realisations of it. More precisely, Grothendieck conjectures that there exists a rigid tensor abelian category $\textbf{M}_{\text{hom}}$ over $\mathbb{Q}$ and a functor $M_{gm}$
\begin{equation}
M_{gm}: \textbf{SmProj}/k^{\text{op}} \rightarrow \textbf{M}_{\text{hom}}
\end{equation} 
such that for every Weil cohomology theory $H^*$, there exists a functor $H^*_m$ that factors through $M_{gm}$
\begin{equation}
\begin{tikzcd}
\textbf{SmProj}/k^{\text{op}} \arrow[rd,"H^*"'] \arrow[r, "M_{gm}"] & \textbf{M}_{\text{hom}} \arrow[d,dotted,"H^*_m"] \\
& \text{Gr}^{\geq 0}\, \text{Vec}_K 
\end{tikzcd}.
\end{equation}

Now we will introduce the construction of the category of motives $\textbf{M}_{\sim}$ where $\sim$ is the rational equivalence or the numerical equivalence \cite{Andre, Scholl}. Given two non-singular projective varieties $X$ and $Y$, the group of correspondences from $X$ to $Y$ with degree $r$ is defined as
\begin{equation}
\text{Corr}^r(X,Y):=C^{\text{dim}\,X+r}(X \times Y).
\end{equation}
The composition of correspondences 
\begin{equation}
\text{Corr}^r(X,Y) \times \text{Corr}^s(Y,Z) \rightarrow \text{Corr}^{r+s}(X,Z)
\end{equation}
is defined by  
\begin{equation}
g \times h \rightarrow h \circ g:=(p_{13})_*\big((p_{12})^*g\cdot (p_{23})^*h\big),
\end{equation}
where $p_{12}$ is the natural projection morphism from $X \times Y \times Z$ to $X \times Y$, etc \cite{Scholl}. For a morphism $f:Y \rightarrow X$, its graph $\Gamma_f$ in $X \times Y$ is an algebraic variety that is isomorphic to $Y$, therefore $\Gamma_f$ is an element of $\text{Corr}^0(X,Y)$ \cite{Hartshorne}. A correspondence of $\text{Corr}^0(X,Y)$ can be seen as a multi-valued morphism from $Y$ to $X$. A correspondence $\gamma$ defines a homomorphism from $H^*(X)$ to $H^*(Y)$ by
\begin{equation} \label{eq:cyclemap}
\gamma_*: x \mapsto p_{2,*}\,(\,p_1^*\, x \, \cup \, \text{cl}(\gamma)),
\end{equation}
where $p_1$ (resp. $p_2$) is the projection morphism from $X \times Y$ to $X$ (resp. $Y$). While the homomorphism $(\Gamma_f)_*$ induced by $\Gamma_f$ is just the pullback homomorphism $f^*$. The category $\textbf{M}_{\sim}$ is constructed in three steps \cite{Andre, Scholl}.

\begin{enumerate}
	\item Construct a category whose objects are formal symbols
	\begin{equation}
	\{M_{gm}(X):X \in \textbf{SmProj}/k \}.
	\end{equation}
	The morphisms between two objects are given by
	\begin{equation}
	\text{Hom}(M_{gm}(X),M_{gm}(Y)):=\text{Corr}^0_{\sim}(X,Y)_{\mathbb{Q}},
	\end{equation}  
	where we have
	\begin{equation}
	\text{Corr}^r_{\sim}(X,Y)=\text{Corr}^r(X,Y)/\sim,\,\text{Corr}^r_{\sim}(X,Y)_{\mathbb{Q}}=\text{Corr}^r_{\sim}(X,Y) \otimes_{\mathbb{Z}} \mathbb{Q}.
	\end{equation}
	This category can be seen as the linearisation of $\textbf{SmProj}/k^{\text{op}}$.
	
	\item Take the pseudo-abelianisation of the category constructed in Step 1 and denote this new category by $\textbf{M}^{\text{eff}}_{\sim}$. More explicitly, the objects of $\textbf{M}^{\text{eff}}_{\sim}$ are formally
	\begin{equation}
	\{(M_{gm}(X),e):X \in \textbf{SmProj}/k\,\text{and}\,\, e \in \text{Corr}^0_{\sim}(X,X)_{\mathbb{Q}} , \,e^2=e\},
	\end{equation}
	 and the morphisms between two objects are given by
	\begin{equation}
	\text{Hom}((M_{gm}(X),e),(M_{gm}(Y),f)):=f \circ \text{Corr}^0_{\sim}(X,Y)_{\mathbb{Q}} \circ e .
	\end{equation}
    Let the graph of the identity morphism of $\mathbb{P}^1$ be $\Delta_{\mathbb{P}^1}$, and in this category the object $(M_{gm}(\mathbb{P}^1),\Delta_{\mathbb{P}^1})$ has a decomposition given by \cite{Andre, Scholl}
	\begin{equation}
	(M_{gm}(\mathbb{P}^1),\Delta_{\mathbb{P}^1})=M_{gm}^0(\mathbb{P}^1) \oplus M_{gm}^2(\mathbb{P}^1).
	\end{equation}
	The component $M_{gm}^0(\mathbb{P}^1)$ is also denoted by $\mathbb{Q}(0)$, while the component $M_{gm}^2(\mathbb{P}^1)$ is also denoted by $\mathbb{Q}(-1)$.
	
	\item  The category $\textbf{M}_{\sim}$ is constructed from $\textbf{M}^{\text{eff}}_{\sim}$ by inverting the object $\mathbb{Q}(-1)$. The objects of $\textbf{M}_{\sim}$ are formally
	\begin{equation}
	\{(M_{gm}(X),e,m):X \in \textbf{SmProj}/k,\,\, e \in \text{Corr}^0_{\sim}(X,X)_{\mathbb{Q}} , \,e^2=e,\,\,\text{and}\,\,m \in \mathbb{Z}\},
	\end{equation}
	and the morphisms between two objects are given by
	\begin{equation}
	\text{Hom}((M_{gm}(X),e,m),(M_{gm}(Y),f,n)):=f \circ \text{Corr}^{n-m}_{\sim}(X,Y)_{\mathbb{Q}} \circ e.
	\end{equation}
	The category $\textbf{M}^{\text{eff}}_{\sim}$ is isomorphic to the full subcategory of $\textbf{M}_{\sim}$ generated by objects of the form $(M_{gm}(X),e,0)$.
\end{enumerate}
Given two objects of $\textbf{M}_{\sim}$, the morphisms between them form a rational vector space, which is finite dimensional if $\sim$ is the numerical equivalence. Direct sum in $\textbf{M}_{\sim}$ is essentially defined by \cite{Scholl}
\begin{equation}
(M_{gm}(X),e,m) \oplus (M_{gm}(Y),f,m):=(M_{gm}(X \amalg Y),e \oplus f,m),
\end{equation}
while tensor product in $\textbf{M}_{\sim}$ is essentially defined by
\begin{equation}
(M_{gm}(X),e,m) \otimes (M_{gm}(Y),f,n):=(M_{gm}(X \times Y),e \times f,m+n),
\end{equation}
Dual operation in $\textbf{M}_{\sim}$ is defined by
\begin{equation}
(M_{gm}(X),e,m)^{\vee}:=(M_{gm}(X),e^t,\text{dim} \, X-m)
\end{equation}
where $e^t$ means the transpose of $e$. In fact, the object $\mathbb{Q}(0)$ is a unit of $\textbf{M}_{\sim}$ \cite{Andre, Scholl}. From the construction of $\textbf{M}_{\sim}$, there is a functor
\begin{equation}
M_{gm}:\textbf{SmProj}/k^{\text{op}} \rightarrow \textbf{M}_{\sim},
\end{equation}
which sends $X$ to $(M_{gm}(X),\Delta_X,0)$ and $f: Y \rightarrow X$ to $\Gamma_f$. Hence we deduce that every Weil cohomology theory $H^*$ automatically factors through $\textbf{M}_{\text{rat}}$
\begin{equation}
\begin{tikzcd}
\textbf{SmProj}/k^{\text{op}} \arrow[rd,"H^*"'] \arrow[r, "M_{gm}"] & \textbf{M}_{\text{rat}} \arrow[d,dotted,"H^*_{\text{rat}}"] \\
& \text{Gr}^{\geq 0}\, \text{Vec}_K 
\end{tikzcd}.
\end{equation}
However the category $\textbf{M}_{\text{rat}}$ is not abelian \cite{Andre, Scholl}. On the other hand, the category $\textbf{M}_{\text{num}}$ has been proved to be abelian and semi-simple \cite{Jannsen,Andre, Scholl}, but it is not known whether an arbitrary Weil cohomology theory $H^*$ will factor through it. The problem is that, if an algebraic cycle $\gamma$ is numerically equivalent to 0, then it will define a zero morphism in $\textbf{M}_{\text{num}}$. So in order for $H^*$ to factor through $\textbf{M}_{\text{num}}$, we will need the induced homomorphism $\gamma_*$ in the formula \ref{eq:cyclemap} to be zero. However, this is not known currently, but it is conjectured to be true by Grothendieck.

\textbf{Conjecture D} \textit{If an algebraic cycle $\gamma$ is numerically equivalent to 0, then $\text{cl}(\gamma)$ is zero for every Weil cohomology theory.}

This conjecture also implies that the homological equivalence relation $\sim_{H^*}$ defined by a Weil cohomology theory $H^*$ is the same as the numerical equivalence.

\subsection{The conjectured abelian category of the mixed motives}
The theory of the pure motives can be seen as the universal Weil cohomology theory for non-singular projective varieties over $k$, so one might wonder what is the universal (Bloch-Ogus) cohomology theory for the arbitrary varieties over $k$. Beilinson conjectured that there exists a rigid tensor abelian category of mixed motives, denoted by $\textbf{MM}_k$, that has similar properties as $\textbf{MHS}_{\mathbb{Q}}$ which forms the universal Bloch-Ogus cohomology theory for arbitrary varieties over $k$ \cite{LevineMM}. Here we list several expected properties of the conjectured abelian category $\textbf{MM}_k$.

\begin{enumerate}
	\item $\textbf{MM}_k$ is a rigid tensor abelian category such that the morphisms between two objects form a vector space over $\mathbb{Q}$. It contains the Tate objects $\mathbb{Q}(n),n \in \mathbb{Z}$ which satisfy
	\begin{equation}
	\mathbb{Q}(m) \otimes \mathbb{Q}(n)=\mathbb{Q}(m+n),
	\end{equation}
	while $\mathbb{Q}(0)$ is a unit object of $\textbf{MM}_k$. 
	
	\item There exists a contravariant functor $M_{gm}$ from the category of varieties over $k$ to the derived category of $\textbf{MM}_k$
	\begin{equation}
	M_{gm}:\textbf{Var}/k^{\text{op}} \rightarrow D^b( \textbf{MM}_k).
	\end{equation} 	
	
	\item The full subcategory of $\textbf{MM}_k$ generated by the semi-simple objects is equivalent to the category of pure motives.
	
	\item For every object $\mathcal{M}$ of $\textbf{MM}_k$, there exists a finite weight filtration $W_*(\mathcal{M})$ such that all the graded pieces $\text{Gr}_{l}^W(\mathcal{M})$ are pure motives. 
	
	\item If $\sigma:k \rightarrow \mathbb{C}$ is an embedding, there exists a Hodge realisation functor
	\begin{equation}
	\mathfrak{R}_{\sigma}: \textbf{MM}_k \rightarrow \textbf{MHS}_{\mathbb{Q}},
	\end{equation}
	which is compatible with all the structures of $\textbf{MM}_k$ and $\textbf{MHS}_{\mathbb{Q}}$. For every variety $X$ over $k$, $\mathfrak{R}_{\sigma}(H^q(M_{gm}(X)))$ is the $q$-th Betti cohomology $H_{B,\sigma}^q(X)$ together with the (natural) rational MHS on it \cite{DeligneIII}.	
	
	\item The abelian subcategory of the mixed Tate motives $\textbf{TM}_{k}$ is the smallest full abelian subcategory of $\textbf{MM}_k$ that contains the Tate objects $\mathbb{Q}(n),n\in \mathbb{Z}$ while also being closed under extension.

\end{enumerate}

The construction of an abelian category $\textbf{MM}_k$ that possesses all the expected properties is still beyond reach. However now there are several constructions of a triangulated tensor category that satisfy nearly all the expected properties of the derived category of $\textbf{MM}_k$, except those properties that need the triangulated category to be realised as the derived category of an abelian category, like a motivic $t$-structure. One notable example is Voevodsky's construction of $\textbf{DM}(k,\mathbb{Q})$ \cite{VoevodskyMM,VoevodskyCycles}.

\section{Mixed Hodge complexes}

In this section, we will give a routine proof to a well-known result, Proposition \ref{MTHequivalence}, as we can not find a precise reference. Let $D_{\textbf{MHS}_{\mathbb{Q}}}^b$ be the bounded derived category of rational mixed Hodge complexes, the definition of which is left to the paper \cite{BeilinsonMHC}.
\begin{definition}
Let $D_{\textbf{MHT}_{\mathbb{Q}}}^b$ be the full subcategory of $D_{\textbf{MHS}_{\mathbb{Q}}}^b$ that consists of $\mathbb{Q}$-mixed Hodge complexes whose cohomologies are mixed Hodge-Tate objects
\begin{equation}
D_{\textbf{MHT}_{\mathbb{Q}}}^b:=\{F^{\bullet} \in D_{\textbf{MHS}_{\mathbb{Q}}}^b : \underline{H}^q(F^{\bullet}) \in  \textbf{MHT}_{\mathbb{Q}},~\forall~q\in\mathbb{Z} \}.
\end{equation}
\end{definition}
The restriction of the cohomology functor $\underline{H}^*$ to $D_{\textbf{MHT}_{\mathbb{Q}}}^b$ is of the form
\begin{equation} \label{eq:HbarDTH}
\underline{H}^*:D_{\textbf{MHT}_{\mathbb{Q}}}^b \rightarrow \textbf{MHT}_{\mathbb{Q}}.
\end{equation}
Now we will prove $D_{\textbf{MHT}_{\mathbb{Q}}}^b$ is actually a full triangulated subcategory of $D_{\textbf{MHS}_{\mathbb{Q}}}^b$, but first we need the following lemma.
\begin{lemma} \label{THsubquotient}
The category $\textbf{MHT}_{\mathbb{Q}}$ of mixed Hodge-Tate objects is closed under taking sub-quotient object.
\end{lemma}

\begin{proof}
Suppose $B$ is an object of $\textbf{MHT}_{\mathbb{Q}}$ and $A$ is a sub-object of $B$ in the bigger category $\textbf{MHS}_{\mathbb{Q}}$, i.e.
\begin{equation}
A \subset B.
\end{equation}
First we will prove that $A$ is also an object of $\textbf{MHT}_{\mathbb{Q}}$. As $B$ is a mixed Hodge-Tate object, it has a finite filtration given by objects $\{B_i\}_{i=0}^N$ of $\textbf{MHT}_{\mathbb{Q}}$
\begin{equation}
0=B_0 \subsetneqq B_1 \subsetneqq \cdots \subsetneqq B_N=B
\end{equation}
such that the quotients are all pure Hodge-Tate objects
\begin{equation}
B_i/B_{i+1} \simeq \mathbb{Q}(n_i),~n_i \in \mathbb{Z}.
\end{equation}
Let $j$ be the integer such that $A \subset B_{j}$, but $A \not\subset B_{j-1}$, then there exists a nonzero surjective morphism $f_j$ from $A$ to $B_j/B_{j-1} \simeq \mathbb{Q}(n_j)$ which induces a short exact sequence in $\textbf{MHS}_{\mathbb{Q}}$
\begin{equation}
\begin{tikzcd}
0 \arrow[r] & \text{ker}\,f_{j} \arrow[r] & A \arrow[r,"f_j"] & \mathbb{Q}(n_j) \arrow[r]& 0
\end{tikzcd}
\end{equation}
To show $A$ is an object of $\textbf{MHT}_{\mathbb{Q}}$, we only need to show $\text{ker}\,f_{j}$ is an object of $\textbf{MHT}_{\mathbb{Q}}$, which is done by an easy induction on the dimension of $A$. The subquotient case follows easily from the above proof and the fact that $\textbf{MHT}_{\mathbb{Q}}$ is an abelian category.
\end{proof}

\begin{proposition}
$D_{\textbf{MHT}_{\mathbb{Q}}}^b$ is a triangulated subcategory of $D_{\textbf{MHS}_{\mathbb{Q}}}^b$.
\end{proposition}

\begin{proof}
Given a morphism $f: A^{\bullet} \rightarrow B^{\bullet} $ in $D_{\textbf{MHS}_{\mathbb{Q}}}^b$, its distinguished triangle is of the form \cite{BeilinsonMHC}
\begin{equation}
A^{\bullet} \rightarrow B^{\bullet} \rightarrow \text{Cone}\,f \rightarrow A^{\bullet}[1],
\end{equation}
which induces a long exact sequence in $\textbf{MHS}_{\mathbb{Q}}$
\begin{equation} \label{eq:longexactsequenceMHC}
\cdots \rightarrow \underline{H}^q (A^{\bullet}) \rightarrow \underline{H}^q (B^{\bullet}) \rightarrow \underline{H}^q (\text{Cone}\,f) \rightarrow \underline{H}^{q+1} (A^{\bullet}) \rightarrow \cdots.
\end{equation}
If both $A^{\bullet}$ and $B^{\bullet}$ are objects of $D_{\textbf{MHT}_{\mathbb{Q}}}^b$, then for every $q \in \mathbb{Z}$ we have
\begin{equation}
\underline{H}^q (B^{\bullet}) \in \textbf{MHT}_{\mathbb{Q}},~\underline{H}^{q+1} (A^{\bullet}) \in \textbf{MHT}_{\mathbb{Q}}.
\end{equation}
Lemma \ref{THsubquotient} immediately implies that $\underline{H}^q (\text{Cone}\,f)$ is an object of $\textbf{MHT}_{\mathbb{Q}}$, hence $\text{Cone}\,f$ is an object of $D_{\textbf{MHT}_{\mathbb{Q}}}^b$, thus this proposition follows immediately.

\end{proof}

For an object $A^{\bullet}$ of $D_{\textbf{MHT}_{\mathbb{Q}}}^b$, the truncated complex $\tau_{\leq i}(A^{\bullet})$ is also an object of $D_{\textbf{MHT}_{\mathbb{Q}}}^b$. From Lemma 3.5 of \cite{BeilinsonMHC}, the functor $\underline{H}^*$ in \ref{eq:HbarDTH} is the cohomological functor of a non-degenerated $t$-structure on $D_{\textbf{MHT}_{\mathbb{Q}}}^b$, which is essentially the restriction of the $t$-structure of $D_{\textbf{MHS}_{\mathbb{Q}}}^b$. The inclusion $\textbf{MHT}_{\mathbb{Q}} \hookrightarrow D_{\textbf{MHT}_{\mathbb{Q}}}^b$ is an equivalence between $\textbf{MHT}_{\mathbb{Q}}$ and the heart of this $t$-structure on $D_{\textbf{MHT}_{\mathbb{Q}}}^b$. In Section 3 of \cite{BeilinsonMHC}, Beilinson proves that for $A^{\bullet},B^{\bullet} \in \textbf{MHS}_{\mathbb{Q}}$
\begin{equation}
\text{Hom}^i_{D^b(\textbf{MHS}_{\mathbb{Q}})}(A^{\bullet},B^{\bullet})= \text{Hom}^i_{D_{\textbf{MHS}_{\mathbb{Q}}}^b}(A^{\bullet},B^{\bullet}),
\end{equation}
from which the natural functor $D^b(\textbf{MHS}_{\mathbb{Q}}) \rightarrow D_{\textbf{MHS}_{\mathbb{Q}}}^b $ is an equivalence of categories. Since $\textbf{MHT}_{\mathbb{Q}}$ is a full subcategory of $\textbf{MHS}_{\mathbb{Q}}$, for $A^{\bullet},B^{\bullet} \in \textbf{MHT}_{\mathbb{Q}}$ we immediately have
\begin{equation}
\text{Hom}^i_{D^b(\textbf{MHT}_{\mathbb{Q}})}(A^{\bullet},B^{\bullet})= \text{Hom}^i_{D_{\textbf{MHT}_{\mathbb{Q}}}^b}(A^{\bullet},B^{\bullet}).
\end{equation} 
Therefore, the proof in \cite{BeilinsonMHC} shows the following proposition.
\begin{proposition} \label{MTHequivalence}
The functor  $D^b(\textbf{MHT}_{\mathbb{Q}}) \rightarrow D_{\textbf{MHT}_{\mathbb{Q}}}^b $ is an equivalence of categories.
\end{proposition}
This result is certainly well known, but we can not find a precise reference, therefore we provide a routine proof here. This equivalence is essentially induced by the equivalence $D^b(\textbf{MHS}_{\mathbb{Q}}) \hookrightarrow D^b_{\textbf{MHS}_{\mathbb{Q}}}$.

\end{document}